\documentclass[twoside,english]{amsart}
\usepackage{mathptmx}
\usepackage[T1]{fontenc}
\usepackage[latin9]{inputenc}
\usepackage[a4paper]{geometry}
\usepackage{verbatim}
\usepackage{refstyle}
\usepackage{units}
\usepackage{amsthm}
\usepackage{amsmath}
\usepackage{amssymb}

\makeatletter

\AtBeginDocument{\providecommand\secref[1]{\ref{sec:#1}}}
\AtBeginDocument{\providecommand\defref[1]{\ref{def:#1}}}
\AtBeginDocument{\providecommand\propref[1]{\ref{prop:#1}}}
\AtBeginDocument{\providecommand\lemref[1]{\ref{lem:#1}}}
\AtBeginDocument{\providecommand\thmref[1]{\ref{thm:#1}}}
\AtBeginDocument{\providecommand\exaref[1]{\ref{exa:#1}}}
\AtBeginDocument{\providecommand\remref[1]{\ref{rem:#1}}}
\RS@ifundefined{subref}
  {\def\RSsubtxt{section~}\newref{sub}{name = \RSsubtxt}}
  {}
\RS@ifundefined{thmref}
  {\def\RSthmtxt{theorem~}\newref{thm}{name = \RSthmtxt}}
  {}
\RS@ifundefined{lemref}
  {\def\RSlemtxt{lemma~}\newref{lem}{name = \RSlemtxt}}
  {}

\usepackage{latexsym, amsfonts, amssymb, pifont, amsmath}
\theoremstyle{plain}
\numberwithin{equation}{section}
\numberwithin{figure}{section}
\numberwithin{table}{section}

  \def\AS{\\ }
  \def\SU{\substack}
  \def\OV{\overset}

\usepackage{refstyle}
    \newtheorem{stthm}{\protect\theoremname}[section]
    
    \newtheorem{spprop}{\protect\propositionname}[section]

 \def\Compatamsartbabel#1#2{%
 \def\rightmark{#1}%
 \def\leftmark{#2}%
 }
 \usepackage{mathrsfs}
 \def\MR{\mathscr}
 \newref{sec}{%
 name = \RSsectxt,
 names = \RSsecstxt,
 lsttxt = \RSlsttxt,
 lsttwotxt = \RSlsttwotxt,
 refcmd = \ref{#1}}
 \newref{sec}{%
 name = \RSsectxt,
 names = \RSsecstxt,
 lsttxt = \RSlsttxt,
 lsttwotxt = \RSlsttwotxt,
 refcmd = \ref{#1}}
 \theoremstyle{definition}
 \newtheorem{sddefn}{\protect\definitionname}[section]
  \newref{def}{%
  name = \definitionname~,
  lsttxt = \RSlsttxt,
  lsttwotxt = \RSlsttwotxt,
  refcmd = {\ref{#1}}}
 \usepackage{enumitem}
 \theoremstyle{remark}
 \newtheorem*{rem*}{\protect\remarkname}
 \newref{rem}{%
 name = \remarkname~,
 lsttxt = \RSlsttxt,
 lsttwotxt = \RSlsttwotxt,
 refcmd = {\ref{#1}}}
 \theoremstyle{definition}
  \newtheorem{sexmexample}{\protect\examplename}[section]
  \newref{exa}{%
  name = \examplename~,
  lsttxt = \RSlsttxt,
  lsttwotxt = \RSlsttwotxt,
  refcmd = {\ref{#1}}}
 \newref{prop}{%
 name = \propositionname~,
 lsttxt = \RSlsttxt,
 lsttwotxt = \RSlsttwotxt,
 refcmd = {\ref{#1}}}
 \theoremstyle{plain}
 \newtheorem{spllem}{\protect\lemmaname}[spprop]
 \newref{lem}{%
 name = \lemmaname~,
 lsttxt = \RSlsttxt,
 lsttwotxt = \RSlsttwotxt,
 refcmd = {\ref{#1}}}
 \newref{sec}{%
 name = \RSsectxt,
 names = \RSsecstxt,
 lsttxt = \RSlsttxt,
 lsttwotxt = \RSlsttwotxt,
 refcmd = \ref{#1}}
 \newref{thm}{%
 name = \theoremname~,
 lsttxt = \RSlsttxt,
 lsttwotxt = \RSlsttwotxt,
 refcmd = {\ref{#1}}}
 \newenvironment{lyxlist}[1]
   {\begin{list}{}
     {\settowidth{\labelwidth}{#1}
      \setlength{\leftmargin}{\labelwidth}
      \addtolength{\leftmargin}{\labelsep}
      }}
   {\end{list}}
 \theoremstyle{plain}
 \newcounter{LonSP}[spprop]
 \newtheorem{llemsp}[LonSP]{\protect\lemmaname}
 \newref{lem}{%
 name = \lemmaname~,
 lsttxt = \RSlsttxt,
 lsttwotxt = \RSlsttwotxt,
 refcmd = {\ref{#1}}}
 \theoremstyle{remark}
 \newtheorem*{claim*}{\protect\claimname}
 \newref{cla}{%
 name = \claimname~,
 lsttxt = \RSlsttxt,
 lsttwotxt = \RSlsttwotxt,
 refcmd = {\ref{#1}}}
 \theoremstyle{plain}
 \newtheorem*{cor*}{\protect\corollaryname}
 \newref{cor}{%
 name = \corollaryname~,
 lsttxt = \RSlsttxt,
 lsttwotxt = \RSlsttwotxt,
 refcmd = {\ref{#1}}}
 \theoremstyle{remark}
 \newtheorem{strem}[stthm]{\protect\remarkname}
 \newref{rem}{%
 name = \remarkname~,
 lsttxt = \RSlsttxt,
 lsttwotxt = \RSlsttwotxt,
 refcmd = {\ref{#1}}}
 \theoremstyle{plain}
 \newtheorem{sllem}{\protect\lemmaname}[section]
 \newref{lem}{%
 name = \lemmaname~,
 lsttxt = \RSlsttxt,
 lsttwotxt = \RSlsttwotxt,
 refcmd = {\ref{#1}}}
 \usepackage{enumitem}
 \theoremstyle{plain}
 \newtheorem*{prop*}{\protect\propositionname}
 \newref{prop}{%
 name = \propositionname~,
 lsttxt = \RSlsttxt,
 lsttwotxt = \RSlsttwotxt,
 lsttxt = \RSlsttxt,
 lsttwotxt = \RSlsttwotxt,
 refcmd = {\ref{#1}}}
 \theoremstyle{remark}
 \newtheorem{srrem}{\protect\remarkname}[section]
 \newref{rem}{%
 name = \remarkname~,
 lsttxt = \RSlsttxt,
 lsttwotxt = \RSlsttwotxt,
 refcmd = {\ref{#1}}}
 \theoremstyle{plain}
 \newtheorem*{lem*}{\protect\lemmaname}
 \newref{lem}{%
 name = \lemmaname~,
 lsttxt = \RSlsttxt,
 lsttwotxt = \RSlsttwotxt,
 refcmd = {\ref{#1}}}
 \theoremstyle{remark}
 \newtheorem*{note*}{\protect\notename}
 \newref{note}{%
 name = \notename~,
 lsttxt = \RSlsttxt,
 lsttwotxt = \RSlsttwotxt,
 refcmd = {\ref{#1}}}

\font\Special=cmb10 at 0.8ex
\font\SpecialB=cmb10 at 0.65ex
\def\LB#1{\OV{{\mbox{\Special (} }{\raisebox{-.13ex}{$\overline{\, \, \,}$}} {\mbox{\Special )}}}{L_{#1}}}
\def\LTB#1{\OV{{\mbox{\Special (} }{\raisebox{-.13ex}{$\overline{\, \, \,}$}} {\mbox{\Special )}}}{\widetilde{L_{#1}}}}
\def\PB#1{\OV{{\mbox{\SpecialB (} }{\raisebox{-.16ex}{$\overline{\,}$}} {\mbox{\SpecialB )}}}{#1}}

\def\LBB#1#2{\OV{{\mbox{\Special (} }{\raisebox{-.13ex}{$\overline{\, \, \,}$}} {\mbox{\Special )}}}{L_{#1}^{#2}}}
\def\PBG#1{\OV{{\mbox{\Special (} }{\raisebox{-.16ex}{$\overline{\, \, \,}$}} {\mbox{\Special )}}}{#1}}
\font\Starg=cmr10 at 4ex
\def\STR{\raisebox{-1.5ex}{\mbox{\Starg *}}}
\usepackage{type1cm}
\usepackage[naturalnames,bookmarksopen,colorlinks=true,linkcolor=red,citecolor=blue]{hyperref}

\newcommand\mynobreakpar{\par\nobreak\@afterheading}

\makeatother

\usepackage{babel}
 \providecommand{\claimname}{Claim}
 \providecommand{\corollaryname}{Corollary}
 \providecommand{\definitionname}{Definition}
 \providecommand{\examplename}{Example}
 \providecommand{\lemmaname}{Lemma}
 \providecommand{\notename}{Note}
 \providecommand{\propositionname}{Proposition}
 \providecommand{\remarkname}{Remark}
 \providecommand{\theoremname}{Theorem}

\begin{document}
\Compatamsartbabel{EXTREMAL BASIS, GEOMETRICALLY SEPARATED DOMAINS}{PHILIPPE CHARPENTIER \& YVES DUPAIN}

\title{Extremal Bases, Geometrically Separated Domains and Applications}

\author{Philippe Charpentier \& Yves Dupain}
\begin{abstract}
In this paper we introduce the notion of extremal basis of tangent
vector fields at a boundary point of finite type of a pseudo-convex
domain in $\mathbb{C}^{n}$, $n\geq3$. Using this notion we define
the class of geometrically separated domains at a boundary point and
we give a description of their complex geometry. Examples of such
domains are given, for instance, by locally lineally convex domains,
domains with locally diagonalizable Levi form at a point or domains
for which the Levi form have comparable eigenvalues near a point.
Moreover we show that geometrically separated domains can be localized.
We also give an example of a non geometrically separated domains.
Next we define what we call {}``adapted pluri-subharmonic function''
and give sufficient conditions, related to extremal bases, for their
existence. Then, for these domains, when such functions exist, we
prove global and local sharp estimate for the Bergman and Szeg\"o projections.
As an application, we strengthen a result by C. Fefferman, J. J. Kohn
and M. Machedon (\cite{F-K-M-d-bar-Diag-Levi-Form}) for the local
H\"older estimate of the Szeg\"o projection removing the arbitrary small
loss in the H\"older index and giving a stronger non-isotropic estimate.
\end{abstract}



\address{Philippe Charpentier \& Yves Dupain: Universit\'e Bordeaux 1, Institut
de Math\'ematiques, 351, cours de la Lib\'eration, 33405 Talence, France}

\email{philippe.charpentier@math.u-bordeaux1.fr, yves.dupain@math.u-bordeaux1.fr}

\maketitle

\section{Introduction}

\global\long\def\brond{\MR{B}}
\global\long\def\rhotilde{\tilde{\rho}}
\global\long\def\Lrond{\MR{L}}
\global\long\def\Lbar#1{\overline{L_{#1}}}
\global\long\def\Btilde{\widetilde{\MR{B}}}
\global\long\def\Ltildep#1{\widetilde{L_{#1}}}
\global\long\def\Ftilder{\widetilde{F^{\rho}}}
\global\long\def\Hrond{\MR{H}}
\global\long\def\Ltilde{\widetilde{L}}
\global\long\def\Lbarp#1#2{\overline{L_{#1}^{#2}}}
\global\long\def\Lrondt{\widetilde{\MR{L}}}
\global\long\def\Crond{\MR{C}}

The study of the regularity with sharp estimates for the Bergman and
Szeg\"o projections for pseudo-convex domains in $\mathbb{C}^{n}$ became
very active for domains of finite type when D. Catlin proved his fundamental
characterization of subelliptic estimates (\cite{Catlin-Est.-Sous-ellipt.}).

Quite quickly, the case of domains in $\mathbb{C}^{2}$ was completely
solved by D. Catlin in \cite{Catlin-Bergman-dim-2}, A. Nagel, J.-P.
Rosay, E. M. Stein and S. Wainger in \cite{N-R-S-W-Bergman-dim-2},
M. Christ in \cite{Christ-88-CR-dimension-3}, C. Fefferman
and J. J. Kohn in \cite{Fefferman-Kohn-88-dimension-2} and J. McNeal
in \cite{McNeal-Bergman-C2-hors-diag}.

In higher dimensions, the situation is more complicated and, until
now, there are only partial results. One of the main difficulties
is the description of the geometry of the domain: there are some special
bases of the complex tangent space at the boundary playing an important
role in this description and also in the Lipschitz estimates of the
projectors. Thus the first results concern domains for which these
bases are more or less evident. For example, the class of domains
for which the Levi form have rank larger than $n-2$ was studied by
M. Machedon in \cite{Machedon-88-one-degenerate} (see also S. Cho
\cite{Cho-Bergman-94,Cho-Bergman-96}, \cite{Ahn-Cho-99-Berg-pro-rank-n-1})
and, even in that case, the situation is not so simple. An other example
is given by decoupled domains, treated by several authors (see for
example \cite{McNeal-Aspekte}, \cite{Chang-Grellier}).

A typical example where the choice of special bases is essential,
and not evident, is the case of convex domains in $\mathbb{C}^{n}$.
In \cite{McNeal-convexes-94,McNeal-unif-subel-est-convex} J. Mc Neal
introduced some special bases (called $\varepsilon$-extremal in
\cite{Bruna-Charp-Dupain-Annals})
and gave a description of the complex geometry with the construction
of a pseudo-distance near the boundary related to these bases. With
that geometry, and a construction of a ``good'' pluri-subharmonic
function, he proved sharp point-wise estimates for the Bergman kernel
and its derivatives. Using this geometry J. Mc Neal and E. M. Stein
(\cite{McNeal-Stein-Bergman} and \cite{McNeal-Stein-Szego})
proved sharp estimates for the Bergman and Szeg\"o projections.

More recently similar results were obtained, when the Levi form has
comparable eigenvalues, by K. Koenig in \cite{Koenig-Cauchy-Riemann-eigenvalues-Levi}
and S. Cho in \cite{Cho-03-Bergman-comparable-Math-Anal-Appl}, \cite{Cho-02-Bergman-comparable-Korean}.

In \cite{F-K-M-d-bar-Diag-Levi-Form} C. L. Fefferman, J. J. Kohn
and M. Machedon studied the case where the Levi form is locally diagonalizable
near a point $p_{0}$ of the boundary. They solved the $\bar{\partial}_{b}$-Neuman
problem and deduced that if $f$ is a $L^{2}(\partial\Omega)$ function
which is locally in the classical Lipschitz space $\Lambda_{\alpha}$
(near $p_{0}$) then, for all $\varepsilon>0$ it's Szeg\"o projection
$Sf$ is locally (near $p_{0}$) in $\Lambda_{\alpha-\varepsilon}$
(an application of our theory will remove the loss of $\varepsilon$
in this estimate and get, in fact, a better non-isotropic estimate).\bigskip{}

The main idea of the present paper is to introduce a general notion
of ``extremal basis'' of the complex tangent space at a boundary
point of a pseudo-convex domain in $\mathbb{C}^{n}$, $n\geq3$, generalizing
the $\varepsilon$-extremal bases of the convex case. With this notion
we define a class of pseudo-convex domains, containing all previously
studied classes, called ``geometrically separated'', for which
a good family of extremal bases exists near a point of the boundary.
The fundamental properties of an extremal basis allow one to prove that,
for these domains, there exists an associated structure of homogeneous
space on the boundary (and an extension of that structure inside the
domain) which describes the complex geometry of the domain. An important
property of domains which are geometrically separated at a boundary
point is that this structure can be nicely localized (see the end
of \secref{Notations-and-organization} for more details).

Moreover, when special pluri-subharmonic functions (called ``adapted
pluri-subharmonic functions'' in this paper) exist, this structure
is used to obtain sharp global and local estimates for the Bergman kernel,
the Bergman and Szeg\"o projection and the classical invariant
metrics. The existence of such adapted pluri-subharmonic functions
for geometrically separated domains is not evident in general. For
example, if the domain is locally convex (or more generally lineally convex),
this is done using special support functions (see
\cite{Die-For-support-function-convex,McNeal-unif-subel-est-convex})
which cannot exist, in general, without convexity. Here we prove their
existence, under an additional condition (which is satisfied, for
example, when the Levi form is locally diagonalizable) on the extremal
bases, for the domain and also for the localized one (see the end
of \secref{Notations-and-organization} for more details).

\section{Notations and organization of the paper\label{sec:Notations-and-organization}}

In all the paper, $\Omega=\{\rho<0\}$ denotes a bounded domain in
$\mathbb{C}^{n}$, $n\geq3$, with a $\MR C^{\infty}$ boundary, and
$\rho\in\MR C^{\infty}(\mathbb{C}^{n})$ is a defining function of
$\Omega$ such that $\left|\nabla\rho\right|=1$ on $\partial\Omega$.
We denote by $N=\frac{1}{\left|\nabla\rho\right|}\sum\frac{\partial\rho}{\partial\bar{z}_{i}}\frac{\partial}{\partial z_{i}}$
the unitary complex normal vector field to $\rho$ (i.e. $N\rho\equiv1$
and $\left\Vert N\right\Vert \equiv1$).

For each point $p$ of the boundary let us denote $T_{p}^{1,0}(\partial\Omega)$
the subbundle of $T_{p}(\partial\Omega)$ of tangential complex vectors
and $T_{p}^{0,1}(\partial\Omega)$ its conjugate. As usual, we will
say that a family $(L_{i})_{1\leq i\leq n-1}$ of $\MR{C}^{\infty}$
vector fields is a basis of the complex tangent space at $\partial\Omega$
in a open neighborhood $V\subset\partial\Omega$ of a point $p_{0}$
in $\partial\Omega$ if it is a basis of sections of $T^{1,0}(\partial\Omega)$
in $V$ (i.e. $L_{i}(\rho)\equiv0$ in $V$, a condition which is
independent of the defining function).

Clearly, every $\MR{C}^{\infty}$ vector field $L$ in an open neighborhood
$V\subset\partial\Omega$ can be extended to an open neighborhood $V(p_{0})\subset\mathbb{C}^{n}$
so that $L(\rho)\equiv0$ in $V(p_{0})$. Of course this extension
depends on the defining function $\rho$, but all the stated results
will be independent of such a choice. Thus, in all the paper,
the tangent vector fields considered in $V(p_{0})$ are always supposed
to annihilate $\rho$ in $V(p_{0})$, and we will use the terminology
of ``vector fields tangent to $\rho$'' for this property.

\medskip{}

Let $L$ and $L'$ be two $(1,0)$ vector fields tangent to $\rho$.
The bracket $[L,\overline{L'}]$ being tangent to $\rho$, can
be written
\[
[L,\overline{L'}]=2\sqrt{-1}c_{LL'}T+L''
\]
 where $T$ is the imaginary part of $N$ and $L''\in T_{p}^{1,0}(\partial\Omega)\oplus T_{p}^{0,1}(\partial\Omega)$.
Thus $c_{LL'}=[L,\overline{L'}](\partial\rho)=\left\langle \partial\rho;[L,\overline{L'}]\right\rangle $.
The Levi form of $\partial\Omega$ at $p$ is defined as the hermitian
form whose value at $(L,\overline{L'})$ is the number $c_{LL'}$.
The pseudo-convexity of $\Omega$ means that this hermitian form is
non-negative. If $(L_{i})_{1\leq i\leq n-1}$ is a local basis of
$(1,0)$ vector fields tangent to $\rho$, then $(c_{L_{i}L_{j}})_{i,j}$
is the matrix of the Levi form in the given basis. This matrix
will be generally denoted $\left(c_{ij}\right)_{i,j}$.

\medskip{}

Let $p_{0}\in\Omega$ and $V(p_{0})$ be a neighborhood of $p_{0}$
in $\mathbb{C}^{n}$. If $W$ is a set of $\MR{C}^{\infty}\left(V\left(p_{0}\right)\right)$
$(1,0)$ complex vector fields, then $\Lrond(W)$ denotes the set of all
lists $\Lrond=\left(L^{1},\ldots,L^{k}\right)$ such that $L^{j}\in W\cup\overline{W}$,
and, for $l\in\mathbb{N}$, $\Lrond_{l}(W)$ denotes the set of such
lists $\Lrond$ of length $\left|\Lrond\right|=k\in\left\{ 0,1,\ldots,l\right\} $.
If $W$ contains only one vector field $L$, then we will write $\Lrond(L)$ and $\Lrond_l(L)$
instead of $\Lrond(\{L\})$ and $\Lrond_l(\{L\})$.
Moreover, if $\left|\Lrond\right|=k\geq2$, we denote 
\[
\Lrond(\partial\rho)=L^{1}\ldots L^{k-2}\left(\left\langle \partial\rho,\left[L^{k-1},L^{k}\right]\right\rangle \right).
\]
Note that if $L^{k-1}$ and $L^k$ are both $(1,0)$ or both $(0,1)$ then
$\left\langle \partial\rho,\left[L^{k-1},L^{k}\right]\right\rangle$
is identically zero. Thus if $\Lrond(\partial\rho)$ is not identically zero, it is equal to
$\pm$ a derivative of the value taken by the Levi form on $(L^{k-1},L^k)$ or $(L^k,L^{k-1})$.

Let $L$ be a $\MR C^{\infty}(V(p_{0}))$ $(1,0)$ complex vector
field tangent to $\rho$ and $M\geq2$ be an integer. We define the
weight $F_{M}^{\Omega}(L,p,\delta)=F^{\Omega}(L,p,\delta)=F^{\Omega}(L)$
associated to $L$ at the point $p\in V(p_{0})$ and to $\delta>0$
by
\[
F^{\Omega}(L,p,\delta)=\sum_{\Lrond\in\Lrond_{M}(L)}\left|\frac{\Lrond(\partial\rho)(p)}{\delta}\right|^{2/\left|\Lrond\right|}=2\sum_{\widetilde\Lrond\in\Lrond_{M-2}(L)}\left|\frac{\widetilde\Lrond\left(c_{LL}\right)}{\delta}\right|^{2/\left|\widetilde\Lrond\right|+2}.
\]
where $\widetilde\Lrond\left(c_{LL}\right)=L^1\ldots L^k\left(c_{LL}\right)$ if
$\widetilde\Lrond=\left(L^1,\ldots, L^k\right)$.
Moreover, for the complex normal direction $N$ we define $L_{n}=N$
and $F^{\Omega}(L_n,p,\delta)=F^{\Omega}(N,p,\delta)=\delta^{-2}$. When there is no ambiguity
(typically when there is only one domain) we will omit the superscript
$\Omega$.

Note that, with the conditions on $\rho$, the functions $\Lrond(\partial\rho)$
restricted to $\partial\Omega$ do not depend on the choice of the
defining function $\rho$. By the finite type hypothesis, for $\delta$ small,
the weights will be large. Thus if we consider them in $\delta$-strips near the boundary,
they are intrinsically attached to the boundary of the domain and do not
depend on the choice of the defining function $\rho$.

\medskip

In all the paper the defining function $\rho$ of $\Omega$ is fixed
and the number $M$ also. When we say that some number depends on ``$\vartheta$''
and on ``the data'', we mean that it depends on ``$\vartheta$'',
$n$, $M$, and $\rho$ but neither on the point $p$ in $V(p_{0})$
nor on $\delta\leq\delta_{0}$.

If $\brond=\{L_{1},\ldots,L_{n-1}\}$ is a $\MR C^{\infty}$ basis
of $(1.0)$ vector fields tangent to $\rho$ in $V(p_{0})$, and $\Lrond\in\Lrond(\brond\cup\{N\})$,
we denote 
\[
F(p,\delta)^{\Lrond/2}=\prod_{i=1}^{n}F(L_{i},p,\delta)^{l_{i}/2},
\]
 where $l_{i}=l_{i}(\Lrond)$ is the number of times $L_{i}$ or $\Lbar i$
appears in $\Lrond$, $i\leq n-1$, and $l_{n}=l_{n}(\Lrond)$ the
number of times $N$ or $\overline{N}$ appears in $\Lrond$ (and
thus $\left|\Lrond\right|=k=\sum_{i=1}^{n}l_{i}$).\bigskip{}

The organization of the paper is as follows:

In \secref{extremal-basis} we define the notion of extremal basis
and give some examples. Then we give their basic properties and, in
\secref{extremal-basis-coordinate-system}, we prove a
fundamental property of an extremal basis at a point of finite type:
there exists a coordinate system which is adapted
to the basis in the sense that all the derivatives of the matrix
of the Levi form (in that basis) are controlled by the weights attached
to the basis. We give also some sufficient conditions of
extremality for a given basis, useful for some examples. Finally,
in \secref{localization-extremal-basis} we show how the existence
of extremal bases can be localized in the sense that, near a boundary
point $p_{0}$ of $\Omega$ of finite type, if there exist extremal
bases at every boundary points near $p_{0}$, then one can construct
a small pseudo-convex domain $D$ of finite type inside the original
domain, containing a piece of the boundary of $\Omega$ in its boundary
such that there exist extremal bases a every point of the boundary
of $D$.

In \secref{Geometrically-separated-domains} we define the notion
of geometrically separated domains at a point $p_{0}$ of its boundary
and give examples. Then we show that a geometrically separated domain
is automatically equipped with a local structure of homogeneous space
on its boundary. In \secref{Localization-geom-sep-structure} we prove
that the structure of geometrically separated domain can always be
localized (in the sense described above).

In \secref{Adapted-pluri-subharmonic-function-geom-sep} we study
the existence of pluri-subharmonic functions adapted to a given
geometrically separated domain. In particular, we prove their existence
when the domain is {}``strongly'' geometrically separated at a point
$p_{0}$ of its boundary, and we prove that, in this case, such functions
exist for the localized domain at every point of its boundary.

In the last Section (\ref{sec:Applications-to-complex-analysis})
we show that all the sharp global and local results for Bergman kernel,
Bergman and Szeg\"o projections and invariant metrics can be established
for geometrically separated domains when there exist adapted pluri-subharmonic
functions. The local sharp estimate of the Szeg\"o projection when the
Levi form is locally diagonalizable is an example of these results.

\section{Extremal bases\label{sec:extremal-basis}}

\subsection{Definition and examples}
\begin{sddefn}
\label{def:basis-extremal}Let $\Omega$ and $V(p_{0})$ defined on
\secref{Notations-and-organization}. Let $\brond=\{L_{1},\ldots,L_{n-1}\}$
be a $\MR C^{\infty}$ basis of $(1,0)$ vector fields tangent to
$\rho$ in $V(p_{0})$ and $M$ an integer. Let $p\in V(p_{0})$ and
$0<\delta$. We say that $\brond=\{L_{1},\ldots,L_{n-1}\}$ is $(M,K,p,\delta)$-extremal
(or simply $(K,p,\delta)$-extremal or $K$-extremal) if the $\Crond^{2M}$
norms, in $V(p_{0})$, of all $L_{i}$ are bounded by $K$, the Jacobian
of $\brond$ is bounded from below by $1/K$ on $V(p_{0})$, and the
two following conditions are satisfied:
\begin{enumerate}
\item [EB$_{\text{1}}$]For any vector field $L$ of the form $L=\sum_{i=1}^{n-1}a_{i}L_{i}$,
$a_{i}\in\mathbb{C}$, we have
\[
\frac{1}{K}\sum_{i=1}^{n-1}\left|a_{i}\right|^{2}F(L_{i},p,\delta)\leq F(L,p,\delta)\leq K\sum_{i=1}^{n-1}\left|a_{i}\right|^{2}F(L_{i},p,\delta).
\]

\item [EB$_{\text{2}}$]For all indexes $i,j,k$ such that $i,j<n$, $k\leq n$
and all lists $\Lrond$ of $\Lrond_{M}\left(\brond\cup\{N\}\right)$,
\[
F(L_{k},p,\delta)^{1/2}\left|\Lrond a_{\PB{i},\PB{j}}^{\PB{k}}(p)\right|\leq KF(p,\delta)^{\Lrond/2}F(L_{i},p,\delta)^{1/2}F(L_{j},p,\delta)^{1/2},
\]
 where $a_{\PB{i},\PB{j}}^{\PB{k}}$ is the coefficient of the bracket
$\left[\LB{i},\LB{j}\right]$ in the direction $\LB{k}$ (with $L_{n}=N$),
and $\LB{i}$ means $L_{i}$ or $\overline{L_{i}}$.
\end{enumerate}
\end{sddefn}

\begin{rem*}
In general this Definition depends of the choice of the defining function
$\rho$. But note that, for $p\in\partial\Omega$, it does not: it
depends only on the restriction of $\brond$ to $\partial\Omega\cap V(p_{0})$.\end{rem*}
\begin{sexmexample}
\label{exa:Example-extremal-basis}~
\begin{enumerate}
\item \emph{Locally lineally convex domains.} A first example of extremal
basis concerns the case of a locally convex domain near a point of
finite type: it can be easily shown, using the work of Mc Neal \cite{McNeal-convexes-94} (see also \cite{Hef04}),
that if $\Omega$ is convex near a point of finite type $p_{0}\in\partial\Omega$,
if the canonical coordinate system is chosen so that the last coordinate
is the complex normal at $p_{0}$, and, if $P$ is the projection
onto the complex tangent space of the defining function of $\Omega$
parallel to the last coordinate, then for each point $p$ in a small
neighborhood of $p_{0}$, and each $\delta\leq\delta_{0}$, the $P$-projection
of the first $n-1$ vectors of the Mc Neal $\delta$-extremal basis
at $p$ (c.f. \cite{Bruna-Charp-Dupain-Annals,McNeal-convexes-94})
is $(K,p,\delta)$-extremal in our sense for a constant $K$ depending
only on the data.\\
More generally, the same thing can be done for \emph{locally lineally
convex domains} using the work of Conrad, M. \cite{Conrad_lineally_convex} (recall that
$\Omega$ is said lineally convex at a point $p\in\partial\Omega$ if there is a neighborhood
$U$ of $p$ such that the intersection of the complex tangent space to $\partial\Omega$ at
$p$ with $\Omega\cap U$ is empty; see
\cite{Kiselman-lineally-convex,Diederich-Fornaess-Support-Func-lineally-cvx} for the
precise definition and a useful characterization). Some details are given in
\secref{Examples-The-lineally-convex-case}.
\item \emph{Levi form with comparable eigenvalues.} A second example is
given by a pseudo-convex domain having a point of finite type $p_{0}\in\partial\Omega$
where the eigenvalues of the Levi form are comparable (see \cite{Koenig-Cauchy-Riemann-eigenvalues-Levi,Cho-02-Bergman-comparable-Korean,Cho-03-Bergman-comparable-Math-Anal-Appl,Cho-02-metrics-comparable}).
Indeed, in \cite{Cho-03-Bergman-comparable-Math-Anal-Appl} it is
proved that any (normalized) basis of the complex tangent space is
$K$-extremal for a well controlled constant $K$.
\item \emph{Locally diagonalizable Levi form.} In \secref{sufficient-cond-of-extremality}
we will show that if at a point of finite type $p_{0}\in\partial\Omega$
the Levi form is locally diagonalizable then the basis diagonalizing
the Levi form is $K$-extremal for a constant $K$ depending only
on the data (in fact, this basis is $K$-strongly-extremal (see \defref{basis-strongly-extremal})
for every constant $\alpha>0$ with $\delta\leq\delta_{0}$, $\delta_{0}$
small depending on $\alpha$).
\item \emph{Localization.} Another important example will be given in \secref{localization-extremal-basis}:
for any $\tau>0$ there exists $M(\tau)$ such that if a family of
$(M(\tau),K,p,\delta)$-extremal bases exists in a neighborhood of
a boundary point $p_{0}$, of finite type $\tau$, of $\Omega$, then
one can construct a small smooth pseudo-convex domain $D$ contained in $\Omega$ and containing
a neighborhood of $p_{0}$ in $\partial\Omega$ in its boundary and
for which there exists an $(M(\tau),K',q,\delta)$-extremal basis at
every points $q\in\partial D$.
\end{enumerate}
\end{sexmexample}

\subsection{Basic properties of extremal bases}

The first property states that an extremal basis at $p$ can be orthogonalized
at the point $p$:
\begin{spprop}
\label{prop:extremality-and-normalization}For any $K$ there exists
a constant $K'$ depending only on $K$ and the data such that, if
$\brond$ is a basis of complex $(1,0)$ vector fields tangent to
$\rho$ in an open set $V(p_{0})$ which is $(K,p,\delta)$-extremal, then
there exists a basis $\brond'$, orthonormal at $p$ which is $(K',p,\delta)$-extremal.\end{spprop}
\begin{proof}
We can suppose that the vector fields $L_{i}$ of $\brond$ are ordered
such that $F(L_{i+1},p,\delta)\leq F(L_{i},p,\delta)$, for $i<n-1$.
Then, using the Graam-Schmidt process, we first define a basis $\brond_{1}$
by decreasing induction, $L_{i}^{1}=\sum_{j=i}^{n-1}\alpha_{i}^{j}L_{j}$,
$\alpha_{i}^{j}\in\mathbb{C}$, and $\sum\left|a_{i}^{j}\right|^{2}=1$.
The determinant condition implies that there exists $c>0$ such that
$\left|\alpha_{i}^{i}\right|>c$. Then 
\[
F(L_{i}^{1},p,\delta)\simeq_{K}\sum_{j\geq i}\left|\alpha_{i}^{j}\right|^{2}F(L_{j},p,\delta)\simeq_{K}F(L_{i},p,\delta).
\]
Now, let $L=\sum_{i}a_{i}L_{i}^{1}$ be a linear combination, with
constant coefficients, of the $L_{i}^{1}$. Then
\[
F(L,p,\delta)\simeq_{K}\sum_{k}\left|\sum_{i\leq k}a_{i}a_{i}^{k}\right|^{2}F(L_{k},p,\delta)\simeq_{K}\sum\left|a_{k}\right|^{2}F(L_{k},p,\delta),
\]
 using that $\left|\sum_{i\leq k}a_{i}a_{i}^{k}\right|\geq c\left|a_{k}\right|-\sum_{i<k}\left|a_{i}\right|$
and the fact that the $F(L_{k},p,\delta)$ are decreasing. This proves
EB$_{\text{1}}$ for $\brond_{1}$.

Note now that the decreasing property shows that property EB$_{\text{2}}$ for
$\brond$ trivially implies the same property for $\brond_{1}$ because $L_{i}^{1}$ involves
only fields $L_{j}^{1}$ for $j\geq i$.

Finally, define $\brond'$ by $L'_{i}=L_{i}^{1}/\left\Vert L_{i}^{1}\right\Vert $.
The condition on the $\MR{C}^{2M}$ norm of the vectors $L_{i}$ immediately
implies the result.
\end{proof}

Let us now prove that the mixed derivatives of the Levi form in the
directions of an extremal basis are controlled by the pure ones, that
is by the weights associated to the vector fields of the basis:
\begin{spprop}
\label{prop:control-lists}Let $\brond=\{L_{i},\,1\leq i\leq n-1\}$
be a $\Crond^{\infty}$ basis of complex $(1,0)$ vector fields tangent
to $\rho$ in $V(p_{0})$ which is $(K,p,\delta)$-extremal for a
fixed $\delta>0$. Let $\Lrond$ be a list of vector fields belonging
to $\Lrond_{M}(\brond\cup\{N\})$. Then there exits a constant $C>0$
depending only on $\Omega$ and $K$ such that $\left|\Lrond(\partial\rho)(p)\right|\leq C\delta F^{\MR{L}/2}(p,\delta)$.\end{spprop}
\begin{proof}
Recall the notation $c_{ij}=\left\langle \partial\rho,\left[L_{i},\Lbar j\,\right]\right\rangle $.
\begin{spllem}
\label{lem:lemma1-est-lists}With the previous notations (and the
definition of the coefficients $a_{ij}^{s}$ given in \defref{basis-extremal}):
\[
L_{j}c_{ik}=L_{i}c_{jk}+\sum a_{k\bar{i}}^{\bar{s}}c_{js}-\sum a_{ij}^{s}c_{sk}-\sum a_{j\bar{k}}^{\bar{s}}c_{is},
\]
\[
\Lbar jc_{ik}=\Lbar kc_{ij}+\sum a_{i\bar{k}}^{s}c_{sj}+\sum a_{\bar{j}i}^{s}c_{sk}-\sum a_{\bar{k}\bar{j}}^{\bar{s}}c_{is}.
\]
\end{spllem}
\begin{proof}
The first formula is simply obtained considering the coefficient of
$\Im\mathrm{m}N$ in the Jacobi's identity applied to the bracket $\left[L_{j},\left[L_{i},\Lbar k\,\right]\right]$,
and the second by the same way using $\left[\,\overline{L_{j}},\left[L_{i},\Lbar k\,\right]\right]$.
\end{proof}
The proof of \textit{\propref{control-lists}} is done by induction
on the length of the lists. Suppose first $\left|\Lrond\right|=2$.
Hypothesis EB$_{\text{1}}$ imply that, for all numbers $a$ and $b$
and all index $i$ and $j$,
\[
\left|\left|a\right|^{2}c_{ii}+\left|b\right|^{2}c_{jj}+a\bar{b}c_{ij}+\bar{a}bc_{ji}\right|\lesssim\delta\left(\left|a\right|^{2}F_{i}+\left|b\right|^{^{2}}F_{j}\right).
\]
Suppose both $F_{i}$ and $F_{j}$ are non zero. Taking $a=F_{j}^{1/2}F_{i}^{-1/2}\lambda$
and $b=\mu$, $\left|\lambda\right|$ and $\left|\mu\right|$ less
than $1$, the equivalence of norms in finite dimensional spaces gives
the result. If $F_{i}=0$ or $F_{j}=0$ a similar argument gives $c_{ij}=c_{ji}=0$.

To continue the proof, we need the following notation: if $\Lrond\in\Lrond(\brond\cup\{N\})$,
we denote by $l_{i}^{1}$ (resp. $l_{i}^{2}$) the number of times
$L_{i}$ (resp. $\overline{L_{i}}$) appears in $\Lrond$ (thus $l_{i}=l_{i}^{1}+l_{i}^{2}$).

For lists of greater length, we prove, at the same time, by induction,
the estimate and the following Lemma:
\begin{spllem}
\label{lem:lemma2-est-lists}Let $\Lrond$ and $\Lrond'$ be two lists
of $\Lrond_{M}(\brond\cup\{N\})$, $\Lrond(\partial\rho)=\Lrond_{1}c_{ij}$
and $\Lrond'(\partial\rho)=\Lrond'_{1}c_{kl}$, such that $l_{i}^{1}={l'}_{i}^{1}$,
$l_{i}^{2}={l'}_{i}^{2}$. Then $\Lrond(\partial\rho)\simeq\Lrond'(\partial\rho)$
in the sense that
\[
\Lrond(\partial\rho)-\Lrond'(\partial\rho)=\sum_{\left|\widetilde{\Lrond}\right|<\left|\Lrond\right|}a_{\widetilde{\Lrond}}\widetilde{\Lrond}(\partial\rho),
\]
 where $a_{\widetilde{\Lrond}}$ satisfies
$F^{\widetilde{\Lrond}/2}\left|\Lrond''a_{\widetilde{\Lrond}}\right|\lesssim\delta F^{\left(\Lrond+\Lrond''\right)/2}$,
$\forall\Lrond''\in\Lrond_{M}(\brond\cup\{N\})$,
the constant depending only on $K$ and the data.
\end{spllem}
Suppose thus the estimates and the Lemma proved for all lists of length
less than or equal to $N$.

First, we prove \lemref{lemma2-est-lists} for lists of length $N+1$.
Let us write $\Lrond(\partial\rho)=\Lrond_{1}c_{ij}$ and $\Lrond'(\partial\rho)=\Lrond_{2}c_{kl}$.
Then three cases can happen:
\begin{enumerate}
\item $(i,j)=(k,l)$;
\item $i\neq k$, $j\neq l$;
\item $i\neq k$ and $j=l$ or $i=k$ and $j\neq l$.
\end{enumerate}
The first case is a trivial consequence of EB$_{\text{2}}$. For the
second, the hypothesis on the length and case (1) imply that there
exists a list $\widetilde{\Lrond}$ such that $\Lrond_{1}c_{ij}\simeq\widetilde{\Lrond}L_{k}\overline{L_{l}}c_{ij}$
and $\Lrond_{2}c_{kl}\simeq\widetilde{\Lrond}L_{i}\overline{L_{j}}c_{kl}$,
in the sense of \lemref{lemma2-est-lists}. By \lemref{lemma1-est-lists}
and EB$_{\text{2}}$, $\overline{L_{l}}c_{ij}\simeq\overline{L_{j}}c_{il}$.
The result is obtained using another time EB$_{\text{2}}$, \lemref{lemma1-est-lists}
and the induction hypothesis. The third case is similar.

Now we prove the estimate of the Proposition for lists of length $N+1$.
Suppose that the vector fields are ordered so that there exists an
integer $n_{0}\in\{0,\ldots,n-1\}$ such that, for $k\leq n_{0}$,
$F_{k}\neq0$, and, for $n-1\geq k>n_{0}$, $F_{k}=0$. Let $L=\sum a_{j}L_{j}$,
$a_{j}=\varepsilon\lambda_{j}F_{n_{0}}^{1/2}F_{j}^{-1/2}$ if $j\leq n_{0}$
and $a_{j}=\lambda_{j}$ if $j>n_{0}$, with $\left|\lambda_{j}\right|\leq1$.
If we apply the extremality property to $F(L)$, then we obtain, for example,
for all $k\leq N-1$,
\[
\sup_{\left|\lambda_{j}\right|\leq1}\left|L^{k}\overline{L}^{N-k-1}c_{LL}\right|\lesssim\delta\varepsilon^{(N+1)/2}F_{n_{0}}^{(N+1)/2}
\]
 with the convention $F_{n_{0}}=0$ if $n_{0}=0$. Writing $L^{k}\overline{L}^{N-k-1}c_{LL}=\sum C_{\alpha\beta}\lambda^{\alpha}\bar{\lambda}^{\beta}$,
the equivalence of norms in finite dimensional spaces gives, when
$\varepsilon\rightarrow0$,
\begin{equation}
\left.\begin{array}{ll}
C_{\alpha\beta}=0 & \mbox{if there exits }j>n_{0}\mbox{ such that }\alpha_{j}+\beta_{j}>0,\\
\left|C_{\alpha\beta}\right|\lesssim\delta F^{(\alpha+\beta)/2} & \mbox{otherwise.}
\end{array}\right.\label{eq:est-listes}
\end{equation}

Let $\MR{E}_{\alpha\beta}$ be the set of lists $\Lrond$ such that
$l_{i}^{1}(\Lrond)=\alpha_{i}$ and $l_{i}^{2}=\beta_{i}$. Then $C_{\alpha\beta}=\sum_{\Lrond\in\MR{E}_{\alpha\beta}}\Lrond(\partial\rho)$.
Now, \lemref{lemma2-est-lists} and the induction hypothesis give
the required estimation for each list in $\MR{E}_{\alpha\beta}$ and
finishes the proof of the Proposition.
\end{proof}

The statement of the last Proposition is not really a statement on
the vector fields of an extremal basis but on the linear space generated
by an extremal basis. In fact the following Proposition is easily proved:
\begin{spprop}
\label{prop:control-lists-space-generated}In the conditions of \propref{control-lists},
there exists a constant $C$ such that, if ${L'}_{j}$, $1\leq j\leq k$
are vector fields belonging to the linear space generated by the extremal
basis $(L_{i})_{i}$ then for every $\Lrond\in\Lrond_{M}({L'}_{1},\ldots,{L'}_{k})$,
if ${L'}_{j}$ or $\overline{{L'}_{j}}$ appear $l'_{j}$ times in $\Lrond$,
$\left|\Lrond(\partial\rho)(p)\right|\lesssim\delta\prod_{j}F({L'}_{j},p,\delta)^{l'_{j}/2}$.
\end{spprop}

\subsection{Adapted coordinates system for points of finite $1$-type\label{sec:extremal-basis-coordinate-system}}

\subsubsection{Definition of an adapted coordinate system and statement of the main
result}

Let $p_{0}\in\partial\Omega$ and $V(p_{0})$ a neighborhood of $p_{0}$
in $\mathbb{C}^{n}$.
\begin{sddefn}
\label{def:basis-and-coordinates-adapted}A basis $\brond=(L_{1},\ldots,L_{n-1})$
of sections of $(1,0)$ complex tangent vector fields to $\rho$ in
$V(p_{0})$ and a coordinate system in $\mathbb{C}^{n}$, $z=\Phi_{p}^{\delta}(Z)$,
are called $(M,K,\delta)$-\emph{adapted} (or simply $(K,\delta)$-\emph{adapted)
at the point} $p$ in $V(p_{0})$ if $\Phi_{p}^{\delta}$ and $(\Phi_{p}^{\delta})^{-1}$
are polynomial (of degree less than $(2M)^{n-1}$) diffeomorphisms
of $\mathbb{C}^{n}$ centered at $p$ (i.e. $\Phi_{p}^{\delta}(p)=0$)
satisfying (with the notation $F_{i}=F_{i}(p,\delta)=F(L_{i},p,\delta)$):
\begin{enumerate}
\item The coefficients of the polynomials of $\Phi_{p}^{\delta}$ and $(\Phi_{p}^{\delta})^{-1}$
(and the Jacobians of $\Phi_{p}^{\delta}$ and $(\Phi_{p}^{\delta})^{-1}$)
are bounded by $K$;
\item For all $\left|\alpha\right|\leq2M$, $\frac{\partial^{\alpha}\left(\rho\circ(\Phi_{p}^{\delta})^{-1}\right)(0)}{\partial{z'}^{\alpha}}=\frac{\partial^{\alpha}\left(\rho\circ(\Phi_{p}^{\delta})^{-1}\right)(0)}{\partial\overline{z'}^{\alpha}}=0$,
$z'=(z_{1},\ldots,z_{n-1})$;
\item If $L_{i}=\sum a_{i}^{j}\frac{\partial}{\partial z_{i}}$, then $a_{i}^{j}(0)=\delta_{ij}$
and, for all $\Lrond\in\Lrond_{M}(\brond\cup\{N\})$,
\[
\left|\Lrond a_{i}^{j}\right|\leq K\mbox{ in }\Phi_{p}(V(p_{0}))\mbox{ and }F_{j}^{1/2}\left|\Lrond a_{i}^{j}(0)\right|\leq KF_{i}^{1/2}F^{\Lrond/2};
\]
\item For all $(\alpha,\beta)$, $\left|\alpha+\beta\right|\leq M$, $\left|\frac{\partial^{\alpha\beta}\left(\rho\circ(\Phi_{p}^{\delta})^{-1}\right)(0)}{\partial z^{\alpha}\partial\bar{z}^{\beta}}\right|\leq K\min\left\{ \delta F^{(\alpha+\beta)/2},1\right\} $.
\end{enumerate}
\end{sddefn}

One of our main goals is to prove the following existence Theorem:
\begin{stthm}
\label{thm:existence-coordinate-system}Suppose $p_{0}$ is of finite
$1$-type $\tau$, and choose an integer $M$ larger than $2\left(\frac{2\left(\frac{\tau}{2}\right)^{n-1}+1}{2}\right)^{n-1}$.
For any positive constant $K$, there exist a constant $\delta_{0}>0$,
a neighborhood $V(p_{0})$, both depending on the data, and a constant
$K'$ depending on $K$ and the data such that if $\brond=\{L_{i},\,1\leq i\leq n-1\}$
is a $\MR C^{\infty}$ basis of $(1,0)$ complex vector fields tangent
to $\rho$ in $V(p_{0})$ which is $(M,K,p,\delta)$-extremal at a
point $p\in V(p_{0})\cap\partial\Omega$, then there exists a coordinate
system $(z_{i})_{1\leq i\leq n}$ centered at $p$ which is $(K',\delta)$-adapted
to $\brond$.
\end{stthm}
The proof is divided in two steps: for the first one, in the next Section, we work without
the assumption of finite type and we construct an adapted coordinate
system using modified weights; in the second one, which is \secref{existence-coord-syst-finite-type},
we use the finite type hypothesis to deduce the Theorem.

\subsubsection{Construction of an adapted coordinate system\label{sec:Construction-of-adapted-coord-system}}

In this Section we suppose that the integer $M$ is fixed. Let $p\in V(p_{0})$
and $\delta>0$. Suppose $\brond=(L_{1},,L_{n-1})$ is a basis of
$(1,0)$ vector fields tangent to $\rho$ in $V(p_{0})$, satisfying
the following properties:
\begin{lyxlist}{00}
\item [{(A)}] The $\MR{C}^{2M}(V(p_{0}))$ norms of all $L_{i}$ are bounded
by $K$ and $\brond$ is ordered so that $F(L_{i+1},p,\delta)\leq F(L_{i},p,\delta)$.
\item [{(B)}] Let $p\in W(p_{0})\Subset V(p_{0})$ and $\delta>0$. Denoting
$\widetilde{F_{i}}=F_{i}+1=F(L_{i},p,\delta)+1$:

\begin{lyxlist}{00}
\item [{(B$_{\text{1}}$)}] For all list $\Lrond\in\Lrond_{M}(\brond\cup\{N\})$,
$\left|\Lrond(\partial\rho)(p)\right|\leq K\delta\widetilde{F}(p,\delta)^{\Lrond/2}$;
\item [{(B$_{\text{2}}$)}] $\brond$ satisfies condition EB$_{\text{2}}$
of \defref{basis-extremal} with the $F(L_{s},p,\delta)$ replaced
by the $\widetilde{F_{s}}$.
\end{lyxlist}
\end{lyxlist}

Then under these hypotheses, we have:
\begin{spprop}
\label{prop:construction-coordinate-system}There exists a constant
$K'$ depending on $K$, $M$ and the data (but neither on $p$ nor
on $\delta$) such that there exists a $(M,K',\delta)$-adapted coordinate
system to $\brond$ at $p$ in the sense of \defref{basis-and-coordinates-adapted},
the weights $F(L_{i},p,\delta)$ being replaced by $\widetilde{F_{i}}$.\end{spprop}
\begin{proof}
In \cite{Charpentier-Dupain-Geometery-Finite-Type-Loc-Diag} (Prop
3.2, p. 85) we proved that hypothesis (A) implies the existence of
a coordinate system $\Phi_{p,\delta}$ satisfying conditions (1) and
(2) of \defref{basis-and-coordinates-adapted} and
\begin{equation}
\left\{ \begin{array}{c}
\mbox{For }j<i<n,\mbox{ and }\alpha=(\alpha_{1},\ldots,\alpha_{n-1})\in\mathbb{N}^{n-1}\mbox{ such that }\left|\alpha\right|\leq M,\,\alpha_{p}=0\mbox{ if }p>i\mbox{ or }p\leq j,\\
\frac{\partial^{\alpha}a_{i}^{j}(0)}{\partial{z'}^{\alpha}}=0.
\end{array}\right.\label{eq:choice-der-coeff-fields}
\end{equation}

We now prove that under condition (B) the two last properties of \defref{basis-and-coordinates-adapted}
(with the $\widetilde{F_{i}}$) are satisfied. This follows quite
closely the ideas of p. 87-90 of \cite{Charpentier-Dupain-Geometery-Finite-Type-Loc-Diag},
but, as the context here is more general and as it is a fundamental
tool, we write it completely.

Let $\Lrond\in\Lrond(\brond\cup\{N\})$ be considered as a differential
operator. Denoting by $D^{\alpha\beta}$ the derivative $\frac{\partial^{\alpha+\beta}}{\partial z^{\alpha}\partial\bar{z}^{\beta}}$
in the coordinate system $z=\Phi_{p}^{\delta}$, it is easy to see
that, if $\left|\Lrond\right|=S$,
\[
\Lrond=\sum_{\SU{m\in\mathbb{N}^{n}\AS1\leq\left|m\right|\leq S}}\sum_{\alpha_{i}+\beta_{i}=m_{i}}c_{\alpha\beta}^{\Lrond}D^{\alpha\beta}
\]
 where
\[
c_{\alpha\beta}^{\Lrond}=c_{\alpha\beta}=\sum_{p=1}^{S}\sum*\PB{a_{j_{1}}^{i_{1}}}\cdots\PB{a_{j_{p}}^{i_{p}}}\prod_{k=p+1}^{S}D^{s_{k}}\left(\PB{a_{j_{k}}^{i_{k}}}\right)
\]
 where the summation in the second formula is taken over the derivatives
associated to the multiindex $s_{k}$ satisfying $\sum_{k=p+1}^{S}s_{k}+(m_{1},\cdots,m_{n})=\sum_{k=1}^{S}\chi(i_{k})$,
$\sum_{k=1}^{S}\chi(j_{k})=(l_{1},\cdots,l_{n-1},l_{n})$ and the
coefficients $*$ are absolute constants. The following Lemma is then
easily established:
\begin{llemsp}
\label{lem:Lemma-1-coord-sys-estimate-c-alpha-beta}If for all $s\in\mathbb{N}^{n}$,
$\left|s\right|\leq S$, we have $\left|D^{s}a_{j}^{i}(0)\right|\lesssim_{K_{1}}\widetilde{F}^{s/2}\widetilde{F_{j}}^{1/2}\widetilde{F_{i}}^{-1/2}$, then
\begin{equation}
\left|c_{\alpha\beta}(0)\right|\lesssim\widetilde{F}^{\Lrond/2}\widetilde{F}^{-\frac{\alpha+\beta}{2}}.\label{eq:estimate-c-alpha-beta-Lemma1}
\end{equation}

\end{llemsp}
To fix notations, recall that if $f$ is a $\MR{C}^{2}$ function
and $L$ and $L'$ are two vector fields, then $\left\langle \partial\bar{\partial}f;L,\bar{L}\right\rangle =\overline{L'}Lf+\left[L,\overline{L'}\right]\left(\partial f\right)$,
and, in particular, if $L\rho=0$, $\left\langle \partial\bar{\partial}\rho;L,\overline{L}\right\rangle =\left[L,\overline{L}\right]\left(\partial\rho\right)=c_{LL}$,
where $c_{LL}$ is the coefficient of the Levi form in the direction
$L$. In all the proof that follows, we denote $\left[L_{i},\overline{L_{j}}\right]\left(\partial\rho\right)=c_{ij}$.

To state the second Lemma let us introduce the notation $\tilde{\rho}=\rho\circ(\Phi_{p}^{\delta})^{-1}$:
\begin{llemsp}
\quad\mynobreakpar
\begin{enumerate}
\item For every multiindex $l$, $\left|l\right|\leq2M$, we have $\left|D^{l}\tilde{\rho}(0)\right|\lesssim\delta\widetilde{F}^{l/2}$,
where $D^{l}$ is any derivative $\frac{\partial^{\left|l\right|}}{\partial z^{\alpha}\partial{\bar{z}}^{\beta}}$
with $\left|\alpha+\beta\right|=l$.
\item For every multiindex $m\neq(0,\cdots,0)$, $\left|m\right|<M$, and
every $i,j$, $\left|D^{m}\PB{a_{i}^{j}}(0)\right|\lesssim\widetilde{F}^{m/2}\widetilde{F}_{i}^{1/2}\widetilde{F}_{j}^{-1/2}$.
\end{enumerate}
\end{llemsp}
\begin{proof}
Note first that, for (2), it suffices to get the estimate for $D^{m}a_{i}^{j}(0)$
and that the estimate (1) (resp. (2)) is trivial if $l_{n}>0$ (resp.
$m_{n}>0$) (recall $F_{n}=\delta^{-2}$ and the fact that the $\MR{C}^{2M}$ norms
of the fields $L_{i}$ are controlled). We then suppose
$l_{n}=m_{n}=0$. The proof is done by induction: the induction hypothesis
$\MR{P}_{k_{0}}$ is the two conclusions of the proposition for $\left|l\right|\leq k_{0}$
and $\left|m\right|<k_{0}$.

Remark first that $\MR{P}_{k_{0}}$ and the first property of $\MR{P}_{k_{0}+1}$
imply the second property of $\MR{P}_{k_{0}+1}$ for $j=n$: this
is evident if $i=j=n$ and, if $i<j=n$, $L_{i}r\equiv0$ implies
\[
a_{i}^{n}=\left(\frac{\partial\tilde{\rho}}{\partial z_{n}}\right)^{-1}\sum_{k=1}^{n-1}a_{i}^{k}\frac{\partial\tilde{\rho}}{\partial z_{k}},
\]
 and the result is clear because $\frac{\partial\tilde{\rho}}{\partial z_{k}}(0)=0$
for $k<n$.

Moreover, note also that, the weights $\widetilde{F}_{i}$, $i\leq n-1$,
being {}``decreasing'', the second inequality of $\MR{P}_{k_{0}}$
is trivial if $i\leq j<n$ and if $i=n$. Thus it suffices to prove
this inequality when $j<i<n$.

Let us now prove $\MR{P}_{k_{0}}$ by induction. The case $k_{0}=1$
is trivial. Let us study first the case $k_{0}=2$. By definition
of the coordinate system, $\frac{\partial^{2}\tilde{\rho}}{\partial z_{i}\partial z_{j}}(0)=0$,
and, using the notations and remarks stated before the statement of
the Lemma, we have
\begin{equation}
a_{i}^{i}\overline{a_{j}^{j}}\frac{\partial^{2}\tilde{\rho}}{\partial z_{i}\partial\bar{z_{j}}}=c_{i\bar{j}}-\sum_{(k,p)\neq(i,j)}a_{i}^{k}\overline{a_{j}^{p}}\frac{\partial^{2}\tilde{\rho}}{\partial z_{k}\partial\bar{z}_{p}}\label{eq:relation-c-{i}{j-bar}-derivees-secondes}
\end{equation}
 which implies $\frac{\partial^{2}\tilde{\rho}}{\partial z_{i}\partial\bar{z}_{j}}(0)=c_{i\bar{j}}(0)$
and gives the first inequality by definition of $F$. To prove the
second inequality, let us look at the definition of the functions
$a_{\PB{i}\PB{j}}^{\PB{k}}$. Writing the bracket $[L_{i},\overline{L_{p}}]$
with the coordinate system and taking the component of $\frac{\partial}{\partial z_{j}}$,
we get
\begin{equation}
\sum_{k'=1}^{n-1}a_{i\bar{p}}^{k'}a_{k'}^{j}=-\sum_{k=1}^{n}\overline{a_{p}^{k}}\frac{\partial}{\partial\bar{z}_{k}}\left(a_{i}^{j}\right)-c_{i\bar{p}}a_{n}^{j}.\label{eq:comp-en-dz-{j}-crochet-L-{i}L-{p}bar}
\end{equation}
 Extracting the term $\frac{\partial}{\partial\bar{z}_{p}}\left(a_{i}^{j}\right)$
and taking all at zero we obtain $\frac{\partial}{\partial\bar{z}_{p}}\left(a_{i}^{j}\right)(0)=a_{i\bar{p}}^{j}(0)$
and the inequality follows from (B$_{\text{2}}$) hypothesis.

We have now to consider $\frac{\partial a_{i}^{j}}{\partial z_{q}}$.
If $q\leq j$, the inequality comes from the decreasing property of
$\widetilde{F}_{k}$, and, if $j<q\leq i$, this derivative is
zero at the origin by the properties of the coordinate system. Suppose
then $j<i<q$. Looking at the Lie bracket $[L_{i},L_{q}]$ and taking
the component of $\frac{\partial}{\partial z_{j}}$, we obtain
\begin{equation}
a_{i}^{i}\frac{\partial}{\partial z_{i}}\left(a_{q}^{j}\right)-a_{q}^{q}\frac{\partial}{\partial z_{q}}\left(a_{i}^{j}\right)=\sum_{k\neq q}a_{q}^{k}\frac{\partial}{\partial z_{k}}\left(a_{i}^{j}\right)-\sum_{k\neq i}a_{i}^{k}\frac{\partial}{\partial z_{k}}\left(a_{q}^{j}\right)+\sum_{p=1}^{n-1}a_{iq}^{p}a_{p}^{j},\label{eq:crochet-de-Lie-[L-{i},M-{q}]-sur-dz-{j}}
\end{equation}
 and then, at the origin, $\frac{\partial}{\partial z_{q}}\left(a_{i}^{j}\right)(0)=\frac{\partial}{\partial z_{i}}\left(a_{q}^{j}\right)(0)-a_{iq}^{j}(0)=-a_{iq}^{j}(0)$,
by the properties of the coordinate system, and the conclusion comes
again from (B$_{\text{2}}$). This proves $\MR{P}_{2}$.

Let us now suppose $\MR{P}_{k_{0}}$ verified ($k_{0}<2M$). Let $D^{\tilde{l}}$
be a derivative of order $k_{0}+1$. If $D^{\tilde{l}}$ is purely
holomorphic or anti-holomorphic, then $D^{\tilde{l}}\tilde{\rho}(0)=0$.
Then we suppose $D^{\tilde{l}}=D^{l}\frac{\partial}{\partial z_{i}}\frac{\partial}{\partial\bar{z_{j}}}$,
and we denote by $\widetilde{\Lrond}=\Lrond L_{i}\overline{L_{j}}$
a list of vectors fields associated to $D^{\tilde{l}}$ (in the obvious
sense that, if $\partial/\partial z_{i}$ (resp. $\partial/\partial\bar{z_{i}}$)
appears $l_{i}$ (resp $\bar{l_{i}}$) times in $D^{l}$ then $L_{i}$
(resp. $\overline{L_{i}}$) appears $l_{i}$ (resp $\bar{l_{i}}$)
times in $\Lrond$). Applying (\ref{eq:relation-c-{i}{j-bar}-derivees-secondes}),
we get
\begin{eqnarray}
D^{l}\left(\frac{\partial^{2}\tilde{\rho}}{\partial z_{i}\partial\bar{z_{j}}}\right)(0) & = & \Lrond c_{i\bar{j}}(0)-\sum_{\SU{l_{1}\neq0\AS l_{1}+l_{2}=l}}*D^{l_{1}}\left(a_{\bar{i}}^{i}\overline{a_{j}^{j}}\right)D^{l_{2}}\left(\frac{\partial^{2}\tilde{\rho}}{\partial z_{i}\partial\overline{z_{j}}}\right)(0)\nonumber \\
 &  & -\sum_{(k,p)\neq(i,j)}D^{l}\left(a_{i}^{k}\overline{a_{j}^{p}}\frac{\partial^{2}\tilde{\rho}}{\partial z_{k}\partial\overline{z_{p}}}\right)(0)\label{eq:derivees-listes-c-{i}{j}}\\
 &  & -\sum_{\left|\alpha'\right|+\left|\beta'\right|<k_{0}-1}c_{\alpha'\beta'}D^{\alpha'\beta'}(c_{i\bar{j}})(0),\nonumber 
\end{eqnarray}
 with $*=0\textrm{ or }1$. The first term of the right hand side of
(\ref{eq:derivees-listes-c-{i}{j}}) satisfies the desired inequality
(i.e. $\lesssim\delta\widetilde{F}^{l/2}\widetilde{F}_{i}^{1/2}\widetilde{F}_{j}^{1/2}$
in modulus) by (B$_{\text{1}}$). For the second, $l_{1}$ being non
$0$, we can apply the induction hypothesis to $D^{l_{2}}\left(\frac{\partial^{2}\tilde{\rho}}{\partial z_{i}\partial\overline{z_{j}}}\right)(0)$
to get the right estimate. The third term is of the same nature because,
for $(k,p)\neq(i,j)$, $a_{i}^{k}\overline{a_{j}^{p}}(0)=0$. If we
replace $c_{i\bar{j}}$ by its expression in (\ref{eq:relation-c-{i}{j-bar}-derivees-secondes}),
the induction hypothesis $\MR{P}_{k_{0}}$implies directly (for $s<k_{0}-1$):
\[
\left|D^{s}c_{i\bar{j}}(0)\right|\lesssim\delta F^{s/2}F_{i}^{1/2}F_{j}^{1/2},
\]
and then, using \lemref{Lemma-1-coord-sys-estimate-c-alpha-beta}
for $S=k_{0}$ (whose hypothesis are also verified by the induction
hypothesis $\MR{P}_{k}$), we prove that the last term in (\ref{eq:derivees-listes-c-{i}{j}})
satisfies also the right estimate.

We finish now proving the second inequality of $\MR{P}_{k_{0}+1}$.
It suffices to consider the case $j<i<n$. Let us first look at a
derivative $D^{m}$ of the form $D^{m}=D^{s}\frac{\partial}{\partial\bar{z_{p}}}$,
$\left|s\right|=k_{0}-1$. Using formula (\ref{eq:comp-en-dz-{j}-crochet-L-{i}L-{p}bar}),
we can write
\[
D^{m}a_{i}^{j}=D^{s}\left(\sum_{t=1}^{n-1}\divideontimes a_{i\bar{p}}^{t}a_{t}^{j}-\sum_{t\neq p}\divideontimes\overline{a_{p}^{t}}\frac{\partial}{\partial\bar{z_{t}}}(a_{i}^{j})+\divideontimes c_{i\bar{p}}a_{n}^{j}\right)=D^{s}(A)-D^{s}(B)+D^{s}(C),
\]
 where $\divideontimes$ is equal to $\frac{1}{a_{p}^{p}}$. In $D^{s}(B)$,
to get a non zero term at $0$, $\overline{a_{p}^{t}}$ must be derivated
because $p\neq t$; this gives derivatives of $\frac{\partial}{\partial\bar{z_{k}}}(a_{i}^{j})$
of order $<k_{0}-1$ which are well controlled by the induction hypothesis
and then $\left|D^{s}(B)(0)\right|\lesssim\widetilde{F}^{m/2}\widetilde{F}_{i}^{1/2}\widetilde{F}_{j}^{1/2}$.

Consider now the terms $D^{s}\left(\divideontimes a_{i\bar{p}}^{t}a_{k}^{j}\right)$. 
\begin{claim*}
For $\left|l\right|\leq k$, $D^{l}\left(a_{i\PB{p}}^{t}\right)\lesssim F_{i}^{1/2}F_{p}^{1/2}F_{t}^{-1/2}F^{l/2}$.\end{claim*}
\begin{proof}
[Proof of the Claim]We do it by induction on $\left|l\right|$. (B$_{\text{2}}$)
proves the result for $\left|l\right|=0$. Assume the claim proved
for $\left|l\right|<k'\leq k_{0}-1$ and suppose $\left|l\right|=k'$.
Then,
\[
D^{l}a_{i\PB{p}}^{t}(0)=\Lrond^{l}a_{i\PB{p}}^{t}(0)+\sum_{\left|s'\right|<l}c_{s'}(0)D^{s'}a_{i\PB{p}}^{t}(0).
\]
 But, by (B$_{\text{2}}$), 
\[
\left|\Lrond^{l}a_{i\bar{p}}^{t}(0)\right|\lesssim F^{l/2}F_{i}^{1/2}F_{p}^{1/2}F_{t}^{-1/2},
\]
and for the second term of the previous identity, we have $\left|s'\right|<l$
and we can apply the induction hypothesis and \lemref{Lemma-1-coord-sys-estimate-c-alpha-beta}
whose hypotheses are satisfied, using $\MR{P}_{k_{0}}$, because $\left|l\right|\leq k_{0}$.
\end{proof}
Then the estimate of $D^{s}\left(\divideontimes a_{i\bar{p}}^{k}a_{k}^{j}\right)$
follows from the induction hypothesis $\MR{P}_{k_{0}}$ because $\left|s\right|<k_{0}$.
Thus 
\[
\left|D^{s}(A)(0)\right|\lesssim\widetilde{F}^{m/2}\widetilde{F}_{i}^{1/2}\widetilde{F}_{j}^{1/2}.
\]

Finally, the terms $D^{s}\left(\divideontimes c_{i\bar{p}}a_{n}^{j}\right)$
satisfy also the good estimates because $a_{n}^{j}(0)=0$ and, for
$\left|s'\right|<k_{0}-1$, we have seen that $\left|D^{s'}(c_{i\bar{p}})(0)\right|\lesssim\delta\widetilde{F}^{s'/2}\widetilde{F}_{i}^{1/2}\widetilde{F}_{p}^{1/2}$,
and, the derivatives of $a_{n}^{j}$ are controlled by the induction
hypothesis $\MR{P}_{k_{0}}$.

To finish, we have to consider the case where $D^{m}$ is a holomorphic
derivative. Note that the inequality is trivial if $i\leq j$ or if
there exists $k\leq j$ such that $m_{k}\neq0$. Suppose then, for
all $k\leq j$, $m_{k}=0$ and $j<i<n$. Let $q$ be the largest index
such that $m_{q}>0$. If $q\leq i$, we have $D^{m}a_{i}^{j}(0)=0$
by the properties of the coordinate system. If $q>i$, then write $D^{m}=D^{s}\frac{\partial}{\partial z_{q}}$.
To conclude it suffices then to use (\ref{eq:crochet-de-Lie-[L-{i},M-{q}]-sur-dz-{j}}),
the first Claim and the fact that $D^{s}\frac{\partial}{\partial z_{i}}\left(a_{q}^{j}\right)(0)=0$
also by the properties of the coordinates system. This completes the
proof of the Lemma.
\end{proof}
To finish the proof of \propref{construction-coordinate-system},
it suffices to note that, in addition to the estimates of the coefficients
$c_{\alpha\beta}^{\Lrond}$ given by \lemref{Lemma-1-coord-sys-estimate-c-alpha-beta},
we also have, for $\left|\alpha+\beta\right|\leq2M$,
\begin{equation}
D^{\alpha\beta}=\sum_{1\leq\left|\Lrond\right|\leq\left|\alpha+\beta\right|}d_{\Lrond}^{\alpha\beta}\Lrond,\label{eq:relation-lists-derivatives}
\end{equation}
 with $\left|d_{\Lrond}^{\alpha\beta}(0)\right|\lesssim\widetilde{F}^{(\alpha+\beta)/2}(p,\delta)\widetilde{F}^{-\Lrond/2}$.
\end{proof}

For an extremal basis we have thus proved (using \propref{control-lists}):
\begin{cor*}
If $\brond$ is $(M,K,p,\delta)$-extremal, for $\delta$ small enough,
there exists a coordinate system $(M,K'(K),\delta)$-adapted to $\brond$
in the sense of \defref{basis-and-coordinates-adapted} with the weights
$F_{i}$ replaced by $\widetilde{F_{i}}=F_{i}+1$.
\end{cor*}

\subsubsection{Proof of \thmref{existence-coordinate-system}\label{sec:existence-coord-syst-finite-type}}

If $p_{0}$ is a point of finite $1$-type $\tau$, then, by a Theorem
of D'Angelo (see \cite{DAngelo-type-fini-82,Catlin-Est.-Sous-ellipt.})
there exists a neighborhood $U(p_{0})$ such that, if $p\in\partial\Omega\cap U(p_{0})$,
then $p$ is of finite $1$-type less than $\tau'=2\left(\frac{\tau}{2}\right)^{n-1}$.
We assume that $V(p_{0})\subset U(p_{0})$. Then, if $\brond$ is a $(M,K,p,\delta)$-extremal
basis, by the Corollary of \propref{construction-coordinate-system}
we have a coordinate system $\Phi_{p,\delta}$ adapted to $\brond$
in terms of the $\widetilde{F_{i}}$. Suppose $M$ larger than $2\left(\frac{\tau'}{2}\right)^{n-1}$.
Then, considering the manifold $\zeta\mapsto(0,\ldots0,\zeta,0,\ldots,0)$,
$\left|\zeta\right|\leq\sigma$, Theorem 3.4 of \cite{Catlin-Est.-Sous-ellipt.}
(applied with a suitable constant $\sigma$) gives us a derivative
of $\tilde{\rho}=\rho\circ\Phi_{p,\delta}$ which is bounded from
below by a constant depending only on the data. The last property
of \defref{basis-and-coordinates-adapted} shows thus that $\widetilde{F_{i}}(p,\delta)\gtrsim\delta^{-2/M}$
with a constant depending only on the data, and, of course, the same
is true for $F_{i}(p,\delta)$.

This proves the following essential Proposition:
\begin{spprop}
\label{prop:finite-type-imply-large-F(L-p-delta)}Let $p_{0}\in\partial\Omega$
be a point of finite $1$-type $\tau$. Let $M=M(\tau)=\left[2\left(\frac{\tau}{2}\right)^{n-1}\right]+1$.
Then for any integer $K$ there exist a real number $\delta_{0}>0$
and a constant $C$, depending on $K$ and the data, such that, if
there is a coordinate system $(M,K,\delta)$-adapted to a basis
$\brond=(L_{1},\ldots,L_{n-1})$ at $p_{0}$, then $F_{M}(L_{i},p_{0},\delta)\geq C\delta^{-2/M}$.
In particular, if $\tau'=2\left(\frac{\tau}{2}\right)^{n-1}$ and
$M'=M'(\tau)=\left[2\left(\frac{\tau'}{2}\right)^{n-1}\right]+1$,
for any integer $K$ there exists a neighborhood $V(p_{0})$, a real
number $\delta_{0}>0$ and a constant $C$ (depending on $\tau$,
$\Omega$ and $K$) such that, for $p\in V(p_{0})\cap\partial\Omega$
and $0<\delta\leq\delta_{0}$, if there is a coordinate system
$(M',K,\delta)$-adapted to a basis $\brond=(L_{1},\ldots,L_{n-1})$
at $p$, then $F(L_{i},p,\delta)\geq C\delta^{-2/M'}$.
\end{spprop}
This proves completely \thmref{existence-coordinate-system}.
\begin{strem}
Note that the proofs show that if a basis $\brond$ satisfies only
properties (A) and (B) of the beginning of \secref{Construction-of-adapted-coord-system},
then, under the assumption of finite $1$-type, the conclusions of
\propref{finite-type-imply-large-F(L-p-delta)} and \thmref{existence-coordinate-system}
are still valid.
\end{strem}
A simple consequence (which will be used in \secref{localization-extremal-basis})
of the minoration of the weights $F_{i}$ is the following:
\begin{sllem}
\label{lem:multiplication-fields-by-function-Cinfinity}Suppose the
point $p_{0}$ of finite $1$-type $\tau$. For any $K$, there exist
two constants $C$ and $\delta_{0}$, depending only on $K$, $\tau$
and the data, such that if $\brond=\{L_{i}^{p,\delta},\, i<n\}$ is
$(K,p,\delta)$-extremal, $p\in V(p_{0})\cap\partial\Omega$, and
$(\alpha_{i})$ is a family of $\MR{C}^{\infty}$ functions, of $\MR{C}^{2M}$
norm $\leq K$ and $1/K\leq\left|\alpha_{i}\right|\leq K$, then the
basis $\brond_{1}=\{L_{i}\}$, where $L_{i}=\frac{1}{\alpha_{i}}L_{i}^{p,\delta}$,
is $(C,p,\delta)$-extremal, and, moreover, $F\left(\sum a_{i}L_{i},p,\delta\right)\simeq_{C}F\left(\sum a_{i}L_{i}^{p,\delta},p,\delta)\right)$,
$a_{i}\in\mathbb{C}$.
\end{sllem}

\subsubsection{Associated polydiscs and pseudo-balls for finite type points\label{sec:polydisc-associated-extremal-basis}}

In this Section we assume $p_{0}$ is of finite $1$-type $\tau$ and
we choose $M=M'(\tau)$. Now we will associate to an adapted coordinate
system some special ``polydiscs'' and give some related properties.
\begin{sddefn}
\label{def:definition-polydissc-ext-basis}Let $W(p_{0})\Subset V(p_{0})$
small enough. Suppose that for some point $p\in W(p_{0})\cap\partial\Omega$
and $0<\delta$ there is a basis $\brond(p,\delta)=\left\{ L_{i}^{p,\delta}\right\} $
of $(1,0)$ vector fields tangent to $\rho$ in $V(p_{0})$ satisfying
conditions (A) and (B) (of \secref{Construction-of-adapted-coord-system})
and let $\Phi_{p}^{\delta}=\Phi_{p}$ be a coordinate system which
is $(K,\delta)$-adapted to $\brond(p,\delta)$. Then the functions
$F(L_{i},p,\delta)=F_{i}(p,\delta)$ do not vanish and, for $0<c<1$,
we denote 
\[
\Delta_{c}(p,\delta)=\{z\in\mathbb{C}^{n}\mbox{ such that }\left|z_{i}\right|<cF_{i}^{-1/2},\,1\leq i\leq n\},
\]
and 
\[
B^{c}(p,\delta)=\Phi_{p}^{-1}(\Delta_{c}(p,\delta))\cap V(p_{0}).
\]

\end{sddefn}
Taylor's formula, \propref{control-lists} and \thmref{existence-coordinate-system}
lead easily to the following properties (denoting $L_{i}=L_{i}^{p,\delta}$):
\begin{spprop}
\label{prop:properties-vector-coord-weights-balls-3-5}There exist
three constants $c_{0}$, $K_{0}$ and $\delta_{0}$, depending only
on $K$ and the data, such that the following properties hold:
\begin{enumerate}
\item If $L_{i}=\sum a_{i}^{j}\frac{\partial}{\partial z_{j}}$ and $\frac{\partial}{\partial z_{j}}=\sum b_{j}^{i}L_{i}$,
$\left|\alpha+\beta\right|\leq M$, for $z\in\Delta_{c_{0}}(p,\delta)$,
\begin{eqnarray*}
\left|D^{\alpha\beta}a_{i}^{j}(z)\right| & \leq & K_{0}F^{(\alpha+\beta)/2}(p,\delta)F_{i}^{1/2}(p,\delta)F_{j}^{-1/2}(p,\delta),\\
\left|D^{\alpha\beta}b_{i}^{j}(z)\right| & \leq & K_{0}F^{(\alpha+\beta)/2}(p,\delta)F_{i}^{1/2}(p,\delta)F_{j}^{-1/2}(p,\delta).
\end{eqnarray*}

\item If $\Lrond\in\Lrond_{M}\left(\brond(p,\delta)\cup\{N\}\right)$, $\left|\Lrond\right|=S$,
and $D^{T}$ is a derivative in the coordinate system $(z)$ with
$\left|T\right|\leq M$, then $\Lrond=\sum_{\left|s\right|\leq S}c_{s}D^{s}$,
$D^{T}=\sum_{\left|\Lrond'\right|\leq\left|T\right|}d_{\Lrond'}\Lrond'$,
and, for $z\in\Delta_{c_{0}}(p,\delta)$ and $q=\Phi_{p}(z)$ we have
\begin{eqnarray*}
\left|c_{s}(z)\right| & \leq & K_{0}F^{(\Lrond-s)/2}(p,\delta),\\
\left|d_{\Lrond'}(q)\right| & \leq & K_{0}F^{(\Lrond-\Lrond')/2}(p,\delta).
\end{eqnarray*}

\item For $L=\sum a_{i}L_{i}$, $a_{i}\in\mathbb{C}$, for all $q\in B^{c_{0}}(p,\delta)$,
$\frac{1}{2}F(L,p,\delta)\leq F(L,q,\delta)\leq2F(L,p,\delta)$.
\item For all list $\Lrond$, $\left|\Lrond\right|\leq M$ belonging to
$\Lrond_{M}(\brond)$ and all point $q\in B^{c}(p,\delta)$,

\begin{enumerate}
\item $\left|\Lrond(\partial\rho)(q)\right|\leq K_{0}\delta F(p,\delta)^{\Lrond/2}$,
\item with the notation introduced in EB$_{\text{2}}$ in \defref{basis-extremal},
\[
\left|\Lrond a_{\PB{i}\PB{j}}^{\PB{k}}(q)\right|\leq K_{0}F^{\Lrond/2}(p,\delta)F_{i}^{1/2}(p,\delta)F_{j}^{1/2}(p,\delta)F_{k}^{-1/2}(p,\delta).
\]

\end{enumerate}
\item $\rho(B^{c}(p,\delta))\subset[-\frac{1}{2}\delta,\frac{1}{2}\delta]$.
\end{enumerate}
\end{spprop}

The proofs are almost straightforward computations.

\smallskip{}

In \secref{Geometrically-separated-domains} we will need to use two
other kinds of {}``pseudo-balls'' and we will prove that they are
closely related to the {}``polydisc'' $B^{c}$:
\begin{sddefn}
\label{def:definition-pseudo-balls-curves-exp}Suppose that $\brond=(L_{1},\ldots,L_{n-1})$
is a basis satisfying conditions (A) and (B) (at a point of finite
$1$-type).
\begin{enumerate}
\item Denote $\mathcal{Y}_{i}=\Re\mathrm{e}L_{i}$ and $\mathcal{Y}_{i+n}=\Im\mathrm{m}L_{i}$,
$1\leq i\leq n$ (recall $L_{n}=N$). Then we denote by $B_{\MR{C}}^{c}(\brond,p,\delta)$
the set of points $q\in V(p_{0})$ for which there exists a piecewise
$\MR C^{1}$ curve $\varphi:[0,1]\rightarrow V$ such that $\varphi(0)=p$,
$\varphi(1)=q$ and $\varphi'(t)=\sum a_{i}\mathcal{Y}_{i}(\varphi(t))$,
with $\max(\left|a_{i}\right|,\left|a_{i+n}\right|)\leq cF^{-1/2}(L_{i},p,\delta)$,
$0<c<1$.
\item $\mathrm{exp}_{p}$ denoting the exponential map based at $p$ associated
to the vector fields $\mathcal{Y}_{i}$ (defined in (1)), for $0<c<1$, we put
\[
B_{\mathrm{exp}}^{c}(p,\delta)=\left\{ q=\mathrm{exp}_{p}(u_{1},\ldots,u_{2n}),\mbox{ such that }\max(\left|u_{i}\right|,\left|u_{i+n}\right|)\leq cF_{i}(p,\delta)^{-1/2}\right\} \cap V(p_{0}).
\]

\end{enumerate}
\end{sddefn}

The terminology used in \defref{basis-extremal} is justified by the
following property:
\begin{spprop}
\label{prop:maximality-extremal-balls}Let $\brond=\{L_{1},\ldots,L_{n-1}\}$
be a basis (of $(1,0)$ complex vector fields, tangent to $\rho$
in $V(p_{0})$) satisfying conditions (A) and (B) (for example if
it is $K$-extremal) at $p\in W(p_{0})\cap\partial\Omega$. Let $\brond^{1}=\{L_{1}^{1},\ldots,L_{n-1}^{1}\}$
be another basis in $V(p_{0})$ such that, for all $i$, $L_{i}^{1}=\sum a_{i}^{j}L_{j}$,
$a_{i}^{j}\in\mathbb{C}$, $\sum\left|a_{i}\right|^{2}=1$. Then there
exists a constant $A$ depending only on $K$, $\tau$ and the dimension
$n$ such that $B_{\MR{C}}^{c}(\brond^{1},p,\delta)\subset B_{\MR{C}}^{Ac}(p,\delta)$.
\end{spprop}
The proof of this Proposition immediately follows from property (B).

\subsection{Sufficient conditions of extremality\label{sec:sufficient-cond-of-extremality}}

In this Section we always assume that $p_{0}$ is a point of finite
$1$-type $\tau$ and choose $M=M(\tau)$.

Here and in \secref{PSH-func-strong-geom-sep} we will need a stronger
control on certain derivatives of the coefficients of the Levi form.
Thus we introduce the following condition: suppose $\brond$ is a
basis of $(1,0)$ vector fields tangent to $\rho$ in $V(p_{0})$.
We say that it satisfies condition B($\alpha$), $\alpha>0$, if
for all lists $\Lrond\in\Lrond_{M-2}\left(\brond\right)$ we have
\begin{lyxlist}{00.00.00}
\item [{B($\alpha$)}] for $i\neq j$, $1\leq i,j\leq n-1$, $\left|\Lrond c_{ij}(p)\right|\leq\alpha\delta F(p,\delta)^{\Lrond/2}F(L_{i},p,\delta)^{1/2}F(L_{j},p,\delta)^{1/2}$.
\end{lyxlist}

Note that B($\alpha$) together with conditions (A) and (B) implies
a new condition on the brackets of the vector fields:
\begin{sllem}
\label{lem:lemma32}Suppose $\brond$ satisfies conditions (A) and
(B). Then there exist two constants $K_{1}=K_{1}(K,M,n)$ and $\delta_{0}$ depending
on $K$, $\alpha$ and the data such that, for all $i\neq k$, $i,\, k<n$,
$j\leq n$ and all $\Lrond\in\Lrond_{M}(\brond\cup\{N\})$, if $\brond$
satisfies B($\alpha$) at $(p,\delta)$, $p\in W(p_{0})$, $0<\delta\leq\delta_{0}$, then
\[
\left|\Lrond a_{\PB{i}\PB{k}}^{\PB{j}}(p)\right|\leq K_{1}\alpha F(p,\delta)^{\Lrond/2}F_{i}(p,\delta)^{1/2}F_{k}(p,\delta)^{1/2}F_{j}(p,\delta)^{-1/2}.
\]
\end{sllem}
\begin{proof}
To simplify the notations we write the proof for $a_{j\bar{k}}^{\bar{j}}$.
Choose $\delta_{0}$ so that $C\delta_{0}^{-2/M}>\alpha^{-1}$, where
$C$ is the constant of \propref{finite-type-imply-large-F(L-p-delta)}.
Note that the property is trivial if $l_{n}\neq0$ or if $l_{n}=0$
and $j=n$ ($a_{i\bar{k}}^{\bar{n}}=\frac{1}{2}c_{ik}$ and $a_{ik}^{\bar{n}}=0$),
thus we suppose $l_{n}=0$ and $j<n$. As the property is also trivial
if $j$ or $k$ is $\geq i$, we have to study only the case when
$j<\min(i,k)$.

To simplify the notations, we introduce the following spaces of functions:
\[
\STR_{0}=\{\varepsilon,\,\varepsilon a_{\PB{i}\PB{j}}^{\PB{k}},\,\varepsilon c_{\PB{i}\PB{j}},\textrm{ where }\varepsilon\in\{-1,0,1,-\sqrt{-1},\sqrt{-1}\}\},
\]
 and
\[
\widetilde{\STR}_{k+1}=\bigcup_{i}\LB{i}(\STR_{k})\cup\STR_{k}\textrm{ and }\STR_{k+1}=\left\{ \sum_{i=1}^{3}f_{i},\, f_{i}\in\widetilde{\STR}_{k+1}\right\} .
\]
 The elements of $\STR_{k}$ will be generically denoted by $*_{k}$.

The Jacobi identity applied to the bracket $\left[L_{j},\left[L_{i},\overline{L_{k}}\right]\right]$
implies
\[
a_{i\bar{k}}^{\bar{j}}c_{jj}+L_{j}c_{ik}+\sum_{p\neq j}a_{j\bar{k}}^{\bar{p}}c_{jp}-a_{j\bar{k}}^{\bar{i}}c_{ii}-L_{i}c_{jk}-\sum_{p\neq i}a_{j\bar{k}}^{\bar{p}}c_{ip}-a_{ij}^{k}c_{kk}-\sum_{p\neq k}a_{ji}^{p}c_{pk}=0
\]
 which we write $a_{i\bar{k}}^{\bar{j}}c_{jj}=*_{0}c_{ii}+*_{0}c_{kk}+h$.
Then, by induction on the length $l$ of a list $\Lrond\in\Lrond_{M}(L_{j})$,
it is easy to show that
\[
a_{i\bar{k}}^{\bar{j}}\Lrond c_{jj}=\Lrond h+\sum_{\Lrond'\in\Lrond_{\left|\Lrond\right|}(L_{j})}\left(*_{l}\Lrond'c_{ii}+*_{l}\Lrond'c_{kk}\right)+\sum_{\Lrond'\in\Lrond_{\left|\Lrond\right|-1}(L_{j})}*_{l}\Lrond'c_{jj},
\]
and choosing $\Lrond$ so that $\left|\Lrond c_{jj}(p)\right|\gtrsim\delta F(p,\delta)^{\left(\left|\Lrond\right|+2\right)/2}$,
the Lemma is easily proved using the control on the lists and the
hypothesis.
\end{proof}
\medskip{}

Now we first prove that conditions B($\alpha$), (A) and (B) imply
the extremality of the basis and then that \lemref{lemma32} implies
a better control on mixed lists. This result will be important in
\secref{Adapted-pluri-subharmonic-function-geom-sep}.
\begin{sllem}
\label{lem:Lemma-sufficients-conditions-extremality}Suppose that
$\brond=(L_{1},\ldots,L_{n-1})$ is a basis of $(1,0)$ vector fields
in $V(p_{0})$ satisfying conditions (A) and (B) at a point $p\in V(p_{0})\cap\partial\Omega$
for a fixed $\delta$.

Then there exists a function $\alpha(K)$,
depending on $K$ and the data, such that, if $\brond$ satisfies B($\alpha$)
for $\alpha\leq\alpha(K)$, there exists a constant $K_{1}$, depending
on $K$, $M$ and $n$, such that:

If $\Lrond^{0}\in\Lrond_{M}(\brond)$ satisfies
$\left|\Lrond^{0}c_{ii}(p)\right|\geq\frac{1}{K}\delta F_{i}(p,\delta)F(p,\delta)^{\Lrond/2}$
then there exists $k_{0}$, $2k_{0}+2\leq\left|\Lrond\right|$, such
that $\Re\mathrm{e}\left(\left(L_{i}\overline{L_{i}}\right)^{k_{0}}c_{ii}\right)(p)>\frac{1}{K_{1}}\delta F_{i}(p,\delta)^{(2k_{0}+2)/2}$.
In particular, 
\[
F_{i}(p,\delta)\geq\frac{1}{K'}\sum_{\SU{\Re\mathrm{e}\left(\left(L_{i}\overline{L_{i}}\right)^{k}c_{ii}\right)(p)>0\AS2k+1\leq M}}\left[\frac{\Re\mathrm{e}\left(\left(L_{i}\overline{L_{i}}\right)^{k}c_{ii}\right)(p)}{\delta}\right]^{\frac{2}{2k+2}},
\]
 where $K'$ is a constant depending only on $K$ and the data.\end{sllem}
\begin{proof}
First we fix the notations used in the proof. We know that there is a coordinate system
$\Phi_{p}^{\delta}$ adapted to $\brond$. We denote by $\left(z_{i}\right)$ theses coordinates.
Let $D^{\alpha\beta}$ denote the derivative
$\frac{\partial^{\left|\alpha+\beta\right|}}{\partial z^{\alpha}\partial\bar{z}^{\beta}}$
with respect to $\left(z_{i}\right)$, and if $\Lrond$ is a list of vector fields
let $D^{\Lrond}$ be the derivative $D^{\alpha\beta}$
with $\alpha_{i}=l_{i}^{1}(\Lrond)$ and $\beta_{i}=l_{i}^{2}(\Lrond)$
(notation of \lemref{lemma2-est-lists}).

In the proof we will use
a general result on derivatives of positive functions proved in \secref{Appendix}.

Suppose $\Lrond\in\Lrond_{M}(\brond)$ is such that $\Lrond(\partial\rho)=\Lrond^{0}c_{ii}$
and $\left|\Lrond(\partial\rho)(p)\right|\gtrsim_{K}\delta F_{i}(p,\delta)F(p,\delta)^{\Lrond/2}$.
Then we can write
\[
\Lrond(\partial\rho)=D^{\Lrond^{0}}c_{ii}+\sum_{\left|\alpha+\beta\right|<\left|\Lrond^{0}\right|}c_{\alpha\beta}D^{\alpha\beta}c_{ii}
\]
 with $\left|c_{\alpha\beta}\right|\lesssim_{K}F^{\Lrond^{0}/2}F^{-(\alpha+\beta)/2}$.

Thus there exists a derivative $D^{\alpha\beta}$ satisfying $\left|D^{\alpha\beta}c_{ii}(0)\right|\gtrsim_{K}\delta F_{i}F^{(\alpha+\beta)/2}$
and $\left|\alpha+\beta\right|\leq\left|\Lrond^{0}\right|$ and $\alpha_{n}+\beta_{n}=0$
(indeed, if $\alpha_{n}+\beta_{n}\geq1$, $\left|c_{\alpha\beta}(0)\right|\lesssim_{K}F^{\Lrond^{0}/2}F^{-(\alpha+\beta)/2}\leq\delta F^{\Lrond^{0}/2}$,
and, as $\left|D^{\alpha\beta}c_{ii}\right|\lesssim_{K}1$, $\left|c_{\alpha\beta}D^{\alpha\beta}c_{ii}\right|\ll\delta F^{\Lrond/2}$).
Then applying \lemref{deriv-positives-functions} to the function
$g(z)=\delta F_{i}^{-1}(p,\delta)c_{ii}\circ\Phi_{p,\delta}^{-1}(z')$,
where $z'=\left(cF_{1}^{-1/2}z_{1},\ldots,cF_{n-1}^{-1/2}(p,\delta)z_{n-1},0\right)$
with $c\leq c_{0}$, $c_{0}$ given by \propref{properties-vector-coord-weights-balls-3-5},
we conclude that there exists a derivative $D^{\alpha^{1}\beta^{1}}$,
satisfying $\alpha_{j}^{1}=\beta_{j}^{1}$, $\forall j$, $\alpha_{n}^{1}=\beta_{n}^{1}=0$,
such that $D^{\alpha^{1}\beta^{1}}c_{ii}(0)\geq_{K}F_{i}F^{(\alpha^{1}+\beta^{1})/2}$.

Writing $\Lrond'=\left(\overline{L_{i}}L_{i}\right)^{\alpha_{i}^{1}}\prod_{j\neq i,\, j<n}\left(\overline{L_{j}}L_{j}\right)^{\alpha_{j}^{1}}$
and $\Lrond'c_{ii}=D^{\alpha^{1}\beta^{1}}c_{ii}+\sum_{\left|\alpha+\beta\right|<\left|\Lrond'\right|}c_{\alpha\beta}D^{\alpha\beta}c_{ii}$,
by induction we conclude that there exists a differential operator
$\Lrond^{1}$ of the form $\Lrond^{1}=\left(\overline{L_{i}}L_{i}\right)^{\alpha_{i}}\prod_{j\neq i,\, j<n}\left(\overline{L_{j}}L_{j}\right)^{\alpha_{j}}$
such that $\Re\mathrm{e}\left(\Lrond^{1}c_{ii}\right)(p)\gtrsim_{K}\delta F^{\Lrond^{1}/2}F_{i}$.
Suppose there exists $j\neq i$ such that $\alpha_{j}\neq0$. Then
\[
\Lrond^{1}c_{ii}=\Lrond'\overline{L_{j}}L_{j}c_{ii}=\Lrond'\overline{L_{j}}\left(-\gamma_{i}^{j}c_{jj}+L_{k}c_{jk}+\left(a_{jk}^{i}-a_{\bar{i}j}^{\bar{i}}\right)c_{ii}-\sum_{p\neq i}\left(a_{\bar{i}j}^{\bar{p}}c_{ip}-a_{ji}^{p}c_{pi}\right)+\sum_{p\neq j}\gamma_{i}^{\bar{p}}c_{ip}\right).
\]
 The controls of the coefficients $a_{ij}^{p}$ and of the lists $\Lrond c_{kp}$,
$k\neq p$ (by condition (B)), imply, for $\alpha$ sufficiently
small (depending only on $K$), that
\[
\left|\Lrond'\overline{L_{j}}c_{jj}\right|\gtrsim_{K}\delta F^{\Lrond'/2}F_{j}^{3/2}\mbox{ and }\left|\gamma_{i}^{j}\right|\gtrsim_{K}F_{i}F_{j}^{-1/2}.
\]

Repeating the initial procedure, we conclude that there exists a list
$\Lrond''\in\Lrond(\brond)$, {}``completely even'', $\left|\Lrond''\right|\leq\left|\Lrond'\right|$
such that $\left|\Lrond''c_{jj}\right|\gtrsim_{K}\delta F^{\Lrond''/2}F_{j}$.
Consider then
\[
\Lrond''\overline{L_{j}}c_{ii}=\Lrond''\left(-\gamma_{i}^{j}c_{jj}+L_{k}c_{jk}+\left(a_{jk}^{i}-a_{\bar{i}j}^{\bar{i}}\right)c_{ii}-\sum_{p\neq i}\left(a_{\bar{i}j}^{\bar{p}}c_{ip}-a_{ji}^{p}c_{pi}\right)+\sum_{p\neq j}\gamma_{i}^{\bar{p}}c_{ip}\right).
\]
 Thus $\left|\Lrond''c_{jj}\gamma_{i}^{j}\right|\gtrsim\delta F^{\Lrond''/2}F_{j}^{1/2}F_{i}$,
and, by similar arguments, for $\alpha$ sufficiently small, we conclude
that there exists a list $\Lrond^{2}$, $\left|\Lrond^{2}\right|<\left|\Lrond^{0}\right|$such
that $\Lrond^{2}c_{ii}\gtrsim_{K}\delta F^{\Lrond^{2}/2}F_{i}$, and
we can repeat the procedure. The Lemma is thus proved by induction.\end{proof}
\begin{spprop}
There exist constants $\alpha_{0}$ and $K'$ depending on $K$ and
the data such that if the basis $\brond$ satisfies (A), (B) and B($\alpha$)
for $\alpha\leq\alpha_{0}$ at $(p,\delta)$, $p\in V(p_{0})$, then
$\brond$ is $(K',p,\delta)$-extremal.\end{spprop}
\begin{proof}
We may suppose the basis ordered so that the weights $F_{i}=F(L_{i},p,\delta)$
are ordered decreasingly. Let $L=\sum_{i=1}^{n-1}a_{i}L_{i}$, $a_{i}\in\mathbb{C}$,
$\sum\left|a_{i}\right|^{2}=1$ so that $c_{LL}=\sum_{i=1}^{n-1}\left|a_{i}^{2}\right|c_{ii}$.
Denote $F(L)=F(L,p,\delta)$. By hypothesis (B) it is clear that $F(L)\lesssim_{K}\sum\left|a_{i}\right|^{2}F_{i}$.

To show the converse inequality, we prove the following assertion:
\begin{claim*}
For every constant $K>0$, there exists a constant $K_{1}$, depending
on $K$ and the data, such that:

if $i_{0}\in\{1,\ldots,n-1\}$ and $k_{0}\in\{1,\ldots,M\}$ are such
that $\left|a_{i_{0}}\right|^{2}F_{i_{0}}(p)\geq\frac{\sum\left|a_{i}\right|^{2}F_{i}(p)}{K}$
and $\Re\mathrm{e}\left(L_{i_{0}}\overline{L_{i_{0}}}\right)^{k_{0}}c_{i_{0}i_{0}}(p)>\delta\frac{F_{i_{0}}^{k_{0}+1}(p)}{K}$,
then:
\begin{itemize}
\item either $\Re\mathrm{e}\left(L\bar{L}\right)^{k_{0}}c_{L\bar{L}}>\delta\frac{\left(\sum\left|a_{i}\right|^{2}F_{i}(p)\right)^{k_{0}+1}}{K_{1}}$,
\item or there exist $i_{1}$ and $k_{1}<k_{0}$ such that $\left|a_{i_{1}}\right|^{2}F_{i_{1}}(p)\geq\frac{\sum\left|a_{i}\right|^{2}F_{i}(p)}{K_{1}}$
and $\Re\mathrm{e}\left(L_{i_{1}}\overline{L_{i_{1}}}\right)^{k_{1}}c_{i_{1}i_{1}}(p)>\delta\frac{F_{i_{1}}^{k_{1}+1}(p)}{K_{1}}$.
\end{itemize}
\end{claim*}
\begin{proof}
[Proof of the Claim]We have
\begin{equation}
\left(L\bar{L}\right)^{k_{0}}c_{L\bar{L}}=\sum\left|a_{i}\right|^{2k_{0}+2}\left(L_{i}\overline{L_{i}}\right)^{k_{0}}c_{ii}+\sum\alpha_{\MR{L}}\MR{L}(\partial\rho),\label{eq:claim-proof-extremal-diag}
\end{equation}
 where the second sum contains lists of length $2k_{0}+2$ containing
$L_{i}$ or $\overline{L_{i}}$ for, at least, two different values
of $i$. As
\[
\left|a_{i_{0}}\right|^{2k_{0}+2}\Re\mathrm{e}\left(L_{i_{0}}\overline{L_{i_{0}}}\right)^{k_{0}}c_{i_{0}i_{0}}(p)>\delta\frac{\left(\sum\left|a_{i}\right|^{2}F_{i}(p)\right)^{k_{0}+1}}{K^{k_{0}+2}},
\]
 the conclusion is clear except in the two following cases:
\begin{itemize}
\item in the first sum of (\ref{eq:claim-proof-extremal-diag}), there is a term whose real part is
$<-A=-\delta\frac{\left(\sum\left|a_{i}\right|^{2}F_{i}(p)\right)^{k_{0}+1}}{CK^{k_{0}+2}}$;
\item in the second sum of (\ref{eq:claim-proof-extremal-diag}), there is a term which is,
in modulus, bigger than $A$, with a constant $C$ depending only on $M$ and the coefficients $a_{i}$.
\end{itemize}

Suppose first that there exists an index $i\neq i_{0}$ such that
$\left|a_{i}\right|^{2k_{0}+2}\Re\mathrm{e}\left(L_{i}\overline{L_{i}}\right)^{k_{0}}c_{ii}(p)<-A$.
This implies first $\left|a_{i}\right|^{2}F_{i}(p)\geq\frac{\sum\left|a_{i}\right|^{2}F_{i}(p)}{K'_{1}}$
and secondly $\Re\mathrm{e}\left(L_{i}\overline{L_{i}}\right)^{k_{0}}c_{ii}(p)<-\delta\frac{1}{K''_{1}}F_{i}^{k_{0}+1}$.
By \lemref{Lemma-sufficients-conditions-extremality} there exists
$k_{1}<k_{0}$ such that $\Re\mathrm{e}\left(L_{i}\overline{L_{i}}\right)^{k_{1}}c_{ii}(p)>\delta\frac{1}{K'''_{1}}F_{i}^{k_{1}+1}$.
Thus the second assertion of the Claim is verified.

Suppose now that there is a term $\alpha_{\MR{L}}\MR{L}(\partial\rho)$
in the second sum of (\ref{eq:claim-proof-extremal-diag}) satisfying
$\left|\alpha_{\MR{L}}\MR{L}(\partial\rho)\right|>A$. Denote by $l_{i}$
the number of times the vector fields $L_{i}$ and $\overline{L_{i}}$
appear in $\MR{L}$. If $l_{k}\neq0$, hypothesis (B) implies immediately
$\left|a_{k}\right|^{2}F_{k}\gtrsim\sum\left|a_{i}\right|^{2}F_{i}$
and $\left|\MR{L}\right|(\partial\rho)\gtrsim\delta\prod F_{i}^{l_{i}/2}$.
\end{proof}
\end{proof}
\begin{cor*}
Suppose that $p_{0}\in\partial\Omega$ is a point of finite type $\tau$
where the Levi form is locally diagonalizable. Then there exists a
neighborhood $V(p_{0})$ of $p_{0}$ and constants $K$ and $\delta_{0}>0$
such that at every point $p$ of $V(p_{0})\cap\partial\Omega$ and
for every $0<\delta\leq\delta_{0}$, the basis diagonalizing the Levi
form is $(M,p,\delta)$-extremal (with $M=M'(\tau)$).\end{cor*}
\begin{proof}
Properties (A) and (B) were proved in \cite{Charpentier-Dupain-Geometery-Finite-Type-Loc-Diag},
and, by definition the basis diagonalizing the Levi form satisfies B($\alpha$)
for all $\alpha>0$.
\end{proof}
\bigskip{}

\begin{sddefn}
\label{def:basis-strongly-extremal}$\brond$ is called \emph{$(K,\alpha,p,\delta)$-strongly-extremal}
if it is $(K,p,\delta)$-extremal and, if, it satisfies B($\alpha$)
at $(p,\delta)$.
\end{sddefn}

Note that the first part of \propref{control-lists} says that every
$(K,p,\delta)$-extremal basis is \emph{$(K,\alpha,p,\delta)$}-strongly-extremal
for some large positive number $\alpha$ depending on $K$ and $\Omega$.
Thus this is an extra hypothesis only for small $\alpha$.

The next Proposition shows that for a strongly extremal basis some
derivatives of the diagonal terms of the Levi matrix satisfy a better
control:
\begin{spprop}
\label{prop:automatic-prop-of-strong-extremality}Suppose $p_{0}$ is
of finite $1$-type $\tau$ and let $M=M'(\tau)$. Then there exists a
neighborhood $V(p_{0})$ of $p_{0}$ with the following property:

for $\alpha>0$, there exist constants $\delta_{0}=\delta_{0}(\alpha,\mathrm{data})$ and
$K'=K'(K,\mathrm{data})$ such that:

if $\brond$ is a $(K,\alpha,p,\delta)$-strongly-extremal
basis, ordered so that $F_{i}$ are decreasing, then for all lists
$\Lrond\in\Lrond_{2M}(\brond)$ such that there exists $j>i$ with $l_{j}\neq0$ we have
$\left|\Lrond c_{ii}(p)\right|\leq K'\alpha F(p,\delta)^{\Lrond/2}F_{i}(p,\delta)$.\end{spprop}
\begin{proof}
Let $\Lrond=\Lrond'\LB{j}\LB{p}\Lrond''$ with $j\leq i$ and write
\[
\Lrond c_{ii}=\Lrond'\LB{p}\LB{j}\Lrond''c_{ii}+\sum\Lrond'\left(a_{\PB{j}\PB{p}}^{k}L_{k}+a_{\PB{j}\PB{p}}^{\bar{k}}\overline{L_{k}}\right)\Lrond''c_{ii}.
\]
 Then successive application of \lemref{lemma32} show that there
exists a list $\widetilde{\Lrond}=\widetilde{\Lrond'}L_{j}$ such
that, for all $k$, $\widetilde{l_{k}}=l_{k}$ and $\left|\widetilde{\Lrond}c_{ii}-\Lrond c_{ii}\right|\leq K_{2}\alpha F^{\Lrond/2}F_{i}$.

Now the result is trivial, applying once again \lemref{lemma32}, \lemref{lemma1-est-lists}
and the hypothesis B($\alpha$).\end{proof}
\begin{spprop}
\label{prop:3.9-automatic-lists-inside-strong-extremality}If the
basis $\brond$ is $(K,\alpha,p,\delta)$-strongly extremal, the conclusion
of \propref{automatic-prop-of-strong-extremality} is still valid
at each point $q\in B^{c_{0}}(p,\delta)$ with $\alpha$ replaced
by $2\alpha$ for $\delta\leq\delta(\alpha)$ ($\delta(\alpha)$ depending
on $\alpha$, $K$ and the data).
\end{spprop}

\subsection{Localization of extremal bases\label{sec:localization-extremal-basis}}

\subsubsection{Definition of the local domain\label{sec:definition-of-the-local-domain}}
\begin{sddefn}
\label{def:local-domain}Let $\Omega$ be a bounded pseudo-convex
domain in $\mathbb{C}^{n}$. Suppose that $P_{0}$ is a boundary point
of $\Omega$ and $W(P_{0})\Subset V(P_{0})$ are neighborhoods of $P_{0}$.
Let $O$ be a point of the real normal to $\partial\Omega$ at $P_{0}$
and denote by $d$ the distance from $O$ to $P_{0}$. Let us denote
by $(z_{i})_{1\leq i\leq n}$ the coordinate system obtained translating
the origin at $O$.

Let $\mu>0$ and $\psi(z)=\varphi\left(\left|z\right|^{2}\right)$
where
\[
\varphi(x)=\left\{ \begin{array}{ll}
0 & \mbox{if }x\leq\mu^{2},\\
K_{0}e^{-1/(x-\mu^{2})} & \mbox{if }x\geq\mu^{2},
\end{array}\right.
\]
 with $\frac{4}{3}d\leq\mu\leq2d$.

Let us denote $r(z)=\rho(z)+\psi(z)$. Then $d$ is chosen small enough
and $K_{0}$ large enough such that, in particular:
\begin{itemize}
\item $D=\{r(z)<0\}\subset W(P_{0})$ and $r$ is a defining function of
$D$;
\item $D$ have a $\MR C^{\infty}$ boundary and is pseudo-convex;
\item At each point of $\partial\Omega\setminus\partial D$, the boundary
of $D$ is strictly pseudo-convex;
\item In the closure of $B(0,2\mu)$ the vector $z$ (in the coordinate
system centered at $0$) is not tangent to $\rho$ (i.e. $\sum_{i=1}^{n}\frac{\partial\rho}{\partial z_{i}}z_{i}\neq0$
everywhere in the closure of $B(0,2\mu)$).
\end{itemize}
\end{sddefn}

The fact that such a domain always exits for any $d>0$ small and
$K_{0}>0$ large is based on the construction of R. Gay and A. Sebbar
in \cite{Gay-Sebbar-division-extension} (Th\'eor\`eme 2.1). Simply, note
that, on $\partial D\setminus\partial\Omega$, the function $r$ is
strictly pluri-subharmonic if $K_{0}$ is large enough and $\mu$
small enough (the hessian of $\rho$ is $\mathrm{O}\left(\varphi\left(\left|z\right|^{2}\right)\right)$).
Moreover, if $P_{0}$ is of finite type, then all the boundary points
of $D$ are of finite type because the order of contact of $\partial\Omega$
with $\partial D$ is infinite at the points of $\partial\left(\partial\Omega\cap\partial D\right)$.

The goal of this Section is to prove the following:
\begin{stthm}
\label{thm:existence-ext-basis-local-domain}Suppose that $P_{0}$
is a point of finite $1$-type $\tau$ of $\partial\Omega$ and choose
$M'(\tau)$ (c. f. \propref{finite-type-imply-large-F(L-p-delta)}).
Let $\delta>0$ and $K>0$. If at every point of $\partial\Omega\cap V(P_{0})$
there is a $(K,p,\delta)$-extremal basis then one can construct
the domain $D$ contained in $V(P_{0})$ so that, at every point $p'$
of its boundary there exists a $(K',p',\delta)$-extremal basis with
$K'$ depending only on $K$ and the data.
\end{stthm}

The proof of this theorem is done in the two following sections.

\subsubsection{Preliminary remarks\label{sec:Preliminary-remarks-extremal-basis-local-domain-D}}

We fix now some general notations.

Let $\pi$ be the $\MR C^{\infty}$ projection of $V(P_{0})\cap\bar{\Omega}$
onto $\partial\Omega$ defined with the integral curves of the real
normal to $\rho$. It is clear that there exists a neighborhood
$\mathcal U$ of $\partial\Omega\cap V(p_{0})$ such that $\pi$ is a $\MR C^{\infty}$ diffeomorphism of
$\partial D\cap\mathcal U$ onto an open set of  $\partial\Omega\cap\mathcal U$.

If $L$ is a $\MR C^{\infty}$ vector field, defined on an open set
$U$ of $\partial D\cap\mathcal U$, tangent to $\partial D$, we associate
to it a vector field $L^{\rho}$, defined in the open set $\pi(U)\subset\partial\Omega$,
tangent to $\partial\Omega$ using $\pi$ (considered as a $\MR C^{\infty}$ diffeomorphism of $U$
onto $\pi(U)$) as follows: if $L=\sum a_{i}\frac{\partial}{\partial z_{i}}$,
considering it as an application of $U$ into $\mathbb{C}^{n}$, we
denote by $L\circ\pi^{-1}$ the vector field in $\pi(U)$ defined
by $L\circ\pi^{-1}=\sum a_{i}\circ\pi^{-1}\frac{\partial}{\partial z_{i}}$,
and
\begin{equation}
L^{\rho}=L\circ\pi^{-1}-\beta N,\label{eq:L-rho-associated-to-L}
\end{equation}
 where $N$ is the complex unitary normal to $\rho$ and $\beta=L\circ\pi^{-1}(\rho)$.

Clearly, $L\mapsto L^{\rho}$ is an isomorphism from $T_{\partial D\cap U}^{1,0}$
onto $T_{\partial\Omega\cap\pi(U)}^{1,0}$ ($V(P_{0})$ and $\mathcal U$ sufficiently
small), and thus, we also consider $L$ associated to $L^{\rho}$
by $L=L^{\rho}\circ\pi+\left(\beta\circ\pi\right)N\circ\pi$.
As $L$ is tangent to $\partial D$ and $(L^{\rho}\circ\pi)(\rho)$
is identically zero on $\partial\Omega$, we have
\begin{equation}
\beta\circ\pi(z)=\frac{-\left\langle L^{\rho}\circ\pi,z\right\rangle \varphi'(\left|z\right|^{2})}{(N\circ\pi)(\rho)+\left\langle N\circ\pi,z\right\rangle \varphi'(\left|z\right|^{2})}+k,\label{eq:formula-beta-composed-rho}
\end{equation}
 where $k$ is a $\MR C^{\infty}$ function whose derivatives of order
less than $M$ are $O(\varphi(\left|z\right|^{2}))$ with constants
controlled by the $\MR{C}^{2M}$ norm of $L$, and, if $L=\sum a_{i}\frac{\partial}{\partial z_{i}}$
(in the coordinate system of \defref{local-domain}), $\left\langle L,z\right\rangle $
denotes the usual scalar product $\sum a_{i}\overline{z_{i}}$, and
$\left\langle L,L'\right\rangle =\sum a_{i}a_{i}'$.

With the previous notations, let $P$ be a point of $\partial D$
such that $\psi(P)=0$ (thus $P\in\partial D\cap\partial\Omega$)
and let $V(P)\subset\mathcal U$ be a neighborhood of $P$ such that $\pi$ is a diffemorphism
of $V(P)\cap\partial D$ onto $V(P)\cap\partial\Omega$.

Let $p\in\partial D\cap V(P)$. Essentially, the construction of the
extremal basis $\brond$ at $p$ for $D$ is done using a suitable
basis $\brond^{\rho}$ of the tangent space of $\partial\Omega$ near
the point $\pi(p)$ translated at $p$ (using $\pi$) then projected
onto the tangent space of $\partial D$, to get a basis $\Btilde$
which will be used (in the next section) to define the basis $\brond$.

Currently, we only look at the relation between the weights of the basis
$\Btilde$ and $\brond^{\rho}$.

Thus, if $\Btilde=\{\Ltildep 1,\ldots,\Ltildep{n-1}\}$ is a basis
of $T_{\partial D}^{1,0}$ in $V(P)\cap\partial D$, with our notations,
the basis $\brond^{\rho}=\{L_{1}^{\rho},\ldots,L_{n-1}^{\rho}\}$
of $T_{\partial\Omega}^{1,0}$ (in $V(P)\cap\partial\Omega$) is given
by
\begin{equation}
L_{i}^{\rho}=\Ltildep i\circ\pi^{-1}-\beta_{i}N,\label{eq:L-i-up-to-rho}
\end{equation}
with $\beta_{i}=\Ltildep i\circ\pi^{-1}(\rho)$, and
\begin{equation}
\Ltildep i=L_{i}^{\rho}\circ\pi+(\beta_{i}\circ\pi)N\circ\pi.\label{eq:L-i-rho-down-to-r}
\end{equation}
 with
\begin{equation}
\beta_{i}\circ\pi=\frac{-\left\langle L_{i}^{\rho}\circ\pi,z\right\rangle \varphi'(\left|z\right|^{2})}{(N\circ\pi)(\rho)+\left\langle N\circ\pi,z\right\rangle \varphi'(\left|z\right|^{2})}+k.\label{eq:beta-i-using-L-i-rho-up}
\end{equation}

Let us calculate the weights $F(\Ltildep i,p,\delta)$ in terms of
the weights $F(L_{i}^{\rho},\pi(z),\delta)$ and the derivatives of
$\varphi$. We suppose that $L_{i}^{\rho}$ are normalized. Writing
$\widetilde{c_{ij}}=\left[\Ltildep i,\overline{\Ltildep j}\,\right](\partial r)$
and $c_{ij}^{\rho}=\left[L_{i}^{\rho},\overline{L_{j}^{\rho}}\right](\partial\rho)$,
using that $(N\circ\pi)(\rho)$ is identically $1$ on $\partial\Omega$,
a simple computations shows
\begin{equation}
\begin{array}{rcl}
\widetilde{c_{ij}} & = & c_{ij}^{\rho}\circ\pi+\left\langle L_{i}^{\rho}\circ\pi,L_{j}^{\rho}\circ\pi\right\rangle \varphi'(\left|z\right|^{2})+\left\langle L_{i}^{\rho}\circ\pi,z\right\rangle \overline{\left\langle L_{j}^{\rho}\circ\pi,z\right\rangle }\varphi''(\left|z\right|^{2})+\\
 &  & +\varphi'(\left|z\right|^{2})\sum_{k=1}^{n-1}\left(*\left\langle L_{k}^{\rho}\circ\pi,z\right\rangle +*\left\langle \overline{L_{k}^{\rho}}\circ\pi,z\right\rangle \right)+k,\\
 & = & c_{ij}^{\rho}\circ\pi+\varphi'(\left|z\right|^{2})\left(\left\langle L_{i}^{\rho}\circ\pi,L_{j}^{\rho}\circ\pi\right\rangle +h\right)+\left\langle L_{i}^{\rho}\circ\pi,z\right\rangle \overline{\left\langle L_{j}^{\rho}\circ\pi,z\right\rangle }\varphi''(\left|z\right|^{2})+k,
\end{array}\label{eq:calculus-cij}
\end{equation}
 where all the derivatives of $k$ are $\mathrm{O}(\varphi(\left|z\right|^{2}))$
and the functions $*$ have a bounded $\MR{C}^{M}$ norm, the constants
depending only on $\Omega$ and the $\MR{C}^{2M}$ norms of the $\widetilde{L_{i}}$.

As $L_{i}^{\rho}$ are normalized, we also have
\begin{equation}
\begin{array}{rcl}
\widetilde{c_{ii}} & = & c_{ii}^{\rho}\circ\pi+\varphi'(\left|z\right|^{2})+\left|\left\langle L_{i}^{\rho}\circ\pi,z\right\rangle \right|^{2}\varphi''(\left|z\right|^{2})+\varphi'(\left|z\right|^{2})\sum_{k=1}^{n-1}\left(*\left\langle L_{k}^{\rho}\circ\pi,z\right\rangle +*\left\langle \overline{L_{k}^{\rho}}\circ\pi,z\right\rangle \right)+k\\
 & = & c_{ii}\circ\pi+\varphi'(\left|z\right|^{2})(1+h)+\left|\left\langle L_{i}^{\rho}\circ\pi,z\right\rangle \right|^{2}\varphi''(\left|z\right|^{2})+k
\end{array}\label{eq:c-ii-tilde-versus-c-ii}
\end{equation}
and $d$ is chosen small enough such that the $\MR{C}^{M}$ norm of
$h$ is small.

Now we need to introduce a new notation. Let $L$ be a $\MR C^{\infty}(\partial D\cap V(P))$
vector field tangent to $\partial D$. For $z\in\partial D\cap V(P)$
let us define
\[
\widetilde{F^{\varphi}}(L,z,\delta)=\sum_{k=1}^{M/2}\left(\frac{\varphi^{(k)}(\left|z\right|^{2})}{\delta}\right)^{1/k}+\left|\left\langle L^{\rho}\circ\pi(z),z\right\rangle \right|^{2}\sum_{2}^{M}\left|\frac{\varphi^{(k)}(\left|z\right|^{2})}{\delta}\right|^{2/k}+\delta^{-1/M}.
\]

\begin{sllem}
\label{lem:evaluation-F-tilde-phi}For $\delta$ and $V(P)$
small enough and for $z\in\partial D\cap V(P)$, we have:
\[
\widetilde{F^{\varphi}}(L,z,\delta)\simeq\frac{\varphi'(\left|z\right|^{2})}{\delta}+\left|\left\langle L^{\rho}\circ\pi,z\right\rangle \right|^{2}\frac{\varphi''(\left|z\right|^{2})}{\delta}+\delta^{-1/M}.
\]
\end{sllem}
\begin{proof}
It suffices to consider the case when $\left|z\right|^{2}=\mu^{2}+x>\mu^{2}$.
Note that, for $V(P)$ small, $\varphi^{(k)}(\mu^{2}+x)\simeq Ke^{-1/x}x^{-2k}$
and $\left(\frac{1}{x}\right)^{2k}\leq e^{1/Mx}$, for $k\leq M$.

Suppose $\left(\frac{\varphi^{(k)}(\mu^{2}+x)}{\delta}\right)^{1/k}>\delta^{-1/M}$
and $e^{-1/x}<\delta$. Then
\[
\left(\frac{\varphi^{(k)}(\mu^{2}+x)}{\delta}\right)^{1/k}\simeq\left(\frac{K_{0}e^{-1/x}}{\delta}\right)^{1/k}\frac{1}{x^{2}}\lesssim K_{0}^{1/2}\delta^{-1/kM}\leq\delta^{-1/M},
\]
 for $\delta$ small. Thus, for $\delta\leq\delta_{0}(K_{0})$, $\left(\frac{\varphi^{(k)}(\mu^{2}+x)}{\delta}\right)^{1/k}>\delta^{-1/M}$
implies $e^{-1/x}>\delta$ and $\sum_{1}^{M/2}\left(\frac{\varphi^{(k)}(\mu^{2}+x)}{\delta}\right)^{1/k}\simeq\frac{\varphi'(\mu^{2}+x)}{\delta}$.

Similarly, $\left(\frac{\varphi^{(k)}(\mu^{2}+x)}{\delta}\right)^{2/k}>\delta^{-1/M}$
implies $e^{-1/x}>\delta$ and $\sum_{2}^{M}\left(\frac{\varphi^{(k)}(\mu^{2}+x)}{\delta}\right)^{2/k}\simeq\frac{\varphi''(\mu^{2}+x)}{\delta}$.
\end{proof}

Thus, we denote
\[
F^{\varphi}(L,z,\delta)=\frac{\varphi'(\left|z\right|^{2})}{\delta}+\left|\left\langle L^{\rho}\circ\pi,z\right\rangle \right|^{2}\frac{\varphi''(\left|z\right|^{2})}{\delta}+\delta^{-1/M},
\]
$F_{i}^{\varphi}=F_{i}^{\varphi}(z,\delta)=F^{\varphi}(\widetilde{L_{i}},z,\delta)$,
$1\leq i\leq n-1$ and $F_{n}^{\varphi}=\delta^{-2}$. Let
$\widetilde{L_{n}}$ denotes the unitary complex normal to $r$ (the
defining function of $D$) and $L_{n}^{\rho}$ the unitary complex
normal to $\rho$.
\begin{spprop}
\label{prop:comparison-lists-D-lists-up-5.1}Let $\Lrondt$ be a list
of $\Lrond_{M}\left(\widetilde{\brond}\cup\left\{ \widetilde{L_{n}}\right\} \right)$
and $\Lrond^{\rho}$ be the list obtained replacing $\LTB{i}$ in
$\Lrond$ by $\LBB{i}{\rho}$. Then, reducing $V(P)$ if necessary,
on $\partial D\cap V(P)$ we have ($\tilde{l}_{i}$ denoting the number
of times the vector fields $\Ltildep i$ or $\overline{\Ltildep i}$
appears in $\Lrondt$):
\begin{enumerate}
\item $\left|\Lrondt(c_{ij}^{\rho}\circ\pi)-(\Lrond^{\rho}c_{ij}^{\rho})\circ\pi\right|\lesssim\delta\prod_{k=1}^{n}\left(F_{k}^{\varphi}\right)^{\tilde{l}_{k}/2}$,
for $\left|\Lrondt\right|\geq2$,
\item $\left|\Lrondt\varphi(\left|z\right|^{2})\right|\lesssim\delta\prod_{i=1}^{n}\left(F_{i}^{\varphi}\right)^{\tilde{l}_{i}/2}$,
$\left|\widetilde{\Lrond}\right|\geq2$,
\end{enumerate}
the constants depending only on $\Omega$ and the $\MR{C}^{M+2}$
norms of the $\widetilde{L_{i}}$.\end{spprop}
\begin{proof}
These properties are trivially satisfied if $\tilde{l_{n}}\neq0$,
thus we suppose $\tilde{l}_{n}=0$. Using (\ref{eq:calculus-cij})
and the fact that if $f$ is a $\MR{C}^{\infty}$ function on $\partial\Omega\cap V(P)$
and if $L^{\rho}\rho\equiv0$ then
$\left(L^{\rho}\circ\pi\right)\left(f\circ\pi\right)-\left(L^{\rho}f\right)\circ\pi=O_{f}(\varphi)$
on $\partial D\cap V(P)$, the Proposition is an easy consequence
of (\ref{eq:beta-i-using-L-i-rho-up}) and the following Lemma:
\begin{sllem}
\label{lem:list-psi}Let $\Lrond^{\rho,\pi}$ be a list of
$\Lrond_{M}\left\{ L_{i}^{\rho}\circ\pi,\, i\leq n-1\right\}$
of length
$\geq1$. Then $\left|\Lrond^{\rho,\pi}\psi\right|\lesssim\delta\prod_{i=1}^{n-1}\left(F_{i}^{\varphi}\right)^{l_{i}/2}$.
\end{sllem}
\begin{proof}
[Proof of \lemref{list-psi}]By induction, we have
\[
\Lrond^{\rho,\pi}\psi=\Lrond^{\rho,\pi}\left(\varphi\left(\left|z\right|^{2}\right)\right)=\sum_{l=1}^{\left[\frac{m-1}{2}\right]}*\varphi^{(l)}\left(\left|z\right|^{2}\right)+\sum_{l=\left[\frac{m+1}{2}\right]}^{m}\alpha_{l}\varphi^{(l)}\left(\left|z\right|^{2}\right),
\]
 with 
\[
\alpha_{m-k}=\sum_{\SU{\Lrond^{*}=\{W_{1}^{*},\ldots,W_{m^{*}}^{*}\}\subset\Lrond^{\rho,\pi}\AS m^{*}\leq m-2k}}*\prod_{W_{i}^{*}\in\Lrond^{*}}\left\langle W_{i}^{*},\PB{z}\right\rangle ,
\]
where $\left\langle W_{i}^{*},\PB{z}\right\rangle $ denotes $\left\langle W_{i}^{*},z\right\rangle $
if $W_{i}^{*}$ is of type $(0,1)$ and $\left\langle W_{i}^{*},\bar{z}\right\rangle $
if not, and the $\MR{C}^{2M}$ norms of the functions $*$ are controlled
by the $\MR{C}^{2M}$ norms of the vector fields $\widetilde{L_{i}}$.
Now, the proof of \lemref{evaluation-F-tilde-phi} shows that
\[
\sum_{l=1}^{\left[\frac{m-1}{2}\right]}\frac{*\varphi^{(l)}\left(\left|z\right|^{2}\right)}{\delta}\lesssim\left(\delta^{-1/M}+\frac{\varphi'\left(\left|z\right|^{2}\right)}{\delta}\right)^{m/2},
\]
and it is enough to see that $\left|\alpha_{l}\varphi^{(l)}\left(\left|z\right|^{2}\right)\right|\lesssim\delta\left(F^{\varphi}\right)^{\Lrond/2}$,
for $l\in\left\{ \left[\frac{m+1}{2}\right],\ldots,m\right\} $. If
$l=m$, this follows from \lemref{evaluation-F-tilde-phi}; suppose $l=m-k$,
$k\geq1$.

Suppose $\left|\frac{\varphi^{(m-k)}\left(\left|z\right|^{2}\right)}{\delta}\right|^{2/(m-k)}\geq\delta^{-1/M}$.
By \lemref{evaluation-F-tilde-phi} $\left|\frac{\varphi^{(m-k)}\left(\left|z\right|^{2}\right)}{\delta}\right|^{2/(m-k)}\leq\frac{\varphi''\left(\left|z\right|^{2}\right)}{\delta}$.
Let $\Lrond^{*}\subset\Lrond$ of length $m^{*}=m-2k=\sum_{i=1}^{n-1}l_{i}^{*}$.
The corresponding term in $\alpha_{m-k}$ is bounded by
\begin{eqnarray*}
*\left(\frac{\varphi''\left(\left|z\right|^{2}\right)}{\delta}\right)^{(m-k)/2}\prod_{i}\left|\left\langle L_{i}^{\rho}\circ\pi,z\right\rangle \right|^{l_{i}^{*}} & = & *\left(\frac{\varphi''\left(\left|z\right|^{2}\right)}{\delta}\right)^{k/2}\prod_{i=1}^{n-1}\left(\frac{\varphi''\left(\left|z\right|^{2}\right)}{\delta}\left|\left\langle L_{i}^{\rho}\circ\pi,z\right\rangle \right|^{2}\right)^{l_{i}^{*}/2}\\
 & \lesssim & *\left(\frac{\varphi'\left(\left|z\right|^{2}\right)}{\delta}\right)^{k}\prod_{i=1}^{n-1}\left(F_{i}^{\varphi}\right)^{l_{i}^{*}/2},
\end{eqnarray*}
 because the hypothesis implies $\left(\frac{\varphi''\left(\left|z\right|^{2}\right)}{\delta}\right)^{1/2}\lesssim\frac{\varphi'\left(\left|z\right|^{2}\right)}{\delta}$.
\end{proof}
To finish the proof of \propref{comparison-lists-D-lists-up-5.1}
note that, for $\left|\widetilde{\Lrond}\right|\geq1$,
\[
\left|\widetilde{\Lrond}\left(\beta_{i}\circ\pi\right)(z)\right|\leq F^{\varphi}(z,\delta)^{\widetilde{\Lrond}/2}F_{i}^{\varphi}(z,\delta)^{1/2},
\]
and use (\ref{eq:beta-i-using-L-i-rho-up}).
\end{proof}

Finally the relations between the weights associated to $\Btilde$
and to $\brond^{\rho}$are as follows.

Let $\Ltilde$ a holomorphic vector field on $\partial D$ tangent
to $\partial D$ near $p$ and $L^{\rho}$ the associated vector field
tangent to $\partial\Omega$. Then
\begin{spprop}
\label{prop:relation-lists-tilde-lists-rho}For $V(P)$ sufficiently
small, we have, if $\frac{1}{K}\leq\left\Vert \Ltilde\right\Vert \leq K$,
\[
F(\Ltilde,z,\delta)\simeq F(L^{\rho},\pi(z),\delta)+F^{\varphi}(\Ltilde,z,\delta),
\]
 with constants depending on the $\MR C^{2M}$ norm of $\Ltilde$,
$K$ and the data.\end{spprop}
\begin{proof}
From \propref{comparison-lists-D-lists-up-5.1} it easily follows
that $F(\Ltilde,z,\delta)\lesssim F(L^{\rho}\circ\pi,z,\delta)+F^{\varphi}(\Ltilde,z,\delta)$.
Let us then see that there exists a list $\Lrondt$ composed of $\widetilde{L}$
and $\overline{\Ltilde}$ such that $\Lrondt\widetilde{c_{\Ltilde\Ltilde}}\simeq\delta\left(F(L^{\rho}\circ\pi,z,\delta)+F^{\varphi}(\Ltilde,z,\delta)\right)^{(\left|\Lrondt\right|+2)/2}\stackrel{\mathrm{def}}{=}\delta F^{(\left|\Lrondt\right|+2)/2}$.
If $\frac{\varphi'}{\delta}+\left|\left\langle L^{\rho}\circ\pi(z),z\right\rangle \right|^{2}\frac{\varphi''}{\delta}\simeq F$, then
$c_{\widetilde{L}\widetilde{L}}$ do it. Suppose $\frac{\varphi'}{\delta}+\left|\left\langle L^{\rho}\circ\pi(z),z\right\rangle \right|^{2}\frac{\varphi''}{\delta}\ll F$.
Then, there exists a list $\Lrond^{\rho}$ such that $\left|\Lrond^{\rho}c_{L^{\rho}L^{\rho}}(\pi(z))\right|\simeq\delta F^{(\left|\widetilde{\Lrond}\right|+2)/2}$.
Then calculating $\Lrondt\widetilde{c_{\Ltilde\Ltilde}}$ in term
of $\Lrond^{\rho}(c_{L^{\rho}L^{\rho}})\circ\pi$, the result follows
\propref{comparison-lists-D-lists-up-5.1}, (\ref{eq:c-ii-tilde-versus-c-ii})
and the properties of the functions $h$ and $k$.

\bigskip{}

\end{proof}

\subsubsection{Extremal bases on $D$\label{sec:Extremal-basis-on-local-domain-D} }

In this Section, we assume that $p_{0}$ is of finite type
$\tau$, $M=M'(\tau)$ and that, at all points $q$ of $V(P_{0})\cap\partial\Omega$
and for all $\delta>0$, $0<\delta\leq\delta_{0}$, there exists a
$(K,q,\delta)$-extremal basis. Then we will show that at all points $p$
of $\partial D$ and for all $\delta>0$ there exists a $(K',p,\delta)$-extremal
basis (for $D$) with a constant $K'$ controlled by $K$ and the
data.

\smallskip

If $P$ is a point of $\partial D$ such that $\left|P\right|>\mu$
then $\partial D$ is strictly pseudo-convex near $P$ and the construction
of extremal basis in $V(P)\cap\partial D$ is trivial (for $V(P)$
small). If $\left|P\right|<\tau$ then $V(P)\cap\partial D$ is contained
in $\partial\Omega$ and the existence of extremal basis is the hypothesis.
Thus, we have only to consider neighborhood of points $P\in\partial D$
such that $\left|P\right|=\mu$ (that is points $P$ in the boundary
of $\partial\Omega\cap\partial D$).

As we said before, the final extremal basis for $D$, at $p\in V(P)\cap\partial D$,
will be obtained extending a basis $\widetilde{\brond}$ defined on
$V(P)\cap\partial D$ which is a projection onto the tangent space
to $r$ of a translation of a basis $\brond^{\rho}$, at $\pi(p)$,
tangent to $\rho$.

Formula (\ref{eq:c-ii-tilde-versus-c-ii}) shows that the expressions
$\left\langle L_{i}^{\rho}\circ\pi,z\right\rangle $ plays an important
role: we have to take into account the vector fields which are orthogonal
to $z$. In particular, to construct an extremal basis on $\partial D$,
we cannot simply translate an extremal basis on $\partial\Omega$
and project it onto the tangent space to $\partial D$, because, even
if the basis $(L_{i}^{\rho})$ is extremal, we may have $\left\langle L_{i}^{\rho}\circ\pi,z\right\rangle \neq0$,
for all $i$, and there are linear combinations of the $L_{i}^{\rho}\circ\pi$
which are orthogonal to $z$.

From now the point $p$ and the positive number $\delta$ are fixed. We
suppose we have a $(K,\pi(p),\delta)$-extremal (for $\rho$) basis
$\brond^{\Omega}=\{L_{1}^{\Omega},\ldots,L_{n-1}^{\Omega}\}$ at the
point $\pi(p)$ (the $L_{i}^{\Omega}$ being $\MR{C}^{\infty}$ in
$V(P)$), such that the vectors $L_{i}^{\Omega}(\pi(p))$ are orthogonal
(c.f. \propref{extremality-and-normalization}) and we construct the
basis $\brond^{\rho}=\{L_{1}^{\rho},\ldots,L_{n-1}^{\rho}\}$ using
it. The weights associated to $\brond^{\Omega}$ are denoted $F_{i}^{\Omega}=F_{i}^{\Omega}(\pi(p),\delta)=F^{\Omega}(L_{i}^{\Omega},\pi(p),\delta)$,
and we suppose $F_{i+1}^{\Omega}\leq F_{i}^{\Omega}$, for $i\leq n-2$,
changing the order of $L_{i}^{\Omega}$ if necessary.

Recall that the canonical coordinate system is centered at the point
$O$ of \defref{local-domain}, thus $\left|z(P)\right|=\mu$.

For simplicity of notations, we denote $q=\pi(p)$ (thus $p=\pi^{-1}(q)$,
$\pi$ being considered as a diffeomorphism between open sets of the
boundaries of $\Omega$ and $D$).

Let
\[
\Hrond_{n-1}=\left\{ W=\sum a_{i}L_{i}^{\Omega},\, a_{i}\in\mathbb{C},\,\sum\left|a_{i}\right|^{2}=1,\mbox{ such that }\left\langle W(q),p\right\rangle =0\right\} .
\]
 Let $W_{n-1}=\sum a_{i}^{n-1}L_{i}^{\Omega}\in\Hrond_{n-1}$ such
that $\sum_{i=1}^{n-2}\left|a_{i}^{n-1}\right|^{2}F(L_{i}^{\Omega},q,\delta)=\inf_{W=\sum a_{i}L_{i}^{\Omega}\in\Hrond_{n-1}}\sum_{i=1}^{n-2}\left|a_{i}\right|^{2}F(L_{i}^{\Omega},q,\delta)$,
and define
\begin{equation}
L_{n-1}^{\rho}=\left\{ \begin{array}{ll}
L_{n-1}^{\Omega} & \mbox{if }\sum_{i=1}^{n-2}\left|a_{i}^{n-1}\right|^{2}F(L_{i}^{\Omega},q,\delta)\geq\frac{\varphi''(\left|p\right|^{2})}{\delta}\left|\left\langle L_{n-1}^{\Omega}(q),p\right\rangle \right|^{2},\\
W_{n-1} & \mbox{otherwise}.
\end{array}\right.\label{eq:def-L-rho-n-1}
\end{equation}

Suppose $L_{n-l}^{\rho}$ are defined for $1\leq l\leq k-1<n$. Let $\mathcal{H}_{n-k}=\mathcal{H}_{n-1}\cap\left[\MR{E}(L_{n-1}^{\rho},\ldots,L_{n-k+1}^{\rho})\right]^{\perp}$,
$\MR{E}(L_{n-1}^{\rho},\ldots,L_{n-k+1}^{\rho})$ being the linear
space spaned by $L_{n-1}^{\rho}$, ..., $L_{n-k+1}^{\rho}$, the orthogonality
being taken at $q$. Let $W_{n-k}=\sum_{i=1}^{n-1}a_{i}^{n-k}L_{i}^{\Omega}$
a vector in $\mathcal{H}_{n-k}$ minimizing $\sum_{i=1}^{n-k-1}\left|a_{i}\right|^{2}F(L_{i}^{\Omega},q,\delta)$
for vectors $\sum_{i=1}^{n-1}a_{i}L_{i}^{\Omega}\in\mathcal{H}_{n-k}$.
Let $T_{n-k}$ be a vector field, of norm $1$ at $q$, in $\MR{G}_{n-k}=\MR{E}(L_{n-1}^{\Omega},\ldots,L_{n-k}^{\Omega})\cap\left[\MR{E}(L_{n-1}^{\rho},\ldots,L_{n-k+1}^{\rho})\right]^{\perp}$.
Then $L_{n-k}^{\rho}$is defined by
\[
L_{n-k}^{\rho}=\left\{ \begin{array}{ll}
T_{n-k} & \mbox{if }\sum_{i=1}^{n-k-1}\left|a_{i}^{n-k}\right|^{2}F(L_{i}^{\Omega},q,\delta)\geq\frac{\varphi''(\left|p\right|^{2})}{\delta}\left|\left\langle T_{n-k}(q),p\right\rangle \right|^{2},\\
W_{n-k} & \mbox{otherwise.}
\end{array}\right.
\]

Note that $\{L_{i}^{\rho}(q),\,1\leq i\leq n-1\}$ is orthonormal.
We will note later that if the dimension of $\MR{G}_{n-k}$ is strictly
greater than $1$ then $F^{\rho\varphi}(L_{n-k}^{\rho})$ (see below)
is, up to a multiplicative constant, independent of the choice of
$T_{n-k}$.

The next two Lemmas prove some important properties of the vector fields
$L_{i}^{\rho}$. Let us denote $\brond^{\rho}=\left\{ L_{i}^{\rho},\, i<n\right\} $
and $L_{n}^{\rho}$ the unitary complex normal to $\rho$.

For $L=\sum_{i=1}^{n-1}a_{i}L_{i}^{\rho}$, $a_{i}\in\mathbb{C}$,
let us denote
\[
F^{\rho\varphi}(L)=F(L,q,\delta)+\frac{\varphi'\left(\left|p\right|^{2}\right)}{\delta}+\left|\left\langle L(q),p\right\rangle \right|^{2}\frac{\varphi''\left(\left|p\right|^{2}\right)}{\delta},
\]
$F_{i}^{\rho\varphi}=F^{\rho\varphi}(L_{i}^{\rho})$, $1\leq i\leq n-1$,
$F_{n}^{\rho\varphi}=\frac{1}{\delta^{2}}$ and $\left(F^{\rho\varphi}\right)^{\Lrond/2}=\prod_{i}\left(F_{i}^{\rho\varphi}\right)^{l_{i}/2}$,
if $\Lrond$ is a list of $\Lrond_{M}(\brond^{\rho}\cup\left\{ L_{n}^{\rho}\right\} )$,
with the usual notation for $l_{i}$.

We will show that, up to constants, the vector fields $L_{i}^{\rho}$
give the successive minima of the functions $F^{\rho\varphi}(L)$
for $L=\sum a_{i}L_{i}^{\Omega}$, $\sum\left|a_{i}^{2}\right|=1$.
\begin{sllem}
\label{lem:comparison-F-rho-phi-F-Omega-5.3}There exits a constant
$K'$ depending only on $K$ such that:
\begin{enumerate}
\item If $L=\sum a_{i}L_{i}^{\Omega}$, $\sum\left|a_{i}\right|^{2}=1$,
is orthogonal, at $q$, to the space generated by $L_{j}^{\rho}$,
$i+1\leq j\leq n-1$, $i\leq n-1$, then $F^{\rho\varphi}(L)\geq\frac{1}{K'}F^{\rho\varphi}(L_{i}^{\rho})$;
\item $F^{\rho\varphi}(L_{i}^{\rho})\geq\frac{1}{K'}F^{\rho\varphi}(L_{i+1})$,
$i<n-1$;
\item $F^{\rho\varphi}(L_{i}^{\rho})\geq\frac{1}{K'}F(L_{i}^{\Omega},q,\delta)$,
$i<n$.
\end{enumerate}
\end{sllem}
\begin{proof}
Note first that if property (2) is satisfied for $i\geq k$ then property
(3) is also satisfied for $i\geq k$. Indeed, more generally, if $L$
is orthogonal to the vectors $L_{j}^{\rho}$, $i+1\leq j\leq n-1$,
and if property (2) is satisfied for $i+1,\ldots,n-1$, then
\begin{equation}
F^{\rho\varphi}(L)\gtrsim\max\left\{ F^{\Omega}(L),F^{\Omega}(L_{i+1}^{\rho})\ldots,F^{\Omega}(L_{n-1}^{\rho})\right\} \gtrsim F^{\Omega}(L_{i}^{\Omega},q,\delta)=F(L_{i}^{\Omega},q,\delta)=F_{i}^{\Omega},\label{eq:comparison-weights-ext-basis-local}
\end{equation}
 because the $L_{j}^{\rho}$ and $L$ are orthogonal and the basis
$\left(L_{i}^{\Omega}\right)_{i}$ is extremal.

Now we show that if $L=\sum a_{i}L_{i}^{\Omega}$, $\sum\left|a_{i}\right|^{2}=1$,
then $F^{\rho\varphi}(L)\gtrsim F_{n-1}^{\rho\varphi}$.

If $L_{n-1}^{\Omega}\in\mathcal{H}_{n-1}$, then $L_{n-1}^{\rho}=L_{n-1}^{\Omega}$
and $F_{n-1}^{\rho\varphi}=F(L_{n-1}^{\Omega},q,\delta)+\frac{\varphi'\left(\left|p\right|^{2}\right)}{\delta}$
which gives the result. Suppose thus $L_{n-1}^{\Omega}\notin\mathcal{H}_{n-1}$.
We separate the two cases of (\ref{eq:def-L-rho-n-1}):

Suppose we are in the first case ($L_{n-1}^{\rho}=L_{n-1}^{\Omega}$).
If $L\in\mathcal{H}_{n-1}$, then the inequality is an immediate consequence
of the extremality (EB$_{\text{1}}$) of $\brond^{\Omega}$. Suppose
$L\notin\mathcal{H}_{n-1}$. Then we can write $L=\alpha\left(L_{n-1}^{\rho}+\gamma H\right)$
with $H\in\mathcal{H}_{n-1}$. Writing $H=\sum a'_{i}L_{i}^{\Omega}$,
we have
\[
F^{\rho\varphi}(L)\simeq\left|\alpha\right|^{2}\left[\sum_{i=1}^{n-2}\left|\gamma a'_{i}\right|^{2}F(L_{i}^{\Omega},q,\delta)+\left|1+\gamma a'_{n-1}\right|^{2}F(L_{n-1}^{\Omega},q,\delta)\right]+\frac{\varphi'\left(\left|p\right|^{2}\right)}{\delta}+\left|\alpha\right|^{2}\frac{\varphi''\left(\left|p\right|^{2}\right)}{\delta}\left|\left\langle L_{n-1}^{\Omega}(q),p\right\rangle \right|^{2},
\]
and as $\sum_{i=1}^{n-2}\left|a'_{i}\right|^{2}F(L_{i}^{\Omega},q,\delta)\geq\sum_{i=1}^{n-2}\left|a_{i}^{n-1}\right|^{2}F(L_{i}^{\Omega},q,\delta)\geq\frac{\varphi''\left(\left|p\right|^{2}\right)}{\delta}\left|\left\langle L_{n-1}^{\Omega}(q),p\right\rangle \right|^{2}$,
we obtain
\[
F^{\rho\varphi}(L)\gtrsim\left|\alpha\right|^{2}\left(1+\left|\gamma\right|^{2}\right)\frac{\varphi''\left(\left|p\right|^{2}\right)}{\delta}\left|\left\langle L_{n-1}^{\Omega}(q),p\right\rangle \right|^{2}\gtrsim_{K}\frac{\varphi''\left(\left|p\right|^{2}\right)}{\delta}\left|\left\langle L_{n-1}^{\Omega}(q),p\right\rangle \right|^{2},
\]
because, by equivalence of norms in finite dimensional spaces, $\left|\alpha\right|^{2}\left(1+\left|\gamma\right|^{2}\right)\gtrsim_{K}1$.
The extremality of $\brond^{\Omega}$ implies $F(L,q,\delta)\gtrsim F(L_{n-1}^{\Omega},q,\delta)$,
and the inequality is proved.

Let us now look to the case $L_{n-1}^{\rho}=W_{n-1}$.
The result is trivial if $L\in\MR{H}_{n-1}$, thus we suppose $L\notin\MR{H}_{n-1}$. Using
the same decomposition as before, we get
\[
F^{\rho\varphi}(L)\gtrsim\left|\alpha\right|^{2}\left(1+\left|\gamma\right|^{2}\right)\sum_{i=1}^{n-2}\left|a_{i}^{n-1}\right|^{2}F(L_{i}^{\Omega},q,\delta)\gtrsim\sum_{i=1}^{n-2}\left|a_{i}^{n-1}\right|^{2}F(L_{i}^{\Omega},q,\delta),
\]
 and, as $F^{\rho\varphi}(L)\gtrsim F(L,q,\delta)\gtrsim F(L_{n-1}^{\Omega},q,\delta)$,
we have $F^{\rho\varphi}(L)\gtrsim F(W_{n-1},q,\delta)$.

The induction is as follows. Suppose $F^{\rho\varphi}(L_{i+1}^{\rho})\gtrsim F^{\rho\varphi}(L_{i+2}^{\rho})\gtrsim\ldots\gtrsim F^{\rho\varphi}(L_{n-1}^{\rho})$
and that for all $L=\sum a_{i}L_{i}^{\Omega}$, $\sum\left|a_{i}\right|^{2}=1$,
orthogonal to the $L_{j}^{\rho}$, $i+2\leq j\leq n-1$, then $F^{\rho\varphi}(L)\gtrsim F^{\rho\varphi}(L_{i+1}^{\rho})$.
Let $L=\sum a_{i}L_{i}^{\Omega}$, $\sum\left|a_{i}\right|^{2}=1$,
orthogonal to the $L_{j}^{\rho}$, $i+1\leq j\leq n-1$. Suppose $T_{i}$
is chosen. If $T_{i}\in\MR{H}_{i}$ then $L_{i}^{\rho}=T_{i}$, $F^{\rho\varphi}(T_{i})\lesssim F_{i}^{\Omega}+\frac{\varphi'}{\delta}$
and, using (\ref{eq:comparison-weights-ext-basis-local}), $F^{\rho\varphi}(L)\gtrsim F^{\rho\varphi}(L_{i}^{\rho})$.
Suppose now $T_{i}\notin\MR{H}_{i}$. If $L_{i}^{\rho}=T_{i}$ then, decomposing
$L=\alpha\left(T_{i}+\gamma H\right)$ as in the first step, we obtain
$F^{\rho\varphi}(L)\gtrsim\frac{\varphi''}{\delta}\left|\left\langle T_{i}(q),p\right\rangle \right|^{2}$
and we use (\ref{eq:comparison-weights-ext-basis-local}). Otherwise
$L_{i}^{\rho}=W_{i}$ and, again, the same decomposition gives
$F^{\rho\varphi}(L)\gtrsim\sum_{j=1}^{n-i-1}\left|a_{i}\right|^{2}F(L_{i}^{\Omega},q,\delta)$
and we conclude with (\ref{eq:comparison-weights-ext-basis-local}).

Finally we obtain $F^{\rho\varphi}(L)\gtrsim F^{\rho\varphi}(L_{i}^{\rho})$
(which proves the statement about the choice of $T_{n-k}$), and,
as $L_{i}^{\rho}$ is orthogonal to $L_{i+1}^{\rho},\ldots,L_{n-1}^{\rho}$,
the induction hypothesis imply $F^{\rho\varphi}(L_{i}^{\rho})\gtrsim F^{\rho\varphi}(L_{i+1}^{\rho})$
and finishes the proof.
\end{proof}

We now estimate the brackets of the vector fields $L_{i}^{\rho}$,
$i<n$, at the point $q$.
\begin{sllem}
\label{lem:brackets-L-rho-5.4}Let $\left[\LBB{k}{\rho},\LBB{s}{\rho}\right]=\sum_{t=1}^{n}b_{\PB{k}\PB{s}}^{t}L_{t}^{\rho}+\sum_{t=1}^{n}b_{\PB{k}\PB{s}}^{\bar{t}}\Lbarp t{\rho}$.
For all lists $\Lrond$, of $\Lrond_{M}\left(\brond^{\rho}\cup\left\{ L_{n}^{\rho}\right\} \right)$,
we have 
\[
\left|\Lrond\left(b_{\PB{k}\PB{s}}^{\PB{t}}\right)(q)\right|<K'\left(F^{\rho\varphi}\right)^{\Lrond/2}\left(F_{k}^{\rho\varphi}\right)^{1/2}\left(F_{s}^{\rho\varphi}\right)^{1/2}\left(F_{t}^{\rho\varphi}\right)^{-1/2}
\]
with $K'$ depending only on $K$ and the data.\end{sllem}
\begin{proof}
Note that the Lemma is trivial if $l_{n}(\Lrond)\geq1$ and if $F_{t}^{\rho\varphi}\lesssim\frac{\varphi''(\left|p\right|^{2})}{\delta}$
(because $F_{k}^{\rho\varphi}$ and $F_{s}^{\rho\varphi}$ are both
$\geq$ to $\frac{\varphi'(\left|p\right|^{2})}{\delta}$ and $\frac{\varphi''(\left|p\right|^{2})}{\delta}\geq\delta^{-2/M}$
implies $\left|\frac{\varphi'(\left|p\right|^{2})}{\delta}\right|^{2}\geq\left|\frac{\varphi''(\left|p\right|^{2})}{\delta}\right|$).
Moreover, we also have $F_{t}^{\rho\varphi}\lesssim F(L_{t}^{\rho},q,\delta)+\frac{\varphi''(\left|p\right|^{2})}{\delta}$,
and, if $L_{t}^{\rho}=T_{t}$, then, by the definition of $T_{t}$ and
the extremality of $\brond^{\Omega}$, $F(L_{t}^{\rho},q,\delta)\lesssim F(L_{t}^{\Omega},q,\delta)$,
and, if $L_{t}=W_{t}$, then $F(L_{t}^{\rho},q,\delta)\lesssim F(L_{t}^{\Omega},q,\delta)+\frac{\varphi''(\left|q\right|^{2})}{\delta}$.

Thus, it suffices to prove that if $l_{n}=0$
\[
\left|\Lrond\left(b_{\PB{k}\PB{s}}^{t}\right)(q)\right|\lesssim\left(F^{\rho\varphi}\right)^{\Lrond/2}\left(F_{k}^{\rho\varphi}\right)^{1/2}\left(F_{s}^{\rho\varphi}\right)^{1/2}\left(F(L_{t}^{\Omega},q,\delta)\right)^{-1/2}.
\]

Let us write $L_{k}^{\rho}=\sum\alpha_{k}^{i}L_{i}^{\Omega}$ and
$L_{k}^{\Omega}=\sum\beta_{k}^{i}L_{i}^{\rho}$. Using the notation
$\left[\LBB{i}{\Omega},\LBB{j}{\Omega}\right]=\sum_{i=1}^{n}a_{\PB{i}\PB{j}}^{m}L_{m}^{\Omega}+\sum_{i=1}^{n}a_{\PB{i}\PB{j}}^{\bar{m}}\Lbarp t{\Omega}$,
a computation gives, if $t<n$,
\[
b_{\PB{k}\PB{s}}^{t}=\sum_{m}\left(\sum_{i,j}\PBG{\alpha_{k}^{i}}\PBG{\alpha_{s}^{j}}a_{\PB{i}\PB{j}}^{m}\right)\beta_{m}^{t},
\]
 with $\beta_{m}^{t}=\frac{1}{\det(\alpha)}\sum_{\sigma}\varepsilon_{\sigma}\prod_{i}\alpha_{i}^{\sigma(i)}$,
where $\sigma$ describes the set of permutations from $\{1,\ldots,n-1\}\setminus\{t\}$
onto $\{1,\ldots,n-1\}\setminus\{m\}$, and
\[
b_{\PB{k}\PB{s}}^{n}=\sum_{i,j}\PBG{\alpha_{k}^{i}}\PBG{\alpha_{s}^{j}}C_{\PB{i}\PB{j}}
\]
 with $C_{\PB{i}\PB{j}}=\left[\LBB{i}{\Omega},\LBB{j}{\Omega}\right](\partial\rho)$
(note that this notation gives $c_{ij}=C_{i\bar{j}}$).

First, we prove that, if $t<m$, then $\left|\beta_{m}^{t}\right|\lesssim\left(F_{k}^{\rho\varphi}\right)^{1/2}\left(F_{s}^{\rho\varphi}\right)^{1/2}\left(F(L_{t}^{\Omega},q,\delta)\right)^{-1/2}$
for any $k$ and $s$. In that case, there exists an index $i>t$
such that $\sigma(i)\leq t$; if $L_{i}^{\rho}=T_{i}$ then $\alpha_{i}^{\sigma(i)}=0$,
and if $L_{i}^{\rho}=W_{i}$ then 
\[
\left|\alpha_{i}^{\sigma(i)}\right|\leq\left[\frac{\varphi''(\left|p\right|^{2})}{\delta}\left(F(L_{\sigma(i)}^{\Omega},q,\delta)\right)^{-1}\right]^{1/2}\leq\left(\frac{\varphi''(\left|p\right|^{2})}{\delta}\right)^{1/2}\left(F(L_{t}^{\Omega},q,\delta)\right)^{-1/2}\leq\left(F_{k}^{\rho\varphi}\right)^{1/2}\left(F_{s}^{\rho\varphi}\right)^{1/2}\left(F(L_{t}^{\Omega},q,\delta)\right)^{-1/2},
\]
 because $F_{m}^{\rho\varphi}\gtrsim\delta^{-1/M}+\frac{\varphi'(\left|p\right|^{2})}{\delta}$
and $\frac{\varphi''(\left|p\right|^{2})}{\delta}\geq\delta^{-2/M}$
implies $\left(\frac{\varphi'(\left|p\right|^{2})}{\delta}\right)^{2}\geq\frac{\varphi(\left|p\right|^{2})}{\delta}$.

To finish the proof, it suffices to remark that the extremality of
$\brond^{\Omega}$ implies
\[
\left|\alpha_{k}^{i}\right|\lesssim F(L_{k}^{\rho},q,\delta)^{1/2}F(L_{i}^{\Omega},q,\delta)^{-1/2},
\]
 and
\begin{eqnarray*}
\left|\Lrond\left(a_{\PB{i}\PB{j}}^{m}\right)\right| & \lesssim & \prod F(L_{k}^{\Omega},q,\delta)^{l_{k}/2}F(L_{i}^{\Omega},q,\delta)^{1/2}F(L_{j}^{\Omega},q,\delta)^{1/2}F(L_{m}^{\Omega},q,\delta)^{-1/2}\\
 & \lesssim & \left(F^{\rho\varphi}\right)^{\Lrond/2}F(L_{i}^{\Omega},q,\delta)^{1/2}F(L_{j}^{\Omega},q,\delta)^{1/2}F(L_{t}^{\Omega},q,\delta)^{-1/2},
\end{eqnarray*}
by \lemref{comparison-F-rho-phi-F-Omega-5.3}, for $t\geq m$.
\end{proof}

Then, with the notations introduced before, we consider the basis
at $p$ (for $D$) 
\[
\Btilde=\{\Ltildep 1,\ldots,\Ltildep{n-1}\}\mbox{ with }\Ltildep i=\frac{1}{\left\Vert L_{i}^{\rho}\circ\pi\right\Vert }\left(L_{i}^{\rho}\circ\pi+(\beta_{i}\circ\pi)N^{\rho}\circ\pi\right).
\]

Note that \lemref{comparison-F-rho-phi-F-Omega-5.3} and \lemref{brackets-L-rho-5.4}
are proved for the vector fields $L_{i}^{\rho}$ but it is easy to
see that they are also valid for the vector fields $L_{i}^{\rho}/\left\Vert L_{i}^{\rho}\right\Vert $.

To simplify the notations, in the remainder of the proof, the vector
fields $\frac{L_{i}^{\rho}}{\left\Vert L_{i}^{\rho}\right\Vert }$
will be denoted by $L_{i}^{\rho}$, and the function $\frac{\beta_{i}}{\left\Vert L_{i}^{\rho}\right\Vert }$
will be denoted $\beta_{i}$ so that $\Ltildep i=\left(L_{i}^{\rho}\circ\pi+(\beta_{i}\circ\pi)N^{\rho}\circ\pi\right)$.
\begin{spprop}
The basis $\Btilde$ is $(K',p,\delta)$-extremal for a constant $K'$
depending only on $K$ and the data.\end{spprop}
\begin{proof}
We first prove condition EB$_{\text{1}}$, that is, if $\alpha_{i}$
are complex numbers then
\[
F\left(\sum_{i=1}^{n-1}\alpha_{i}\Ltildep i,p,\delta\right)\simeq\sum_{i=1}^{n-1}\left|\alpha_{i}\right|^{2}F\left(\Ltildep i,p,\delta\right).
\]
 By induction, it suffices to see that, for all $k$,
\[
F\left(\sum_{i=1}^{n-k}\alpha_{i}\Ltildep i,p,\delta\right)\simeq F\left(\sum_{i=1}^{n-k-1}\alpha_{i}\Ltildep i,p,\delta\right)+\left|\alpha_{n-k}\right|^{2}F\left(\Ltildep{n-k},p,\delta\right).
\]

To simplify notations we write $\widetilde{X}=\sum_{i=1}^{n-k-1}\alpha_{i}\Ltildep i$
and $X^{\rho}=\sum_{i=1}^{n-k-1}\alpha_{i}L_{i}^{\rho}$. By \propref{relation-lists-tilde-lists-rho},
we have to prove that
\begin{multline}
F(X^{\rho}+\alpha_{n-k}L_{n-k}^{\rho},q,\delta)+\frac{\varphi''(\left|p\right|^{2})}{\delta}\left|\left\langle \left(X^{\rho}+\alpha_{n-k}L_{n-k}^{\rho}\right)\circ\pi(p),p\right\rangle \right|+\frac{\varphi'(\left|p\right|^{2})}{\delta}\\
\simeq F(X^{\rho},q,\delta)+\left|\alpha_{n-k}\right|^{2}F(L_{n-k}^{\rho},q,\delta)+\frac{\varphi''(\left|p\right|^{2})}{\delta}\left(\left|\left\langle (X^{\rho}\circ\pi)(p),p\right\rangle \right|^{2}+\left|\alpha_{n-k}\right|^{2}\left|\left\langle (L_{n-k}^{\rho}\circ\pi)(p),p\right\rangle \right|^{2}\right)+\frac{\varphi'(\left|p\right|^{2})}{\delta}.\label{eq:proof-prop-B-tilde-extremal-A}
\end{multline}

Indeed, if $\beta(q)=\frac{\left\Vert \sum_{i=1}^{t}\alpha_{i}L_{i}^{\rho}\right\Vert (q)}{\left\Vert \sum_{i=1}^{t}\alpha_{i}L_{i}^{\rho}\right\Vert (p)}$, then the $\MR{C}^{M}$ norm of $\beta^{-1}$ is controlled by $K$ and $F\left(\beta^{-1}\sum_{i=1}^{t}\alpha_{i}\widetilde{L_{i}},p,\delta\right)\simeq_{K}F\left(\sum_{i=1}^{t}\alpha_{i}\widetilde{L_{i}},p,\delta\right)$.

Note that if $Y$ and $Z$ are two linear combinations (with
constant coefficients) of the $L_{i}^{\Omega}$, by extremality, $F(Y+Z,q,\delta)\leq K^{2}\left[F(Y,q,\delta)+F(Z,q,\delta)\right]$,
and then 
\begin{equation}
F(Y+Z,q,\delta)\geq\frac{1}{K^{2}}F(Y,q\delta)-F(Z,q,\delta).\label{eq:F-of-X-pluY-Omega}
\end{equation}
This implies that the left hand side of (\ref{eq:proof-prop-B-tilde-extremal-A})
is $\lesssim$ than the right hand side, and we have only to prove the
converse inequality. To do it, we consider separately the two possibilities
for $L_{n-k}^{\rho}$.

Suppose first $L_{n-k}^{\rho}=T_{n-k}$.

If the right hand side of (\ref{eq:proof-prop-B-tilde-extremal-A})
is equivalent to $F(X^{\rho},q,\delta)+\left|\alpha_{n-k}\right|^{2}F(L_{n-k}^{\rho},q,\delta)$,
by (\ref{eq:F-of-X-pluY-Omega}), we have only to consider the case
when $F(X^{\rho},q,\delta)\simeq\left|\alpha_{n-k}\right|^{2}F(L_{n-k}^{\rho},q,\delta)$.
Using that $F(T_{n-k},q,\delta)\lesssim F(L_{n-k}^{\Omega},q,\delta)$,
\lemref{comparison-F-rho-phi-F-Omega-5.3} gives the result. 

Suppose now that the right hand side of (\ref{eq:proof-prop-B-tilde-extremal-A})
is equivalent to 
\[
\frac{\varphi''(\left|p\right|^{2})}{\delta}\left(\left|\left\langle (X^{\rho}\circ\pi)(p),p\right\rangle \right|^{2}+\left|\alpha_{n-k}\right|^{2}\left|\left\langle (L_{n-k}^{\rho}\circ\pi)(p),p\right\rangle \right|^{2}\right).
\]
 Then, we only have to consider the case when $\left\langle (X^{\rho}\circ\pi)(p),p\right\rangle =-(1+\varepsilon)\alpha_{n-k}\left\langle (L_{n-k}^{\rho}\circ\pi)(p),p\right\rangle $,
with $\varepsilon$ small. Then if $W$ is the vector field $X^{\rho}+(1+\varepsilon)\alpha_{n-k}L_{n-k}^{\rho}$
normalized at $q$, $W\in\mathcal{H}_{n-k}$ and thus $F(W,q,\delta)\geq\frac{\varphi''}{\delta}\left|\left\langle T_{n-k}(q),p\right\rangle \right|^{2}=\frac{\varphi''}{\delta}\left|\left\langle L_{n-k}^{\rho}(q),p\right\rangle \right|^{2}$.
Then $F(X^{\rho},q,\delta)\gtrsim\frac{1}{K^{2}}\left(\frac{\varphi''}{\delta}\left|\left\langle L_{n-k}^{\rho}(q),p\right\rangle \right|^{2}\right)-2F(L_{n-k}^{\rho},q,\delta)$,
and the conclusion follows.

To finish, suppose that $L_{n-k}^{\rho}=W_{n-k}$.

If the right hand side of (\ref{eq:proof-prop-B-tilde-extremal-A})
is equivalent to $\frac{\varphi''(\left|p\right|^{2})}{\delta}\left(\left|\left\langle (X^{\rho}\circ\pi)(p),p\right\rangle \right|^{2}+\left|\alpha_{n-k}\right|^{2}\left|\left\langle (L_{n-k}^{\rho}\circ\pi)(p),p\right\rangle \right|^{2}\right)$,
there is nothing to do because $\left\langle L_{n-k}^{\rho}\circ\pi(p),p\right\rangle =0$.

Suppose then that the right hand side of (\ref{eq:proof-prop-B-tilde-extremal-A})
is equivalent to $F(X^{\rho},q,\delta)+\left|\alpha_{n-k}\right|^{2}F(L_{n-k}^{\rho},q,\delta)$.
As before, the conclusion is evident except if $F(X^{\rho},q,\delta)\simeq\left|\alpha_{n-k}\right|^{2}F(L_{n-k}^{\rho},q,\delta)$.
Suppose 
\[
F(X^{\rho}+\alpha_{n-k}L_{n-k}^{\rho},q,\delta)+\frac{\varphi'(\left|p\right|^{2})}{\delta}\ll\left|\alpha_{n-k}\right|^{2}F(W_{n-k},q,\delta).
\]
 Note that $\left\langle T_{n-k}(q),p\right\rangle \neq0$, and we
can define $W=X^{\rho}+\alpha_{n-k}L_{n-k}^{\rho}+\mu T_{n-k}$ such
that $\left\langle W(q),p\right\rangle =0$. Then by \lemref{comparison-F-rho-phi-F-Omega-5.3},
\[
\left|\alpha_{n-k}\right|^{2}F(W_{n-k},q,\delta)\gg F(L_{n-k}^{\Omega},q,\delta),
\]
and (extremality of $\brond^{\Omega}$) $\left|\left\langle T_{n-k}(q),p\right\rangle \right|^{2}\frac{\varphi''\left(\left|p\right|^{2}\right)}{\delta}>\frac{1}{K}(F(W_{n-k},q,\delta)-KF(L_{n-k}^{\Omega},q,\delta)$.
From this we deduce $\left|\mu\right|\ll\left|\alpha_{n-k}\right|$
and $W$ is of norm almost $1$ at $q$. Then
\begin{eqnarray*}
F(W,q,\delta) & \leq & K^{2}\left(F(X^{\rho}+\alpha_{n-k}L_{n-k}^{\rho},q,\delta)+\left|\mu\right|^{2}F(T_{n-k},q,\delta)\right)\\
 & \ll & \left|\alpha_{n-k}\right|^{2}\left(F(W_{n-k},q,\delta)+F(L_{n-k}^{\Omega},q,\delta)\right),
\end{eqnarray*}
 because $T_{n-k}\in\MR{E}\left(L_{n-k}^{\Omega},\ldots,L_{n-1}^{\Omega}\right)$,
and thus $F(W,q,\delta)\ll F(W_{n-k},q,\delta)$ which contradicts
the definition of $W_{n-k}$.

To see that $\widetilde{\brond}$ satisfy EB$_{\text{2}}$, a simple
computation shows that it suffices to apply \lemref{brackets-L-rho-5.4}
and \propref{comparison-lists-D-lists-up-5.1}.
\end{proof}

Then, by \lemref{multiplication-fields-by-function-Cinfinity} we
conclude:
\begin{spprop}
\label{prop:final-extremal-basis-local-domain}The basis $\brond$
previously defined by $\brond=\left\{ L_{i},\ldots,l_{n-1}\right\} $,
with $L_{i}=L_{i}^{\rho}\circ\pi+\left(\beta_{i}\circ\pi\right)N^{\rho}\circ\pi$
is $(K',p,\delta)$-extremal for a constant $K'$ depending on the
constant $K$ of extremality of $\brond^{\Omega}$ and the data.
\end{spprop}
Now the proof of \thmref{existence-ext-basis-local-domain} is complete.

\section{Geometrically separated domains\label{sec:Geometrically-separated-domains}}

\subsection{Definition and examples\label{sec:geom-sep-def-examples}}
\begin{sddefn}
\label{def:def-geometrically-separable-domain}Let $\Omega=\{\rho<0\}$
be a bounded pseudo-convex domain with ${\MR{C}}^{\infty}$ boundary
($\nabla\rho\neq0$ in a neighborhood of $\partial\Omega$). We say that
$\Omega$ is \emph{$K$-geometrically separated} at $p_{0}\in\partial\Omega$
if $p_{0}$ is a point of finite $1$-type $\tau$ and there exist
two neighborhoods of $p_{0}$, $W(p_{0})\Subset V(p_{0})$, a constant
$\delta_{0}>0$, a constant $K>0$, an integer $M$ larger than $\tau+1$
and a basis $\brond^{0}=\{L_{1}^{0},\ldots,L_{n-1}^{0}\}$ of $(1,0)$
vector fields tangent to $\rho$ in $V(p_{0})$, whose $\MR{C}^{2M}$
norm are bounded by $K$ and their ``determinant'' bounded from below
by $1/K$, and a positive real number $\delta_{0}$ such that:

For each point $p\in W(p_{0})\cap\partial\Omega$ and each $\delta$,
$0<\delta<\delta_{0}$, there exits a $(M,K,p,\delta)$ extremal basis
$\brond(p,\delta)=\{L_{1}^{p,\delta},\ldots,L_{n-1}^{p,\delta}\}$
such that, for each $i$, the vector field $L_{i}^{p,\delta}$ can
be written (on $V(p_{0})$) $L_{i}^{p,\delta}=\sum_{j}a_{i}^{j}L_{j}^{0}$
with $a_{i}^{j}\in\mathbb{C}$, $\sum\left|a_{i}\right|^{2}=1$. In
other words, the $L_{i}^{p,\delta}$ are normalized vector fields
belonging to the vector space $E_{0}$ generated by $\brond^{0}$.
\end{sddefn}
A notable property (that will not be used later) of these domains
is that the weights $F_{i}$ satisfy a better estimate than the one
given in \propref{finite-type-imply-large-F(L-p-delta)}:
\begin{prop*}
Suppose $\Omega$ is geometrically separated at $p_{0}$ (of type
$\tau$). Then for $V(p_{0})$ and $\delta_{0}$ sufficiently small,
there exists a constant $C>0$ depending only on $K$ and $\Omega$,
such that the extremal basis $\brond(p,\delta)=\left\{ L_{i}^{p,\delta},\,1\leq i\leq n-1\right\} $,
$p\in W(p_{0})\cap\partial\Omega$, $0<\delta<\delta_{0}$, satisfies
$F_{M}(L_{i}^{p,\delta},p,\delta)\geq C\delta^{-2/\tau+1}$, for all
$i$ and all $\delta\in\left[0,\delta_{0}\right]$, with $M=\left[\tau\right]+1$.\end{prop*}
\begin{proof}
Suppose there exists a sequence of points $p_{m}$ converging to $p_{0}$,
a sequence $\delta_{m}$ in $]0,\delta_{0}[$ and an integer $i\leq n-1$
such that, denoting $\brond(p_{m},\delta_{m})=\left(L_{1}^{m},\ldots,L_{n-1}^{m}\right)$
the $(M,K,p_{m},\delta_{m})$-extremal basis at $p_{m}$, we have
$\sum_{\Lrond\in\Lrond_{M}(L_{i}^{m})}\left|\Lrond(\partial\rho)(p_{m})\right|\leq1/m$.
Then $L_{i}^{m}=\sum a_{i}^{j}(p_{m})L_{j}^{0}$, $\sum\left|a_{i}^{j}(p_{m})\right|^{2}=1$,
and we may suppose that the sequences $n\mapsto a_{i}^{j}(p_{m})$
converge to complex numbers $a^{j}$ satisfying $\sum\left|a^{j}\right|^{2}=1$.
Then, by uniform convergence, the vector field $L=\sum a^{j}L_{j}^{0}$
satisfies $F_{M}(L,p_{0},\delta)=0$, for all $\delta$. But, we have
$L=\sum b_{k}L_{k}^{p_{0}}$, $\sum\left|b_{k}\right|^{2}\geq_{K}1$,
and, by extremality $F(L,p_{0},\delta)\simeq_{K}\sum\left|b_{k}\right|^{2}F_{M}(L_{k}^{p_{0}},p_{0},\delta)$,
thus there exists $k$ such that $F_{M}(L_{k}^{p_{0}},p_{0},\delta)=0$,
i. e. $\sum_{\Lrond\in\Lrond_{M}(L_{k}^{p_{0}})}\left|\Lrond(\partial\rho)(p_{0})\right|=0$.
Then, by (4) of \defref{basis-and-coordinates-adapted} this contradicts
the definition of the $1$-type.
\end{proof}

Thus, in all the paper, for a geometrically separated domain at a
boundary point $p_{0}$, the integer $M$ could be changed to $\left[\tau\right]+1$.
As this change gives no advantage, we will keep $M=M'(\tau)$ and
then we can apply directly the results of the preceding Sections.
\begin{srrem}
\label{rem:extremal-basis-inside-vs-boundary}Suppose $\Omega$ is
geometrically separated at $p_{0}\in\partial\Omega$. Let $p$ be
a point of $\overline{\Omega}\cap W(p_{0})$. If $\pi$ is the projection
onto $\partial\Omega$ defined in \secref{Preliminary-remarks-extremal-basis-local-domain-D}
let $q=\pi(p)$. Then, reducing $W(p_{0})$ and $\delta_{0}$ if necessary,
if $-\frac{1}{3}\rho(p)<\delta<\delta_{0}$, the basis $B(q,\delta)=(L_{1}^{q,\delta},\ldots,L_{n-1}^{q,\delta})$
is clearly $(2K,p,\delta)$-extremal, and $F_{M}(L_{i}^{q,\delta},p,\delta)\geq C'\delta^{-2/\tau+1}$
for a constant $C'$ depending only on $K$ and the data. \emph{Thus
we will always assume that a geometrically separated domain is equipped,
by definition, with extremal bases of the form given in the definition,
at every point of} $V(p_{0})\cap\overline{\Omega}$ \emph{for} $-\frac{1}{3}\rho(p)<\delta<\delta_{0}$.
\end{srrem}
This is clear, because if $\Lrond\in\Lrond_{M}(\brond)$, then $\left|\Lrond(\partial\rho)(p)-\Lrond(\partial\rho)(\pi(p))\right|=O(\delta)$,
where $O$ depends only on $K$ and $\Omega$. Then EB$_{\text{1}}$
is satisfied because $F_{i}(p,\delta)\geq C\delta^{-2/M}$ with $C$
depending only on $\Omega$ and EB$_{\text{2}}$ is also satisfied
because $F_{k}(p,\delta)\leq\delta^{-2}$ ($\delta_{0}$ small enough).
\begin{sexmexample}
\label{exa:Example-geom-sep-domains}~
\begin{enumerate}
\item The three first examples of extremal basis given in \exaref{Example-extremal-basis}
immediately show that, if $p_{0}$ is a point of finite type of $\partial\Omega$
then $\Omega$ is geometrically separated at $p_{0}$, under one of
the following four conditions:

\begin{enumerate}
\item $\partial\Omega$ is convex near $p_{0}$, or, more generally, lineally
convex near $p_{0}$ (see \secref{Examples-The-lineally-convex-case});
\item The eigenvalues of the Levi form are comparable at $p_{0}$;
\item The Levi form is locally diagonalizable at $p_{0}$.
\item Near $p_{0}$, $\partial\Omega$ belongs to the class introduced by
M. Derridj in \cite{Derridj-Holder-blocs-1999}.
\end{enumerate}
\item Moreover, we will see in \secref{Localization-geom-sep-structure}
that, if $\Omega$ is geometrically separated at $p_{0}$ then the
local domain $D$ defined in \secref{definition-of-the-local-domain}
is geometrically separated at every point of its boundary.
\end{enumerate}
\end{sexmexample}
\begin{sexmexample}
The domain $\Omega=\left\{ z\in\mathbb{C}^{3}\mbox{ such that }\Re\mathrm{e}z_{1}+\left|z_{2}\right|^{6}+\left|z_{3}\right|^{6}+\left|z_{2}z_{3}\right|^{2}<0\right\} $
studied by G. Herbort in \cite{Herbort-logarithm} is not geometrically
separated at $(0,0)$ (see \secref{example-non-geometrically-separated}
for details).
\end{sexmexample}

\subsection{Structure of homogeneous space}

First recall that we define in \secref{polydisc-associated-extremal-basis}
the {}``polydisc'' $B^{c}(\brond,p,\delta)$ (\defref{definition-polydissc-ext-basis})
and the {}``pseudo-balls'' $B_{\mathrm{exp}}^{c}(\brond,p,\delta)$
and $B_{\MR{C}}^{c}(\brond,p,\delta)$ (\defref{definition-pseudo-balls-curves-exp}).

In general, we will just denote by $B_{\mathrm{exp}}^{c}(p,\delta)$
and $B_{\MR{C}}^{c}(p,\delta)$ the pseudo-balls $B_{\mathrm{exp}}^{c}(\brond,p,\delta)$
and $B_{\MR{C}}^{c}(\brond,p,\delta)$ omitting $\brond$, but recall
that, if $\delta_{1}\neq\delta_{2}$, the balls $B_{\mathrm{exp}}^{c}(p,\delta_{1})$
and $B_{\mathrm{exp}}^{c}(p,\delta_{2})$ are not necessarily constructed
with the same basis.

Then by the methods used in \cite{Charpentier-Dupain-Geometery-Finite-Type-Loc-Diag}
(based on the Campbell-Hausdorf formula and the ideas of \cite{Nagel-Stein-Wainger}),
reducing $W(p_{0})$ if necessary, one can prove the following properties
of the balls:
\begin{spprop}
\label{prop:comparison-exp-balls-curves-polydisc}There exist constants
$c_{0}$, $\delta_{0}$, $\alpha$, $\beta$ and $\gamma$ such that,
for $p\in W(p_{0})\cap\partial\Omega$, $\delta\leq\delta_{0}$ and
$c\leq c_{0}$, we have $B_{\mathrm{exp}}^{\alpha c}(p,\delta)\subset B^{c}(p,\delta)\subset B_{\mathrm{exp}}^{\beta c}(p,\delta)$
and $B_{\mathrm{exp}}^{c}(p,\delta)\subset B_{\MR{C}}^{c}(p,\delta)\subset B_{\mathrm{exp}}^{\gamma c}(p,\delta)$.
\end{spprop}
The importance of this Proposition to construct the structure of homogeneous
space is the following: to be able to use Taylor's formula, we have
to work with a coordinates system, which is easy in the sets $B^{c}(p,\delta)$;
the hypothesis of geometric separation and \propref{properties-vector-coord-weights-balls-3-5}
imply that the sets associated to curves are associated to a pseudo-distance;
and, finally, the sets associated to the exponential map are used
to prove that all these sets are equivalent.
\begin{proof}
[Ideas of the proof of \propref{comparison-exp-balls-curves-polydisc}]It
is similar to the proofs of Proposition 3.4 (p. 96) and Lemma 3.16
(p. 101) of \cite{Charpentier-Dupain-Geometery-Finite-Type-Loc-Diag}.
Thus we will only give the main articulations.

The first inclusion comes easily from the control of the coefficients
of the vector fields $L_{i}$ in the coordinate system $(z_{i})$
in the polydisc (\propref{properties-vector-coord-weights-balls-3-5}).
The second one is more complicated.

Let $\mathrm{exp}_{p}$ be the exponential map based at $p$ relatively
to the vector fields $\mathcal{Y}_{i}$ (real an imaginary parts of
the $L_{i}$). Let $\Psi^{p}=\left(\Psi_{i}^{p}\right)_{i=2,\ldots,2n}=\left(\mathrm{exp}_{p}\right)^{-1}$.
We establish the following estimate on the derivatives of the functions
$\Psi_{i}^{p}$: there exist constants $\beta$ and $K_{1}$, depending
on $K$ and the data, such that
\begin{equation}
\mbox{if }q=\mathrm{exp}_{p}(u),\,\max\left\{ \left|u_{i}\right|,\left|u_{i+n}\right|\right\} \leq\beta F_{i}(p,\delta)^{-1/2}\mbox{ then }\left|\mathcal{Y}_{k}\Psi_{j}^{p}(q)\right|\leq K_{1}F_{k}(p,\delta)^{1/2}F_{j}(p,\delta)^{-1/2},\label{eq:proof-comparison-balls}
\end{equation}
 with the notation of \defref{definition-pseudo-balls-curves-exp}.

To prove this, we estimate the derivatives of the exponential map.
Considering, for $u\in\mathbb{R}^{n}$, the vector field $\mathcal{Y}_{u}=\sum u_{i}\mathcal{Y}_{i}$,
the derivatives of $\mathrm{exp}_{p}$ are estimated via the Campbell-Hausdorff
formula. Let $q=q(u)=\mathrm{exp}_{p}(u)$, $\left|u\right|\leq u_{0}$,
\[
\left|d\mathrm{exp}_{p}\left(\frac{\partial}{\partial u_{i}}\right)(u)-\mathcal{Y}_{i}(q)+\sum_{k=2}^{M}\alpha_{k}\left[\mathcal{Y}_{u},\left[\ldots\left[\mathcal{Y}_{u},\mathcal{Y}_{i}\right]\ldots\right]\right](q)\right|\leq C\left|u\right|^{M+1},
\]
 where $\alpha_{k}$ are universal constants corresponding to brackets
of length $k$ (see Lemma 1 (p. 97) of \cite{Charpentier-Dupain-Geometery-Finite-Type-Loc-Diag}).
The brackets are then estimated with \propref{properties-vector-coord-weights-balls-3-5}
and thus (\ref{eq:proof-comparison-balls}) is easily obtained. The second
inclusion of the Proposition is then easily proved.

The equivalence between the sets defined with the exponential map
and the curves is a quite simple consequence of (\ref{eq:proof-comparison-balls}).\end{proof}
\begin{spprop}
\label{prop:pseudodistance-pseudoballs}Let $\Omega$ be a bounded
pseudo-convex domain $K$-geometrically separated at $p_{0}\in\partial\Omega$.
Let $B$ denote one of the sets $B_{\MR{C}}^{c}$, $B_{\mathrm{exp}}^{c}$ or $B^{c}$.
Then there exists a constant $c_{0}>0$, depending on $K$ and the
data such that, for all $c\leq c_{0}$, the sets $B\left(\brond(p,\delta),p,\delta\right)$
are associated to a pseudo-distance in the following sense: there
exists a constant $C$ depending on $K$ and the data (but not on
$c$) such that, if $p\in W(p_{0})\cap\partial\Omega$ and $\delta\leq\delta_{0}$,
and if $q\in B\left(\brond(p,\delta),p,\delta\right)\cap\partial\Omega$,
then
\[
B(\brond(q,\delta),q,\delta)\subset B(\brond(p,\delta),p,C\delta).
\]
\end{spprop}
\begin{rem*}
If we define $\gamma$, on $W(p_{0})\cap\partial\Omega$, by 
\begin{equation}
\gamma(p,q)=\inf\left\{ \delta\mbox{ such that }q\in B(\brond(p,\delta),p,\delta)\right\} ,\label{eq:definition-pseudo-distance}
\end{equation}
then $\gamma$ is a real pseudo-distance.\end{rem*}
\begin{lem*}
\quad\mynobreakpar
\begin{enumerate}
\item For all $A>0$ there exists $B$ depending on $A$ and $K$ such
that
\[
B_{\MR{C}}^{Ac}(\brond(q,\delta),q\delta)\subset B_{\MR{C}}^{c}(\brond(q,B\delta),q,B\delta).
\]
\item For all $B>0$ there exists $C$ depending on $B$ such that
\[
B_{\MR{C}}^{c}(\brond(q,B\delta),q,B\delta)\subset B_{\MR{C}}^{Cc}(\brond(q,\delta),q,\delta).
\]
\end{enumerate}
\end{lem*}
\begin{proof}[Proof of the Lemma]
Let us denote by $L_{i}(q,\delta)$ (resp $L_{i}(q,B\delta)$) the
vector fields of $\brond(q,\delta)$ (resp. $\brond(q,B\delta)$).
By the hypothesis on $\Omega$, we have $L_{i}(q,\delta)=\sum_{k}\beta_{i}^{k}L_{k}(q,B\delta)$,
with $\beta_{i}^{k}$ constants. By extremality,
\begin{eqnarray*}
\left|\beta_{i}^{k}\right| & \leq & KF(L_{i}(q,\delta),q,B\delta)^{1/2}F(L_{k}(q,B\delta),q,B\delta)^{-1/2}\\
 & \leq & KB^{-1/M}F(L_{i}(q,\delta),q,\delta)^{1/2}F(L_{k}(q,B\delta),q,B\delta)^{-1/2},
\end{eqnarray*}
 which proves the first part of the Lemma with $B=(AK(n-1))^{M}$.
The second part is proved similarly with $C=(BK(n-1))^{M}$.
\end{proof}

\begin{proof}[Proof of \propref{pseudodistance-pseudoballs}]
To prove the assertion on the pseudo-distance in the Proposition,
by \propref{comparison-exp-balls-curves-polydisc}, it is enough to
prove that, there exists a constant $K_{0}$ such that if $q,q'\in B_{\MR{C}}^{c}(\brond(p,\delta),p,\delta)$
then $q'\in B_{\MR{C}}^{K_{0}c}(\brond(q,\delta),q,\delta)$. But
there exists $\varphi$, $\MR C^{1}$ piecewise smooth, such that
$\varphi(0)=q$, $\varphi(1)=q'$ and, almost everywhere, $\varphi'(t)=\sum_{i=1}^{2n}a_{i}(t)\mathcal{Y}_{i}(\varphi(t))$,
with $\max\left\{ \left|a_{i}(t)\right|,\left|a_{i+n}(t)\right|\right\} \leq2cF(L_{i}(p,\delta),p,\delta)\leq4cF(L_{i}(p,\delta),q,\delta)$,
if we choose $c$ small enough (\propref{properties-vector-coord-weights-balls-3-5}).
Now, as in the Lemma, writing $L_{i}(p,\delta)=\sum\alpha_{i}^{k}L_{k}(q,\delta)$
(with $\alpha_{i}^{k}$ constants) and using extremality, we easily conclude
$q'\in B_{\MR{C}}^{K_{0}c}(\brond_{1}^{q,\delta},q,\delta)$.
\end{proof}

Let us now define the ``pseudo-balls'' centered at points of $\Omega\cap W(p_{0})$,
denoted $^{\pi}\! B^{c}(q,\delta)$ (resp. $^{\pi}\! B_{\MR{C}}^{c}(q,\delta)$,
$^{\pi}\! B_{\mathrm{exp}}^{c}(q,\delta)$) by
\[
^{\pi}\! B^{c}(q,\delta)=\left\{ q'\in V(p_{0})\mbox{ such that }\pi(q')\in B^{c}\left(\brond(\pi(q),\delta),\pi(q),\delta\right)\mbox{ and }\rho(q')\in[\rho(q)-c\delta,\rho(q)+c\delta]\right\} .
\]

Then:
\begin{stthm}
\label{thm:homogeneous-space-geometrically-separable}Let $\Omega$
be a pseudo-convex domain geometrically separated at $p_{0}\in\partial\Omega$.
There exists a constant $c_{0}>0$, depending on $K$ and the data,
such that, for all $c\leq c_{0}$, the sets $B(q,\delta)$ define
a structure of ``homogeneous space'' on $W(p_{0})\cap\bar{\Omega}$
in the following sense: there exists a constant $C$, depending only
on $K$ and the data (not on $c$) such that, if $q_{1}\in W(p_{0})\cap\bar{\Omega}$,
$\delta<\delta_{0}$, and $q_{2}\in B(q_{1},\delta),$ we have
\[
B(q_{2},\delta)\subset B(q_{1},C\delta)
\]
 and
\[
\mathrm{Vol}\left(B(q,2\delta)\right)\leq C\mathrm{Vol}\left(q,\delta)\right),
\]
 $B$ denoting one of the sets $^{\pi}\! B_{\MR{C}}^{c}$, $^{\pi}\! B_{\mathrm{exp}}^{c}$
or $^{\pi}\! B^{c}$.\end{stthm}
\begin{proof}
The first assertion follows immediately the Proposition. To prove
the second assertion, we use the fact that both $B_{\MR{C}}^{c}\left(\brond(p,\delta),p,\delta\right)$
and $B_{\mathrm{exp}}^{c}\left(\brond(p,\delta),p,\delta\right)$
are equivalent to $B^{c}\left(\brond(p,\delta),p,\delta\right)$,
the fact that the coordinate system associated to the extremal basis
have a Jacobian uniformly bounded from above and below and the preceding
Lemma.\end{proof}
\begin{srrem}
\label{rem:directional-pseudobals}~
\begin{enumerate}
\item For $p\in\partial\Omega$, the sets $^{\pi}\! B^{c}(q,\delta)\cap\partial\Omega$
(for each definition) are the pseudo-balls of a structure of homogeneous
space on $\partial\Omega\cap W(p_{0})$.
\item On $\partial\Omega$, as in \cite{N-R-S-W-Bergman-dim-2}, we could
define equivalent pseudo-balls using complex tangent curves.
\item It is not difficult to see that the pseudo-balls of the structure
of homogeneous space can also be defined directionnally: they are
equivalent to the sets $B_{\mathrm{dir}}^{c}(p,\delta)$ defined as
to be the set of points $q$ of the form $q=\mathrm{exp}_{p}(a,b)$,
where $\mathrm{exp}_{p}$ is the exponnential map associated to the
vector field $a\Re\mathrm{e}L+b\Im\mathrm{m}L$, $L$ being a vector
field of the linear space $E$, the coefficients $a$ and $b$ satisfying
$\max(|a|,|b|)\leq cF(L,p,\delta)^{-\nicefrac{1}{2}}$.
\end{enumerate}
\end{srrem}

\subsection{Localization\label{sec:Localization-geom-sep-structure}}

Suppose that $\Omega$ is $K$-geometrically separated at $p_{0}\in\partial\Omega$,
and consider the domain $D$ constructed in \secref{definition-of-the-local-domain}
near that point. Then $D$ is $K$-geometrically separated at each
point of $\partial\Omega\cap\partial D$, and, by strict pseudo-convexity,
the same is true on $\partial D\setminus\overline{\partial\Omega\cap\partial D}$.

Suppose that $P$ is a point of the boundary of $\partial\Omega\cap\partial D$,
and let $p$ be a point of $V(P)\cap\partial D$ and $\delta$ small
enough (with the notations of the previous Section). Let us denote
by $\brond(p,\delta)=\left\{ L_{1}^{p,\delta},\ldots,L_{n-1}^{p,\delta}\right\} $
the extremal basis given by \propref{final-extremal-basis-local-domain}
and by $\brond^{0,\Omega}=\left\{ L_{1}^{0,\Omega},\ldots,L_{n-1}^{0,\Omega}\right\} $
the basis denoted $B^{0}$ in \defref{def-geometrically-separable-domain}.
Then, by the construction of $\brond(p,\delta)$ made in the previous
Section, we have $L_{i}^{p,\delta}=L_{i}^{\rho}\circ\pi-\beta(L_{i}^{\rho})N^{\rho}\circ\pi$
with $L\mapsto\beta(L)$ linear. Thus, if we define $\brond^{0,D}=\left\{ L_{1}^{0,D},\ldots,L_{n-1}^{0,D}\right\} $
by $L_{i}^{0,D}=L_{i}^{0,\Omega}\circ\pi-\beta(L_{i}^{0,\Omega})N^{\rho}\circ\pi$,
then we see that the vector fields of $\brond(p,\delta)$ are linear combinations
(with constant coefficients) of the vector fields of $\brond^{0,D}$.
Thus, we have proved the following result:
\begin{stthm}
If $\Omega$ is $K$-geometrically separated at $p_{0}\in\partial\Omega$,
then the domain $D$ defined in \defref{local-domain} is $K'$-geometrically
separated (at every point of its boundary) for a constant $K'$ depending
only on $K$ and the data.\end{stthm}
\begin{rem*}
Recall that every point of $\partial D$ is of finite $1$-type.
\end{rem*}

\section{Adapted pluri-subharmonic function for geometrically separated domains\label{sec:Adapted-pluri-subharmonic-function-geom-sep}}

\subsection{Definition and examples\label{sec:Definition-adapted-PSH-func-geom-sep}}
\begin{sddefn}
\label{def:PSH-adapted-geom-separated}Let $\Omega$ be geometrically
separated at $p_{0}$. Let $E$ be the vector space generated by $\brond^{0}\cup\{N\}$,
and, if $L=\sum_{i=1}^{n-1}b_{i}L_{i}^{0}+b_{n}N=L_{\tau}+b_{n}N\in E$
denotes, for $\delta\leq\delta_{0}$, $F(L,q,\delta)=F(L_{\tau},q,\delta)+\frac{\left|b_{n}\right|^{2}}{\delta^{2}}$.

A $\MR{C}^{3}$ pluri-subharmonic function in $\Omega$, $H_{\delta}$,
is said to be $\beta$-adapted to $\brond^{0}$ at $p_{0}$ if there
exists a constant $\beta$ such that the following properties hold:
\begin{enumerate}
\item $\left|H_{\delta}\right|\leq1$ on $\Omega$;
\item For $q\in W(p_{0})\cap\Omega\cap\left\{ \rho\geq-2\delta\right\} $
and for all vector fields $L\in E$,
\[
\left\langle \partial\bar{\partial}H_{\delta};L,\overline{L}\right\rangle (q)\geq\frac{1}{\beta}F(L,q,\delta);
\]

\item For $q\in W(p_{0})\cap\Omega\cap\left\{ \rho\geq-2\delta\right\} $
and for all lists $\Lrond\in\Lrond_{3}(E)$,
\[
\left|\Lrond H_{\delta}\right|(q)\leq\beta\prod_{L\in\Lrond}F(L,q,\delta)^{1/2}.
\]

\end{enumerate}
\end{sddefn}

\begin{srrem}
Note that (3) implies in particular that, for all $\Lrond\in\Lrond_{3}(\brond(\pi(q),\delta)\cup\{N\})$,
\[
\left|\Lrond H_{\delta}\right|(q)\lesssim F(\brond(\pi(q),\delta),q,\delta)^{\Lrond/2}.
\]
\end{srrem}
\begin{sddefn}
A bounded pseudo-convex domain $\Omega$ is called ``$K$-completely
geometrically separated'' at $p_{0}$ if it is $K$-geometrically
separated and, there exists $\delta_{0}>0$ such that, for all $0<\delta\leq\delta_{0}$,
there exists a pluri-subharmonic function $H_{\delta}$ which is $K$-adapted to $\brond^{0}$
at $p_{0}$.
\end{sddefn}
\begin{sexmexample}
\label{exa:Example-PSH-adaped-to-geom-sep}
\quad\mynobreakpar
\begin{enumerate}
\item If the boundary of $\Omega$ is locally convex near $p_{0}$ (a point
of finite type), then it is proved in \cite{McNeal-convexes-94,McNeal-unif-subel-est-convex}
that it is completely geometrically separated at $p_{0}$. More generally,
using the results of \cite{Diederich-Fornaess-Support-Func-lineally-cvx}
it can be shown that if $\Omega$ is locally lineally convex near
$p_{0}$ (see \cite{Kiselman-lineally-convex}) then it is completely
geometrically separated at $p_{0}$ (see \secref{Examples-The-lineally-convex-case}
for some details on the construction). Moreover, when the boundary
of $\Omega$ is locally convex, resp. locally lineally convex, near
$p_{0}$, the local domain $D$ can be chosen convex, resp. lineally
convex, (choosing $d$ small enough and $K_{0}$ large enough) and
thus, in both cases, it is completely geometrically separated at every
point of its boundary. 
\item In \cite{Cho-02-Bergman-comparable-Korean,Cho-02-metrics-comparable,Cho-03-Bergman-comparable-Math-Anal-Appl},
it is proved that, at a point of finite type, if the eigenvalues of
the Levi form are comparable at $p_{0}$ then it is also completely
geometrically separated at $p_{0}$.
\item In the next Section, we prove that geometrically separated domains
whose extremal bases are strongly extremal with a sufficiently small
$\alpha$ are completely geometrically separated, and, moreover that,
for those domains, the local domain defined in \secref{localization-extremal-basis}
is completely geometrically separated at every point of its boundary.
In particular, this applies when the Levi form is locally diagonalizable
at $p_{0}$.
\item It can also be proved that if a domain is of the type considered by
M. Derridj in \cite{Derridj-Holder-blocs-1999} near a boundary point
$p_{0}$ then it is completely geometrically separated at $p_{0}$.
\end{enumerate}
\end{sexmexample}

\subsection{The case of geometrically separated domains with strongly extremal
bases\label{sec:PSH-func-strong-geom-sep}}

In this Section we prove the two following Theorems:
\begin{stthm}
\label{thm:PSH-function-for-geom-separable-with-strong-ext-basis}Suppose
$\Omega$ is $K$-geometrically separated at $p_{0}\in\partial\Omega$.
Then there exists a constant $\alpha_{0}$, depending on $K$ and
the data, such that, if for all $p\in W(p_{0})\cap\partial\Omega$
and $\delta\leq\delta_{0}$, the bases $\brond(p,\delta)$ are $(K,\alpha,p,\delta)$-strongly
extremal (c.f. \defref{basis-strongly-extremal}) with $\alpha\leq\alpha_{0}$
then it is completely geometrically separated at $p_{0}$.
\end{stthm}
The second theorem deals with the local domain $D$ constructed in \secref{definition-of-the-local-domain},
and, in fact contains the first one:
\begin{stthm}
\label{thm:stron-geom-sep-complete-local-domain}Suppose that $\Omega$
is $K$-geometrically separated at $p_{0}\in\partial\Omega$. There
exists a constant $\alpha_{1}$, depending on $K$ and the data such
that, if for all $p\in W(p_{0})\cap\partial\Omega$ and $\delta\leq\delta_{0}$,
the bases $\brond(p,\delta)$ are $(K,\alpha,p,\delta)$-strongly
extremal with $\alpha\leq\alpha_{1}$, then the local domain constructed
in \secref{definition-of-the-local-domain} is $K'$-completely geometrically
separated at every point of its boundary for a constant $K'$ depending
only on $K$ and $\Omega$.
\end{stthm}
We will prove in details the first Theorem and only give the modifications
needed to obtain the second one.

\subsubsection{\label{sec:Proof-of-thm-5.1}Proof of \thmref{PSH-function-for-geom-separable-with-strong-ext-basis}}

Here we suppose that the bases $\brond(p,\delta)$, $p\in W(p)\cap\partial\Omega$,
$\delta\leq\delta_{0}$, are $(K,\alpha,p,\delta)$-strongly extremal
for a constant $\alpha$ not yet fixed. During the proof, we will
impose successive conditions on $\alpha$ (depending on $K$, $M$
and $n$) to be able to construct the good pluri-subharmonic function.
The existence of $\alpha$ will be clear at the end of the proof but
we will not give an explicit value. Now, we fix $\delta>0$.

The ideas of construction are comparable to those developed in
\cite{Charpentier-Dupain-Geometery-Finite-Type-Loc-Diag}
(following ideas of \cite{Catlin-Est.-Sous-ellipt.}) but the technical
proofs are slightly different. On the one hand the basis are local instead of
global and we have to construct local ``almost pluri-subharmonic''
functions and then add them using the structure of homogeneous space
instead of constructing directly a global function. On the other hand,
the control of lists following our hypothesis are weaker than those
following the local diagonalizability of the Levi form. Thus, for
reader's convenience, we will write the proof with enough details.

Let us first introduce some new notations: $\delta$ being fixed, we denote by $Q^{c}(p,\delta)$ the points $q$
in $W(p_{0})$ such that $\pi(q)$ belongs to the polydisc $B^{c}(p,\delta)$,
associated to the extremal basis $\brond(p,\delta)=(L_i^{p,\delta})_i$
(see \defref{definition-polydissc-ext-basis}). If $L$ is a vector
field in $E$ (the vector space generated by $\brond^{0}$ and $N$),
we write it $L=L_{\tau}+a_{n}N$, where $L_{\tau}$ is tangent to $\rho$.
$\Omega$ being geometrically separated we can write $L_{\tau}=\sum_{i=1}^{n-1}a_{i}^{p}L_{i}^{p,\delta}$
($a_{i}^{p}\in\mathbb{C}$). As usual, $c_{ii}^{p}$ will denote the
coefficient of the Levi form associated to the vector field $L_{i}^{p,\delta}\in\brond(p,\delta)$,
and $\Omega_{\varepsilon}=\left\{ -\varepsilon<\rho<0\right\} $.

With these notations, we now we state a local result and show how it leads to
\thmref{PSH-function-for-geom-separable-with-strong-ext-basis}.
For the proof we need only estimates in the strip
$\Omega_{3\delta}=\left\{ -3\delta\leq\rho\leq0\right\} $,
but in \secref{Proof-of-PSH-local-domain} we will need corresponding results
in a larger domain, and thus we state the local result for the sets
$Q^{c}(p,\delta)$:
\begin{spprop}
\label{prop:local-almost-PSH-function}For all constants $C>1$ there
exist constants $\alpha_{0}$ (depending only on $K$, $c$, $C$
and the data), $\beta$ and $\gamma_{1}$ such that if the bases $\brond(p,\delta)$
are $(K,\alpha,p,\delta)$-extremal with $\alpha\leq\alpha_{0}$,
then for all $\delta\leq\delta(\alpha_{0})$ (depending on $\alpha_{0}$,
$K$ and the data) and all point $p\in W(p_{0})\cap\partial\Omega$,
there exists a function $H_{p,\delta}=H$ with support in $Q^{c}(p,\delta)$
satisfying, for every vector field $L\in E$, the following conditions:
\begin{enumerate}
\item $\left|H\right|\leq1$;
\item $\left\langle \partial\bar{\partial}H;L,\bar{L}\right\rangle (q)\geq\beta F(L_{\tau},q,\delta)-\gamma_{1}\left(\sum_{i=1}^{n-1}\left|a_{i}^{p}\right|^{2}\frac{c_{ii}}{\delta}+\frac{\left|a_{n}\right|^{2}}{\delta^{2}}+1\right)(q)$,
for $q\in Q^{c/2}(p,\delta)\cap\Omega_{3\delta}$,
\item $\left\langle \partial\bar{\partial}H;L,\bar{L}\right\rangle (q)\geq-\frac{\beta}{C}F(L_{\tau},q,\delta)-\gamma_{1}\left(\sum_{i=1}^{n-1}\left|a_{i}^{p}\right|^{2}\frac{c_{ii}}{\delta}+\frac{\left|a_{n}\right|^{2}}{\delta^{2}}+1\right)(q)$,
for $q\in Q^{c/2}(p,\delta)\cap\Omega_{3\delta}$,
\item For $\Lrond\in\Lrond_{3}\left(\brond(p,\delta)\cup\{N\}\right)$,
$\left|\Lrond H\right|(q)\leq\gamma_{2}\prod_{L\in\Lrond}F(L,q,\delta)^{1/2}$,
for $q\in Q^{c/2}(p,\delta)\cap\Omega_{3\delta}$.
\end{enumerate}
\end{spprop}

We will prove this Proposition in the next Section. Now we show how
\thmref{PSH-function-for-geom-separable-with-strong-ext-basis} follows
this Proposition:
\begin{proof}
[Proof of \thmref{PSH-function-for-geom-separable-with-strong-ext-basis}]We
cover $\partial\Omega\cap W(p_{0})$ with a minimal system of pseudo-balls
$^{\pi}\! B^{c/2}(p_{k},\delta)\cap\partial\Omega$, $p_{k}\in\partial\Omega$.
As the pseudo-balls are associated to a structure of homogeneous space,
there exists an integer $S$, independent of $\delta$, such that
each point of $W(p_{0})$ belongs to at most $S$ sets $Q^{c}(p_{j},\delta)$.
Applying \propref{local-almost-PSH-function} with $C=2SC_{1}$ we
get a function $H_{p_{k},\delta}$.

For all point $q\in V(P_{0})\cap\Omega_{3\delta}$ there exists $j_{0}$
such that $q\in Q^{c/2}(p_{j_{0}},\delta)$ and thus (denoting $c_{ii}^{k}$
the coefficient of the Levi form in the direction $L_{i}^{p_{k}}$
and $a_{i}^{k}=a_{i}^{p_{k}}$), by \propref{local-almost-PSH-function},
\begin{equation}
\left\langle \partial\bar{\partial}\sum_{k}H_{p_{k},\delta};L,\bar{L}\right\rangle (q)\geq\frac{\beta}{2}F(L_{\tau},q,\delta)-\gamma_{1}\sum_{k\mbox{ s.t. }q\in Q^{c}(p_{k},\delta)}\left(\sum_{i=1}^{n-1}\left|a_{i}^{k}\right|^{2}\frac{\left|c_{ii}^{k}(q)\right|}{\delta}+\frac{\left|a_{n}\right|^{2}}{\delta^{2}}+1\right).\label{eq:eq-3.5--Claim-Prop-3.8}
\end{equation}

Let us consider the function
\[
H=\sum_{k}H_{p_{k},\delta}+Ae^{-\rho/\delta}+B\left|Z\right|^{2},
\]
for suitable constants $A$ and $B$ and $\alpha$ small enough:
\begin{claim*}
There exist constants $A$, $B$, $\gamma$ and ${\alpha'}_{0}$
depending only on $K$ and the data such that if $\alpha\leq{\alpha'}_{0}$:
\begin{enumerate}
\item $H$ is uniformly bounded, independently of $\delta\leq\delta_{0}$,
on $\Omega_{3\delta}$;
\item For any vector field $L\in E$ and any $q\in\Omega_{3\delta}\cap W(p_{0})$,
$\left\langle \partial\bar{\partial}H;L,\bar{L}\right\rangle (q)\geq\frac{\beta}{2}F(L_{\tau},q,\delta)+\frac{\left|a_{n}\right|^{2}}{\delta^{2}}$;
\item For $q\in\Omega_{3\delta}\cap W(p_{0})$ and all lists $\Lrond\in\Lrond_{3}(E)$,
$\left|\Lrond H\right|(q)\leq\gamma_{2}\prod_{L\in\Lrond}F(L,q,\delta)^{1/2}$.
\end{enumerate}
\end{claim*}
\begin{proof}
[Proof of the Claim]For every $k$ such that $q\in Q^{c}(p_{k},\delta)$,
\[
\left\langle \partial\bar{\partial}e^{\rho/\delta};L,\bar{L}\right\rangle (q)=e^{\rho/\delta}\left[\frac{1}{\delta}\left(\frac{1}{2}\sum_{i,j=1}^{n-1}a_{i}^{k}\overline{a_{j}^{k}}c_{ij}^{k}+2\Re\mathrm{e}\left(\sum_{i=1}^{n-1}a_{i}^{k}\overline{a_{n}^{k}}\left\langle \partial\bar{\partial}\rho;L_{i}^{p_{k}},\bar{N}\right\rangle \right)+\left|a_{n}\right|^{2}\left\langle \partial\bar{\partial}\rho;N,\bar{N}\right\rangle \right)+\frac{\left|a_{n}\right|^{2}}{\delta^{2}}\right].
\]

Then, we use the hypothesis of strong extremality and Taylor's formula
to estimate $\left|c_{ij}^{k}\right|$, $i\neq j$, in the set $Q^{c}(p_{k},\delta)\cap\Omega_{3\delta}$.
Using the fact that $c_{ii}=\left|c_{ii}\right|+\mathrm{O}(\delta)$
(recall $\Omega$ is pseudo-convex), this gives a constant $K_{0}$
depending on $K$ and the data such that
\[
\left\langle \partial\bar{\partial}e^{\rho/\delta};L,\bar{L}\right\rangle (q)\geq-K_{0}+\frac{e^{-3}}{2\delta^{2}}\left|a_{n}\right|^{2}+\frac{e^{-3}}{2\delta}\sum_{i=1}^{n-1}\left|a_{i}^{k}\right|^{2}\left|c_{ii}^{k}(q)\right|-4n^{2}K\alpha F(L_{\tau}^{q},q,\delta),
\]
 because, by definition of $c$, in the sets $Q^{c}(p_{k},\delta)$, we have
$F(L_{\tau},q)\leq3F(L_{\tau},p_{k})$ (see \propref{properties-vector-coord-weights-balls-3-5}).

Now if we choose $A=2Se^{3}\gamma_{1}+1$ and $B=K_{0}A+\gamma_{1}$,
the Claim follows easily (\ref{eq:eq-3.5--Claim-Prop-3.8}).
\end{proof}
To finish the proof of \thmref{PSH-function-for-geom-separable-with-strong-ext-basis},
we truncate $H$ to adapt it to good neighborhoods $V(p_{0})$ and $W(p_{0})$
and to the strip $\{\delta_{\Omega}(p)<2\delta\}$, and we add $D\left|z\right|^{2}$
with a large constant $D$. More precisely, the cutting functions are defined as follows:

Let $\vartheta=\vartheta_{1}\vartheta_{2}$ where $\vartheta_{1}(q)=\chi_{1}\left(\frac{1}{2}\frac{\left|q-p_{0}\right|}{r}\right)$,
with $\chi_{1}$ a $\MR{C}^{\infty}$ increasing function equal to
$0$ on $]-\infty,0]$, $1$ on $[1/4,+\infty[$ and $\chi_{1}(t)=t^{4}$
on $[0,1/8]$, and $\vartheta_{2}(q)=\chi_{\delta}(\rho(q))$ with
$\chi_{\delta}(t)=\chi(t/\delta)$, $\chi$ being even, increasing
on $]-\infty,0[$, equal to $0$ on $]-\infty,-4]$, to $1$ on $]-2,0[$
and to $\frac{\left(t+4\right)^{4}}{16}$ for $t\in[-4,-8/3]$.

Then, remarking that $\left\langle \partial\bar{\partial}\vartheta;L,\bar{L}\right\rangle \geq-\mathrm{O}(1)$ the final computation is made as in
\cite[Section 4.2.3]{Charpentier-Dupain-Geometery-Finite-Type-Loc-Diag}.

\end{proof}

\subsubsection{Proof of \propref{local-almost-PSH-function}}

The proof uses essentially the ideas developed in Section 4.1 of \cite{Charpentier-Dupain-Geometery-Finite-Type-Loc-Diag},
except that we have to work locally around the point $p$. Thus the
technique is more complicated (it needs to use the structure of homogeneous
space) and we will give it with some details.

For $p\in W(p_{0})\cap\partial\Omega$ and $\delta\leq\delta_{0}$
fixed, let $\brond(p,\delta)=\{L_{i}^{p,\delta}=L_{i}^{},\,1\leq i\leq n-1\}$
be the $(K,\alpha,p,\delta)$-strongly extremal basis and $\Phi=\Phi_{p}^{\delta}$ be the
adapted change of coordinates at $(p,\delta)$.

For $i=1,\ldots,n-1$ and $l=3,\ldots,M$, let us define
\[
\MR E_{l}^{i}=\{\Re\mathrm{e}(\Lrond(\partial\rho),\,\Im\mathrm{m}(\Lrond(\partial\rho),\,\left|\Lrond\right|=l-1,\,\Lrond=\{L^{1},\ldots,L^{l-1}\},\, L^{k}\in\{L_{i},\overline{L_{i}}\}\},
\]
\[
\MR E^{i}=\bigcup_{l}\MR E_{l}^{i}.
\]
For $\varphi\in\MR E^{i}$, if $\varphi\in\MR E_{l}^{i}$, we denote $\tilde{l}(\varphi)=l$.

Note that $F_{i}(.,\delta)=F(L_{i},.,\delta)\simeq\frac{\left|c_{ii}\right|}{\delta}+\sum_{\varphi_{i}\in\MR E^{i}}\left|\frac{L_{i}\varphi_{i}}{\delta}\right|^{2/\tilde{l}(\varphi_{i})}$.
The functions $\frac{\left|c_{ii}\right|}{\delta}$ and $\left|\frac{L_{i}\varphi_{i}}{\delta}\right|^{2/\tilde{l}(\varphi_{i})}$
are called the \emph{components} of $F_{i}$ and are denoted generically
$f_{i}$. We also define $l(c_{ii})=2$, and, for the other functions
$f_{i}$, $l(f_{i})=\tilde{l}(\varphi_{i})$. In the following proof,
these components cannot be considered individually. Thus, we introduce
the terminology of {}``$(n-1)$-uplet'' of components: $f=(f_{1},\ldots,f_{n-1})$,
where $f_{i}$ are component of $F_{i}$, is called a $(n-1)$-uplet
of components of the weights $F_{i}$. The set of all such $(n-1)$-uplet
is denoted by $\MR{H}$. $\MR{H}$ is ordered by the lexicographic
order.

First we define a cutoff function with support in $Q^{c}(p,\delta)$
and in the set where a component is ''dominant''. More precisely,
if $B$ is a positive number and $f=(f_{i})$ a $(n-1)$-uplet of
components of $F_{i}$, we define, for fixed $c\leq c_{0}$,
\[
\chi_{f,B}=\prod_{i}\chi_{B}\left(\frac{f_{i}\circ\pi}{F_{i}(p,\delta)}\right)\chi_{0}=\chi'_{f,B}\chi_{0},
\]
 where $\chi_{B}(t)=\chi(Bt)$, $\chi:\,[0,+\infty[\mapsto[0,1]$,
being a $\MR C^{\infty}$ function equal to $0$ on $[0,1/2]$ and
$1$ on $[1,+\infty[$, and $\chi_{0}(q)=\chi_{1}\left(\left(\frac{F_{i}(p,\delta)^{1/2}}{c}\Phi_{p}(\pi(q))_{i}\right)_{i}\right)$,
with $\chi_{1}$ a $\MR C^{\infty}$ function identically $1$ on
$B(0,1/2)$ and with compact support in $B(0,1)$.

We say that $f$ is $B$ dominant if $\chi'_{f,B}=1$.

Then, to each $(n-1)$-uplet $f=(f_{i})$ and to each $i$ such that $f_{i}=\left|\frac{L_{i}^{p}\varphi_{i}}{\delta}\right|^{2/l(f_{i})}$,
we associate, for $\lambda>1$, the function
\[
H_{i}(f,\lambda,B)=\lambda^{-3/2}e^{\lambda\psi_{i}}\chi_{f,B},
\]
 where $\psi_{i}(q)=\frac{\varphi_{i}(\pi(q))}{\delta}F_{i}(p,\delta)^{\frac{1-\tilde{l}(\varphi_{i})}{2}}$.
\begin{sllem}
\label{lem:estimates-psi-i-chi}For each constant $B>0$, there exists
a constant $K_{0}$ depending only on $B$, $c$, $K$ and the data
such that, for each $i$, if $q\in Q^{c}(p,\delta)\cap\Omega_{3\delta}$,
for each $L=\sum_{j=1}^{n}a_{j}L_{j}$, $\sum\left|a_{j}\right|^{2}=1$,
we have the following estimates:
\begin{enumerate}
\item $\left|L\psi_{i}(q)\right|\leq K_{0}\left(F(L_{\tau},q,\delta)^{1/2}+\frac{\left|a_{n}\right|}{\delta}\right)$,
and $\left|\bar{L}L(\psi_{i})(q)\right|\leq K_{0}\left(F(L_{\tau},q,\delta)+\frac{\left|a_{n}\right|^{2}}{\delta^{2}}\right)$;
\item $\left|L\chi_{f,B}(q)\right|\leq K_{0}\left(F(L_{\tau},q,\delta)^{1/2}+\frac{\left|a_{n}\right|}{\delta}\right)$,
and $\left|\bar{L}L\chi_{f,B}\right|\leq K_{0}\left(F(L_{\tau},q,\delta)+\frac{\left|a_{n}\right|^{2}}{\delta^{2}}\right)$;
\item $\left|\left[L,\bar{L}\right]\left(\partial(H_{i}(f,\lambda,B)\right)\right|\leq K_{0}\lambda^{-1/2}e^{\lambda\psi_{i}}\left(F(L_{\tau},q,\delta)^{1/2}+\frac{\left|a_{n}\right|}{\delta}\right)$.
\end{enumerate}
\end{sllem}
\begin{proof}
If $q\in\partial\Omega$, the inequality $\left|L\psi_{i}(q)\right|\leq K_{0}\left(F(L_{\tau},q,\delta)^{1/2}+\frac{\left|a_{n}\right|}{\delta}\right)$
follows immediately from \propref{properties-vector-coord-weights-balls-3-5}
and the extremality of the basis $(L_{i})$ at $(p,\delta)$. The
general case for (1) follows.

Property (2) is obtained using the fact that, if $(z)$ is the coordinates system
associated to $\Phi$ and $L_{i}=\sum a_{i}^{j}\frac{\partial}{\partial z_{j}}$,
then $$\left|a_{i}^{j}\right|\lesssim F_{i}^{1/2}(p,\delta)F_{j}^{-1/2}(p,\delta)$$
for $q\in Q^{c}(p,\delta)\cap\partial\Omega$ (\propref{properties-vector-coord-weights-balls-3-5}),
and techniques similar to those used for (1).

Finally (3) is proved similarly, using the estimates of the coefficients of
the brackets $[L_{i}\overline{L_{j}}]$ in $Q^{c}(p,\delta)\cap\partial\Omega$
(\propref{properties-vector-coord-weights-balls-3-5}).
\end{proof}

For $f=(f_{1},\ldots,f_{n-1})$, a $(n-1)$-uplet of components of
the weights $F_{i}$, let us denote by $I$ the set of indexes $i$
such that $f_{i}=\left|\frac{L_{i}^{p}\varphi_{i}}{\delta}\right|^{2/l(f_{i})}$.
Then we consider the function
\[
H(f,\lambda,B)=\sum_{i\in I}H_{i}(f,\lambda,B).
\]

The next Lemma gives some properties of the function $H(f,\lambda,B)$.
To state it we need to introduce the following set:

For $f$ a $(n-1)$-uplet of components of the weights $F_{i}$ and
$B'$ a positive number, we set
\[
U_{B',f}=\left\{ q\in Q^{c}(p,\delta)\mbox{ for which there exists }f'<f\mbox{ such that }f'(q)\mbox{ is }B'\mbox{ dominant}\right\} .
\]

\begin{sllem}
\label{lem:PSH-component-induction-Lemma}Let $f$ be a $(n-1)$-uplet
of component, $A$, $B$ and $\varepsilon$ three positive fixed real
numbers. Then there exist positive constants $\alpha_{0}$, $\lambda$, $A'$,
$B'$, $A'>A$, $B'>B$, $\varepsilon'\leq\varepsilon$ and $K_{1}$, depending only
on $A$, $B$, $\varepsilon$, $K$ and the data, such that, if the
constant $\alpha$ of strong extremality is $\leq\alpha_{0}$, then
the function $H(f,A,B,\varepsilon)=H(f,\lambda,B)=H$ satisfies, on
$Q^{c}(p,\delta)\cap\Omega_{3\delta}$:
\begin{enumerate}
\item $\left|H\right|\leq K_{1}$;
\item If $L=\sum_{i=1}^{n}a_{i}L_{i}^{p}=L_{\tau}+a_{n}N$, $a_{i}\in\mathbb{C}$,
$\sum\left|a_{i}\right|^{2}=1$, then$\left|\left\langle \partial\bar{\partial}H;L,\bar{L}\right\rangle \right|(q)\leq A'\left(F(L_{\tau},q,\delta)+\frac{\left|a_{n}\right|^{2}}{\delta^{2}}\right)$;
\item If $L=\sum_{i=1}^{n}a_{i}L_{i}^{p}=L_{\tau}+a_{n}N$, $a_{i}\in\mathbb{C}$,
$\sum\left|a_{i}\right|^{2}=1$, $q\notin U_{B'}$, $\chi'_{f,B}(q)=1$ and $\chi_{0}(q)\geq\varepsilon$,
then
\[
\left\langle \partial\bar{\partial}H;L,\bar{L}\right\rangle (q)\geq AF(L_{\tau},q,\delta)-K_{2}\left(\sum_{i=1}^{n-1}\left|a_{i}\right|^{2}\frac{\left|c_{ii}(q)\right|}{\delta}+\frac{\left|a_{n}\right|^{2}}{\delta^{2}}+1\right);
\]
\item If $L=\sum_{i=1}^{n}a_{i}L_{i}^{p}=L_{\tau}+a_{n}N$, $a_{i}\in\mathbb{C}$,
$\sum\left|a_{i}\right|^{2}=1$, the condition
$\left\langle \partial\bar{\partial}H;L,\bar{L}\right\rangle (q)\leq-\left(F(L_{\tau},q,\delta)+\frac{\left|a_{n}\right|^{2}}{\delta^{2}}\right)$
implies $q\in U_{B'}$ and $\chi_{0}(q)\geq\varepsilon'$.
\item For all lists $\Lrond\in\Lrond_{3}\left(\brond(p,\delta)\cup\{N\}\right)$,
$\left|\Lrond H(q)\right|\leq K_{2}\left(\prod_{L\in\Lrond}F(L,q,\delta)^{1/2}\right)$.
\end{enumerate}
\end{sllem}
\begin{proof}
Recall that $H=\sum_{i\in I}H_{i}$, thus the properties are trivially
satisfied if $I=\emptyset$ and we suppose $I\neq\emptyset$. The
functions $\left|\psi_{i}\right|$ being bounded by $2$ (see \propref{properties-vector-coord-weights-balls-3-5}),
(1) is satisfied with a constant $K_{1}$ depending only on $\lambda$
and $n$.

Let $i\in I$. Then $\left\langle \partial\bar{\partial}H_{i};L,\bar{L}\right\rangle =\bar{L}LH_{i}+\left[L,\bar{L}\right]\left(\partial H_{i}\right)$,
and as
\[
\bar{L}LH_{i}=\lambda^{-3/2}e^{\lambda\psi_{i}}\left[\left(\lambda^{2}\left|L\psi_{i}\right|^{2}+\lambda\bar{L}L\psi_{i}\right)\chi_{f,B}+\lambda\left(L\psi_{i}\bar{L}\chi_{f,B}+\bar{L}\psi_{i}L\chi_{f,B}\right)+\bar{L}L\chi_{f,B}\right],
\]
 \lemref{estimates-psi-i-chi} implies $\left\langle \partial\bar{\partial}H_{i};L,\bar{L}\right\rangle (q)\geq\lambda^{-3/2}e^{\lambda\psi_{i}}\left(\lambda^{2}\left|L\psi_{i}\right|^{2}\chi_{f,B}-K'_{0}\lambda F(L_{\tau},q,\delta)+\frac{\left|a_{n}\right|^{2}}{\delta^{2}}+1\right)$
and thus shows the existence of a constant $A'$, depending only on the
choice of $\lambda$, $B$, $c$, $K$ and the data, satisfying (2).

Now, if for all $i\in I$, $\left|\lambda\psi_{i}\right|\leq1$, then,
for $\lambda$ large enough, we have $\left\langle \partial\bar{\partial}H;L,\bar{L}\right\rangle \geq-F(L)$.
Thus we suppose that there exists an $i\in I$ such that $\left|\lambda\psi_{i}(q)\right|=\frac{\lambda\left|\varphi_{i}(\pi(q))\right|}{\delta}F_{i}(p,\delta)^{(1-\tilde{l}(\varphi_{i}))/2}\geq1$.
Then there exists a constant $B'>B$, depending on $\lambda$, such
that $\left|\frac{\varphi_{i}(\pi(q))}{\delta}\right|^{2/(\tilde{l}(\varphi_{i})-1)}>\frac{4}{B'}F_{i}(p,\delta)$,
and this implies that there exists a $(n-1)$-uplet $f'<f$ which
is $B'$-dominant at the point $q$. In other words, to each choice
of $\lambda$ we can associate $B'$ such that the first conclusion
in (4) is true. Moreover, $\lambda$, $B$ and $c$ being fixed, $\chi_{1}$
being $\MR{C}^{\infty}$, there exists $\varepsilon'$, depending
on $\lambda$, $B$, $c$ and $\chi_{1}$, such that the hypothesis
of (4) implies the second conclusion (i. e. $\chi_{0}(q)\geq\varepsilon'$).

Let us now show that we can choose $\lambda$ (thus $A'$, $B'$,
$K_{1}$ and $\varepsilon'$ will be fixed) such that (3) is satisfied
if $\alpha$ is small enough. Suppose then $\chi'_{f,B}(q)=1$and
$\chi_{0}(q)>\varepsilon$. The hypothesis of strong extremality and
the invariance of the $F_{i}(q)$ and the $a_{ij}^{k}$ in $B^{c}(p,\delta)$
( Propositions \ref{prop:properties-vector-coord-weights-balls-3-5}
and \ref{prop:3.9-automatic-lists-inside-strong-extremality}) give,
if $\delta\leq\delta(\alpha)$,
\[
\left|L\psi_{i}(q)\right|^{2}\geq\frac{1}{4}\left|\sum_{j\leq i}a_{j}(L_{j}\psi_{i})(q)\right|^{2}-4nC(K)\left(2\alpha\sum_{n-1\geq j>i}\left|a_{j}\right|^{2}F_{j}(p)+\frac{\left|a_{n}\right|^{2}}{\delta^{2}}\right),
\]
and then, by extremality at $p$,
\[
\left|L\psi_{i}(q)\right|^{2}\geq\frac{1}{4}\left|\sum_{j\leq i}a_{j}L_{j}\psi_{i}(q)\right|^{2}-C_{1}(K)\left(\alpha^{2}F(L_{\tau},q,\delta)+\frac{\left|a_{n}\right|^{2}}{\delta^{2}}+1\right).
\]

Now we make use of the following Lemma:
\begin{lem*}
Let $\beta_{i}^{j}$ be complex numbers, $i=1,2,\cdots,n-1$, $j\leq i$,
verifying $\left|\beta_{i}^{i}\right|\geq c\alpha_{i}$ and $\left|\beta_{i}^{j}\right|\leq C\alpha_{j}$
for $j<i$. Then there exists a constant $W=W(c,C,n)$ such that $\sum_{i=1}^{n-1}\left|\sum_{j=1}^{i}\beta_{i}^{j}\right|^{2}\geq W\sum_{i=1}^{n-1}(\alpha_{i})^{2}$.
\end{lem*}
The above lemma implies, using the invariance of $F_{i}(q)$ and $F(L,q)$ in the
ball and the extremality of the basis at $p$, that there exist constants
$W$, $K_{3}$ and $K_{4}$, depending on $B$, $M$, $K$ and the
data, such that:
\[
\sum_{i\in I}\left|L\psi_{i}(q)\right|^{2}+\sum_{i\notin I}\frac{\left|c_{ii}(q)\right|}{\delta}\geq\frac{W}{2K}F(L_{\tau},q,\delta)-\alpha K_{3}\left(F(L_{\tau},q,\delta)+1\right)-K_{4}\frac{\left|a_{n}\right|^{2}}{\delta^{2}},
\]
 and thus, for $\alpha_{0}=W/4KK_{3}$ (depending only on the data
$M$, $K$, $B$, $c$ and $n$),
\[
\sum_{i\in I}\left|L\psi_{i}(q)\right|^{2}+\sum_{i\notin I}\frac{\left|c_{ii}(q)\right|}{\delta}\geq W'F(L_{\tau},q,\delta)-K_{4}\left(\frac{\left|a_{n}\right|^{2}}{\delta^{2}}+1\right).
\]

This finishes the proof of \lemref{PSH-component-induction-Lemma} for a choice of $\lambda$ depending
on $A$, $\varepsilon$, $B$, $M$, $K$ and $c$, $c$ depending
itself only on $M$, $K$ and the data, the property (5) being trivial.
\end{proof}
\begin{proof}
[Proof of \propref{local-almost-PSH-function}]First, note that
there exists a constant $D$, depending on $M$ and $n$, such that,
for $p\in W(p_{0})$ and $\delta\geq\frac{1}{3}\left|\rho(p)\right|$,
there exists a component $f_{i}$ of $F_{i}(p,\delta)$ verifying
$f_{i}(q)\geq\frac{1}{D}F_{i}(p,\delta)$ for all points $q\in B^{c}(p,\delta)$,
$c\leq c_{0}$ (\propref{properties-vector-coord-weights-balls-3-5}).

To define completely our function $H$, we have to define, for each
$(n-1)$-uplet of component $f\in\MR{H}$ (the set of $(n-1)$-uplets
of components of the weights $F_{i}(p,\delta)$), the constants $A_{f}$,
$B_{f}$ and $\varepsilon_{f}$ from which $\lambda(f)$ is constructed.
Let $f^{0}$ be the largest element of $\MR{H}$ for the lexicographic
order. Define $A_{f_{0}}=C4^{Mn+1}$, $B_{f_{0}}=D$ and $\varepsilon_{f_{0}}=1$.
Suppose we have constructed the constants $A_{f}$, $B_{f}$ and $\varepsilon_{f}$
for $f\geq f^{1}$. Consider the constants $A'_{f^{1}}$, $B'_{f^{1}}$
and $\varepsilon'_{f^{1}}$ obtained applying \lemref{PSH-component-induction-Lemma}
for the constants $A_{f^{1}}$, $B_{f^{1}}$ and $\varepsilon_{f^{1}}$,
and define, for $f^{2}$ preceding $f^{1}$, $A_{f^{2}}=3C\sum_{f>f^{2}}A'_{f}$,
$B_{f^{2}}=B'_{f^{1}}$ and $\varepsilon_{f^{2}}=\varepsilon'_{f^{1}}$.
Thus $H=\sum_{f\in\MR{H}}H(f,A_{f},B_{f}\varepsilon_{f})$ is well
defined.

For $q\in Q^{c}(p,\delta)$ define the following subsets of $\MR{H}$:
\[
\begin{array}{rcl}
E_{1}(q) & = & \left\{ f\in\MR{H}\mbox{ such that there exists }f'<f,\mbox{ such that }f'(q)\mbox{ is }B'_{f}\mbox{-dominant and }\chi_{0}(q)\geq\varepsilon'_{f}\right\} ,\\
E_{3}(q) & = & \left\{ f\in\MR{H}\mbox{ such that }\chi'_{f,B_{f}}(q)=1\mbox{ and }\chi_{0}(q)\geq\varepsilon_{f}\right\} ,\\
E_{2}(q) & = & \MR{H}\setminus\left\{ E_{1}(q)\cup E_{3}(q)\right\} .
\end{array}
\]

Note that if $E_{1}(q)$ is not empty, and if $f$ is its smallest
element, then there exists $f'<f$ such that $f'(q)$ is $B'_{f}$
dominant, that is $\chi'_{f',B_{f'}}(q)=1$, and, as $\varepsilon_{f'}\leq\varepsilon'_{f}$,
we also have $\chi_{0}(q)\geq\varepsilon_{f'}$ which means $f'\in E_{3}(q)$,
$f$ being the smallest element of $E_{1}(q)$.

Now suppose first that $q\in Q^{c/2}(p,\delta)$. Then, by definition
of $D$, $E_{3}(q)$ is not empty, and, if $E_{1}(q)$ is also not
empty there exists in $E_{3}(q)$ some strict minorant of $E_{1}(q)$.
Then, by \lemref{PSH-component-induction-Lemma}
\[
\left\langle \partial\bar{\partial}H;L,\bar{L}\right\rangle (q)\geq\left(\sum_{f\in E_{3}(q)}A_{f}-\sum_{f\in E_{1}(q)}A'_{f}-\#E_{2}(q)\right)F(L_{\tau},q,\delta)-\sum K_{2}(A_{f},B_{f},A'_{f})\left(\sum_{i=1}^{n-1}\left|a_{i}\right|^{2}\frac{c_{ii}(q)}{\delta}+\frac{\left|a_{n}\right|^{2}}{\delta^{2}}+1\right),
\]
 for $\alpha$ small enough, depending only on $M$, $K$ and $n$
($\#E_{2}(f)$ denotes the number of elements of $E_{2}(f)$). Then,
the preceding remark and the fact that $\#E_{2}(q)\leq4^{Mn}\leq\frac{1}{4C}A_{f^{0}}$
imply
\[
\left\langle \partial\bar{\partial}H;L,\bar{L}\right\rangle (q)\geq C4^{Mn}F(L_{\tau},q,\delta)-\gamma_{1}\left(\sum_{i=1}^{n-1}\left|a_{i}\right|^{2}\frac{\left|c_{ii}(q)\right|}{\delta}+\frac{\left|a_{n}\right|^{2}}{\delta^{2}}+1\right).
\]

Finally, if $q$ is any point in $B^{c}(p,\delta)$ then $E_{3}(q)$
may be empty, but then $E_{1}(q)$ is also empty, and thus
\[
\left\langle \partial\bar{\partial}H;L,\bar{L}\right\rangle (q)\geq-4^{Mn}F(L_{\tau},q,\delta)-\gamma_{1}\left(\sum_{i=1}^{n-1}\left|a_{i}\right|^{2}\frac{\left|c_{ii}(q)\right|}{\delta}+\frac{\left|a_{n}\right|^{2}}{\delta^{2}}+1\right).
\]

This finishes the proof of \propref{local-almost-PSH-function}, property
(4) being trivial.
\end{proof}

\subsubsection{Proof of \thmref{stron-geom-sep-complete-local-domain}\label{sec:Proof-of-PSH-local-domain}}

If $P$ is a point of the boundary of $D$ then, by the definition of $D$
and \thmref{PSH-function-for-geom-separable-with-strong-ext-basis},
to prove that there exists a pluri-subharmonic function adapted to
the structure of geometrically separated domain near $P$, we have
only to consider the case where $P$ is in the boundary of $\partial\Omega\cap\partial D$.
Thus, with the notations introduced just before, we prove the following
reformulation of \thmref{stron-geom-sep-complete-local-domain}:
\begin{spprop}
\label{prop:PSH-function-local-domain-D}Let $P$ be a point of the
boundary of $\partial\Omega\cap\partial D$, and $V(P)$ the neighborhood
considered in the previous Section. For all $K>0$, there exist constants
$\alpha_{1}$ and $\delta_{1}$ depending on $K$ and the data such
that if $\Omega$ is $K$-geometrically separated at $p_{0}\in\partial\Omega$
and if the extremal bases of $\Omega$ are $(K,\alpha,p,\delta)$-strongly
extremal with $\alpha\leq\alpha_{1}$, then, for $0<\delta\leq\delta_{1}$,
there exists a pluri-subharmonic function $H_{\delta}$ on the local
domain $D$ which is $(\delta,K')$-adapted to $\brond^{0,D}$ at
$P$.\end{spprop}
\begin{proof}
The proof is a modification of the proof of \thmref{PSH-function-for-geom-separable-with-strong-ext-basis}
and we will only indicate the differences. Here we have to consider
both weights associated to the domains $\Omega$ and $D$,
denoted $F^{\Omega}$ and $F^{D}$, which are constructed respectively
with the defining functions $\rho$ and $\rho+\varphi$.

We fix $\delta$ small enough and then we will omit the subscript $\delta$
in the notations of the vector fields. Consider, as in \secref{Proof-of-thm-5.1}
the covering of $\partial D\cap V(P)$ by the pseudo-balls $B^{c/2}(q_{k},\delta)\cap\partial\Omega$
(note that here $P$ plays the role of $p_{0}$ in the previous Sections).

We denote $p_{k}=\pi(q_{k})$ and fix $k$. Let $\left(L_{i}^{k}\right)_{i}=\left(L_{i}\right)_{i}$
be the $\delta$-extremal basis (for $D$) at the point $q_{k}$.
Let $L_{i}^{\rho}$ be the vector field tangent to $\rho$ associated
to $L_{i}$ (i. e. $L_{i}=L_{i}^{\rho}\circ\pi+(\beta\circ\pi)N_{\Omega}\circ\pi$).
We saw (in \secref{Extremal-basis-on-local-domain-D}) that the weights
$F^{D}(L_{i},.,\delta)$ associated to the vector fields $L_{i}$
are equivalent to
\[
F^{\Omega}(L_{i}^{\rho},\pi(.),\delta)+\frac{\varphi'\left(\left|.\right|\right)}{\delta}+\frac{\varphi''\left(\left|.\right|^{2}\right)}{\delta}\left|\left\langle L_{i}^{\rho}\left(\pi(.)\right),.\right\rangle \right|^{2}.
\]

Let $\left(L_{i}^{\Omega,k}\right)_{i}=\left(L_{i}^{\Omega}\right)_{i}$
be the $\delta$-extremal (for $\Omega$) basis at $p_{k}$ so that
the vector fields $L_{i}^{\rho}$ are linear combinations of the $L_{i}^{\Omega}$.
Let
\[
I_{k}=\left\{ i\mbox{ such that }F^{\Omega}\left(L_{i}^{\Omega},p_{k},\delta\right)>\frac{\varphi'\left(\left|q_{k}\right|^{2}\right)}{\delta}\right\} .
\]
 We suppose $I_{k}$ non empty. As the vector fields $L_{i}^{\Omega}$
are ordered so that their weights are decreasing, $I_{k}$ is a segment
of $\mathbb{N}$, $\left\{ 1,2,\ldots,n_{k}\right\} $. Then, we consider
the $n_{k}$-uplets of components of the weights $F_{\Omega}(L_{i}^{\Omega},p_{k},\delta)$,
$i\leq n_{k}$, $f=\left(f_{1},\ldots,f_{n_{k}}\right)$ and the function
\[
\chi_{f,B}=\prod_{i\leq n_{k}}\chi_{B}\left(\frac{f_{i}\circ\pi}{F_{\Omega}\left(L_{i}^{\Omega},p_{k},\delta\right)}\right)\chi_{0},
\]
 where $\chi_{0}(q)=\chi_{1}\left(\frac{F_{D,i}(q,\delta)}{c}\widetilde{\Phi}_{q_{k}}(\pi_{D}(q)\right)$,
$\pi_{D}$ being the projection onto $\partial D$ associated to the
real normal to $D$.

To obtain the good estimates of the derivatives of $\chi_{f,B}$ with
respect to the vector fields $L_{i}$, we first estimate the derivatives
of the functions $f_{i}\circ\pi$ at the point $q_{k}$:
\begin{sllem}
For $i\in I_{k}$, if $\left|f_{i}(p_{k})\right|\geq\frac{1}{2B}F^{\Omega}(L_{i}^{\Omega},p_{k},\delta)$,
for $\Lrond\in\Lrond_{M}\left(L_{1},\ldots,L_{n-1}\right)$, we have
\[
\left|\Lrond\left(f_{i}\circ\pi\right)(q_{k})\right|\lesssim_{B}F^{\Omega}(L_{i}^{\Omega},p_{k},\delta)F^{D}(q_{k},\delta)^{\Lrond/2}.
\]
\end{sllem}
\begin{proof}
Let us consider the case $\left|\Lrond\right|=1$. As $L_{j}=L_{j}^{\rho}\circ\pi+(\beta_{j}\circ\pi)N_{\Omega}\circ\pi$,
for $p=\pi(q)$, 
\[
\left|L_{j}\left(f\circ\pi\right)(q)\right|\leq\left|L_{j}^{\rho}(p)\left(f_{j}\circ\pi\right)(q)\right|+\mathrm{O}\left(\left|\beta_{j}(p)\right|\right).
\]
By (\ref{eq:formula-beta-composed-rho}), $\left|\beta_{j}(p)\right|\lesssim\varphi'\left(\left|p\right|^{2}\right)\lesssim\varphi'\left(\left|q\right|^{2}\right)+\mathrm{O}\left(\varphi\left(\left|q\right|^{2}\right)\right)=\mathrm{O}\left(\varphi'\left(\left|q\right|^{2}\right)\right)$,
thus 
\[
\left|\beta_{j}(p_{k})\right|\lesssim{F^{\Omega}}^{1/2}(L_{i}^{\Omega},p_{k},\delta){F^{D}}^{1/2}(L_{j},q_{k},\delta)
\]
because $i\in I_{k}$. As $L_{j}^{\rho}$ are tangent to $\rho$,
$L_{j}^{\rho}(p)\left(f_{i}\circ\pi\right)(q)=L_{j}^{\rho}(p)(f_{i})(p)$,
and, as $L_{j}^{\rho}$ are in the space spanned by the $L_{i}^{\Omega}$,
by \propref{control-lists-space-generated}, we have
\[
\left|\left(L_{j}^{\rho}(p_{k})(f_{i})\right)(p_{k})\right|\lesssim F^{\Omega}\left(L_{i}^{\Omega},p_{k},\delta\right)F^{\Omega}\left(L_{j}^{\rho},p_{k},\delta\right)^{1/2},
\]
 and thus
\[
\left|\left(L_{j}^{\rho}(p_{k})(f_{i})\right)(q_{k})\right|\lesssim F^{\Omega}\left(L_{i}^{\Omega},p_{k},\delta\right)F^{\Omega}\left(L_{j}^{\rho},p_{k},\delta\right)^{1/2}+\frac{\varphi\left(\left|q_{k}\right|^{2}\right)}{\delta}.
\]

Derivatives of higer order are treated similarly.\end{proof}
\begin{cor*}
Under the same hypothesis, for $q\in Q_{D}^{c}(q_{k},\delta)\cap D_{3\delta}$
and $\Lrond\in\Lrond_{M}\left(L_{1},\ldots,L_{n-1}\right)$, we have
\[
\left|\Lrond\left(f_{i}\circ\pi\right)(q)\right|\lesssim_{B}F^{\Omega}(L_{i}^{\Omega},p_{k},\delta)F^{D}(q_{k},\delta)^{\Lrond/2}.
\]

\end{cor*}

The derivatives of $\chi_{0}$ being trivial, we deduce from (\ref{eq:relation-lists-derivatives})
and Taylor's formula:
\begin{sllem}
For $i\in I_{k}$ and $q\in Q_{D}^{c}(q_{k},\delta)\cap D_{3\delta}$,
for $\Lrond\in\Lrond_{M}\left(L_{1},\ldots,L_{n-1}\right)$, we have
\[
\left|\Lrond\chi_{f,B}(q)\right|\lesssim_{B}F^{D}(q_{k},\delta)^{\Lrond/2}.
\]

\end{sllem}
We now define the basic functions $H(f,\lambda,B)$ used here. Let
$I=\left\{ i\leq n_{k}\mbox{ such that }f_{i}\neq\frac{\left|c_{ii}\right|}{\delta}\right\} $,
and, for $i\in I$, if $f_{i}=\left|\frac{L_{i}\varphi_{i}}{\delta}\right|^{2/l(f_{i})}$, we put
\[
H_{i}(f,\lambda,B)=\lambda^{-3/2}e^{\lambda\psi_{i}}\chi_{f,B}\mbox{ where }\psi_{i}=\frac{\varphi_{i}}{\delta}\left|\frac{L_{i}\varphi_{i}}{\delta}\right|^{\frac{1}{\tilde{l}(\varphi_{i})-1}},
\]
and we define $H$ by $$H=H(f,\lambda,B)=\sum_{i\in I}H_{i}(f,\lambda,B).$$

If $L=\sum_{i=1}^{n}a_{i}L_{i}=L_{\tau}+a_{n}L_{n}$ ($\sum\left|a_{i}\right|^{2}=1$),
as in the proof of \lemref{PSH-component-induction-Lemma}, using
the last Lemma we get
\[
\left\langle \partial\bar{\partial}H_{i};L,\bar{L}\right\rangle (q)\geq\lambda^{-3/2}e^{\lambda\psi_{i}}\left(\lambda^{2}\left|L\psi_{i}\right|^{2}\chi_{f,B}-K_{0}\left(\lambda F^{D}(L_{\tau},q_{k},\delta)+\frac{\left|a_{n}\right|^{2}}{\delta}+1\right)\right),
\]
 for $q\in Q_{D}^{c}(q_{k},\delta)\cap D_{3\delta}$. The estimate
of $\left|L\psi_{i}\right|^{2}$ has now to be done more carefully.

The formula $L_{\tau}(q)=L_{\tau}^{\rho}(p)+\beta(p)N_{\Omega}$ gives
\[
\left|L_{\tau}\psi_{i}\right|^{2}(q)\geq\frac{1}{4}\left|L_{\tau}^{\rho}(p)\psi_{i}\right|^{2}(q)-C\left(\frac{\varphi'\left(\left|q\right|^{2}\right)}{\delta}+\frac{\varphi\left(\left|q\right|^{2}\right)+\delta}{\delta}\right).
\]
Then, decomposing $L_{\tau}^{\rho}$ on the $\delta$-extremal basis at $p_{k}$
($\left(L_{i}^{\Omega}\right)_{i}$), $L_{\tau}^{\rho}=\sum_{j=1}^{n-1}b_{j}L_{j}^{\Omega}$,
we obtain, using the strong extremality hypothesis,
\[
\left|L_{\tau}^{\rho}(p)\psi_{i}\right|^{2}(q)\geq\frac{1}{4}\left|\sum_{j\leq i}b_{j}L_{j}^{\Omega}\psi_{i}\right|^{2}(q)-C\left(\alpha^{2}F^{\Omega}(L_{\tau}^{\rho},p_{k},\delta)+\frac{\varphi\left(\left|q\right|^{2}\right)+\delta}{\delta}\right).
\]

Using the same method, we sum all these inequality to get (writing
$c_{ii}^{\Omega}=\left[L_{i}^{\Omega},\overline{L_{i}^{\Omega}}\right](\partial\rho)$)
\begin{eqnarray*}
\sum_{i\in I}\left|L_{\tau}^{\rho}\psi_{i}\right|^{2}(q) & \geq & \beta F^{\Omega}\left(\sum_{j=1}^{n_{k}}b_{j}L_{j}^{\Omega},p_{k},\delta\right)-C\left(\sum_{i\notin I,\, i\leq n_{k}}\left|b_{i}\right|^{2}\frac{\left|c_{ii}^{\Omega}\right|}{\delta}(q)+\alpha^{2}F^{\Omega}(L_{\tau}^{\rho},p_{k},\delta)+\frac{\varphi\left(\left|q\right|^{2}\right)}{\delta}+1\right)\\
 & \geq & \beta F^{\Omega}(L_{\tau}^{\rho},p_{k},\delta)-C\left(\sum_{i\notin I,\, i\leq n_{k}}\left|b_{i}\right|^{2}\frac{\left|c_{ii}^{\Omega}\right|}{\delta}(q)+\alpha^{2}F^{\Omega}(L_{\tau}^{\rho},p_{k},\delta)+\frac{\varphi'\left(\left|q\right|^{2}\right)}{\delta}+1\right),
\end{eqnarray*}
 and, as $\left|L\psi_{i}\right|^{2}\geq\frac{1}{4}\left|L_{\tau}\varphi_{i}\right|^{2}-C\frac{\left|a_{n}\right|^{2}}{\delta^{2}}$,
we finally obtain
\[
\sum_{i\in I}\left|L\psi_{i}\right|^{2}\geq\beta F^{\Omega}(L_{\tau}^{\rho},p_{k},\delta)-C\left(\sum_{i\notin I,\, i\leq n_{k}}\left|b_{i}\right|^{2}\frac{\left|c_{ii}^{\Omega}\right|}{\delta}(q)+\frac{\left|a_{n}\right|^{2}}{\delta}+\alpha^{2}F^{\Omega}(L_{\tau}^{\rho},p_{k},\delta)+\frac{\varphi'\left(\left|q\right|^{2}\right)}{\delta}+1\right).
\]
 Then the proof is finished as in the previous Section using that,
in $Q_{D}^{c}(q_{k},\delta)\cap D_{3\delta}$, we have
\[
\left\langle \partial\bar{\partial}e^{r/\delta};L,\bar{L}\right\rangle (q)\geq\beta\left(\sum_{i=1}^{n-1}\left|b_{i}\right|^{2}\frac{\left|c_{ii}^{\Omega}\right|}{\delta}(q)+\frac{\left|a_{n}\right|^{2}}{\delta^{2}}+\frac{\varphi'\left(\left|q\right|^{2}\right)+\varphi''\left(\left|q\right|^{2}\right)\left|\left\langle L_{\tau}^{\rho}(p),q\right\rangle \right|^{2}}{\delta}\right)-K\left(\alpha F^{\Omega}(L_{\tau}^{\rho},p_{k},\delta)+1\right).
\]
 Indeed, a direct calculation gives
\begin{equation}
\left\langle \partial\bar{\partial}e^{r/\delta};L,\overline{L}\right\rangle =e^{r/\delta}\left(\frac{2\Re\mathrm{e}\left(a_{n}\left\langle \partial\bar{\partial}r;L_{\tau},\bar{N}\right\rangle \right)}{\delta}+\frac{\left|a_{n}\right|^{2}}{\delta^{2}}\right)(q)+\left\langle \partial\bar{\partial}e^{r/\delta};L_{\tau},\overline{L_{\tau}}\right\rangle (q).\label{eq:hessian-of-r}
\end{equation}

For $q\in\left\{ r\geq-3\delta\right\} $, the first term of (\ref{eq:hessian-of-r})
is $\geq\frac{1}{2e^{3}}\frac{\left|a_{n}\right|^{2}}{\delta^{2}}-K_{0}$.
Let us look at the second term of (\ref{eq:hessian-of-r}).
\[
\left\langle \partial\bar{\partial}e^{r/\delta};L_{\tau},\overline{L_{\tau}}\right\rangle =\frac{e^{r/\delta}}{\delta}\left(\left\langle \partial\bar{\partial}\rho;L_{\tau},\overline{L_{\tau}}\right\rangle +\left\Vert L_{\tau}\right\Vert ^{2}\varphi'\left(\left|q\right|^{2}\right)+\left|\left\langle L_{\tau}^{\rho}(p),q\right\rangle \right|^{2}\varphi''\left(\left|q\right|^{2}\right)\right).
\]

But
\[
\left\langle \partial\bar{\partial}\rho;L_{\tau},\overline{L_{\tau}}\right\rangle (q)=\left\langle \partial\bar{\partial}\rho;L_{\tau}^{\rho}(p),\overline{L_{\tau}^{\rho}}(p)\right\rangle (q)+\mathrm{O}\left(\varphi'\left(\left|q\right|^{2}\right)\left|\left\langle L_{\tau}^{\rho}(p),q\right\rangle \right|^{2}\right)
\]
 and, first we can choose $V$ small enough so that $\varphi'\left(\left|q\right|^{2}\right)\left|\left\langle L_{\tau}^{\rho}(p),q\right\rangle \right|^{2}\ll\varphi'\left(\left|q\right|^{2}\right)+\varphi''\left(\left|q\right|^{2}\left|\left\langle L_{\tau}^{\rho}(p),q\right\rangle \right|^{2}\right)$,
and secondly
\begin{eqnarray*}
\left\langle \partial\bar{\partial}\rho;L_{\tau}^{\rho}(p),\overline{L_{\tau}^{\rho}}(p)\right\rangle (q) & = & \sum b_{i}b_{j}c_{ij}^{\Omega}(q)\\
 & \geq & \sum\left|b_{i}\right|^{2}\left|c_{ii}^{\Omega}\right|(p)+\mathrm{O}\left[\alpha\delta F^{\Omega}(L_{\tau}^{\rho}(p))+\varphi\left(\left|q\right|^{2}\right)+\delta\right].
\end{eqnarray*}
The proof of \propref{PSH-function-local-domain-D} is now complete.\end{proof}

\section{Applications to complex analysis\label{sec:Applications-to-complex-analysis}}

\subsection{Statements of the results for geometrically separated domains\label{sec:Analytic-results-geom-sep}}

In \cite{Charpentier-Dupain-Geometery-Finite-Type-Loc-Diag} and \cite{Charpentier-Dupain-Szego-Barcelone}
we proved that the methods introduced, for the study of the Bergman
and Szeg\"o projection, by A. Nagel, J. P. Rosay, E. M. Stein and S.
Wainger in $\mathbb{C}^{2}$ (\cite{N-R-S-W-Bergman-dim-2}) and by
J. McNeal and E. M. Stein for convex domains (\cite{McNeal-Stein-Bergman,McNeal-Stein-Szego})
can be adapted to pseudo-convex domains having an {}``adapted geometry''.
The work made in the previous Sections shows that it is the case for
completely geometrically separated domains and thus we have the following
sharp estimates:
\begin{stthm}
\label{thm:estimate-bergman-kernel-compl-geom-sep}Suppose $\Omega$
is completely geometrically separated at $p_{0}\in\partial\Omega$.
Let $K_{B}(z,w)$ be the Bergman kernel of $\Omega$. There exists
a neighborhood $W(p_{0})$ of $p_{0}$ such that:
\begin{enumerate}
\item For $p\in W(p_{0})\cap\Omega$, $K_{B}(p,p)\simeq\Pi_{i=1}^{n}F(L_{i}^{p,\delta(p)},p,\delta_{\partial\Omega}(p))$,
where $\delta_{\partial\Omega}(p)$ is the distance from $p$ to $\partial\Omega$.
\item For $p_{1},p_{2}\in W(p_{0})\cap\Omega$, for all integer $N$, there
exists a constant $C_{N}$ depending on $\Omega$ and $N$, such that
for all lists $\Lrond_{Z_{1}}=\{L_{1}^{1},\ldots,L_{1}^{k}\}$ (resp
$\Lrond_{Z_{2}}=\{L_{2}^{1},\ldots L_{2}^{k'}\}$) of length $k\leq N$
(resp. $k'\leq N$) with $L_{1}^{j}\in\brond(\pi(p_{1}),\tau)\cup\{N\}$
(resp. $L_{2}^{j}\in\overline{\brond(\pi(p_{1}),\tau)\cup\{N\}}$),
we have 
\[
\left|\Lrond_{Z_{1}}\overline{\Lrond_{Z_{2}}}K_{B}(Z_{1},Z_{2})(p_{1},p_{2})\right|\leq C_{N}\prod_{i=1}^{n}F(L_{i}^{\pi(p_{1}),\tau},\pi(p_{1}),\tau)^{1+l_{i}/2},
\]
 where $\tau=\delta_{\partial\Omega}(p_{1})+\delta_{\partial\Omega}(p_{2})+\gamma(\pi(p_{1}),\pi(p_{2}))$,
$\gamma(\pi(p_{1}),\pi(p_{2}))$ is the pseudo-distance from $\pi(p_{1})$
to $\pi(p_{2})$ associated to the structure of homogeneous space
and $l_{i}$ is the number of times the vector fields $L_{i}^{\pi(p_{1}),\tau}$
or $\overline{L_{i}^{\pi(p_{1}),\tau}}$ appear in the union of the
lists $\Lrond_{Z_{1}}$ and $\Lrond_{Z_{2}}$.
\end{enumerate}
\end{stthm}
\begin{cor*}
Suppose $\Omega$ satisfies the hypothesis of \thmref{stron-geom-sep-complete-local-domain}.
Let $D$ be the local domain considered in \thmref{stron-geom-sep-complete-local-domain}.
Then the Bergman kernel $K_{D}(z,w)$ of $D$ satisfy all the estimates
stated in the Theorem at any point of its boundary.
\end{cor*}

Using the methods of Section 5 of \cite{Char-Dup-Barcelone} the following
result on invariant metrics is easily proved:
\begin{stthm}
Suppose $\Omega$ is completely geometrically separated at $p_{0}\in\partial\Omega$.
Let us denote by $B_{\Omega}(z,L)$ (resp. $C_{\Omega}(z,L)$, resp.
$K_{\Omega}(z,L)$) the Bergman (resp. Caratheodory, resp. Kobayashi)
metric of $\Omega$ at the point $z\in\Omega$. Then there exists
a neighborhood $V(p_{0})$ such that, for all vector fields $L\in E$ ($E$ being the vector space spaned
by the basis $\brond^{0}$ (see \defref{def-geometrically-separable-domain}),
$L=L_{\tau}+a_{n}N$, we have, for $q\in V(p_{0})\cap\Omega$,
\[
B_{\Omega}(q,L)\simeq C_{\Omega}(q,L)\simeq K_{\Omega}(q,L)\simeq F(L_{\tau},q,\delta(q))+\frac{\left|a_{n}\right|}{\delta_{\partial\Omega}(q)},
\]
where the constants in the equivalences depend only on the constant of geometric separation
and the data.
\end{stthm}
\begin{rem*}
The last point of \remref{directional-pseudobals} and this Theorem
show that the structure of homogeneous space we associate to a completely
geometrically separated domain is essentially unique.\end{rem*}
\begin{stthm}
\label{thm:extimates-Berg-Sezgo-compl-geom-sep}Suppose $\Omega$
is completely geometrically separated at every point of its boundary.
Then the following results hold:
\begin{enumerate}
\item Let $P_{B}$ be the Bergman projection of $\Omega$. Then:

\begin{enumerate}
\item for $1<p<+\infty$ and $s\geq0$, $P_{B}$ maps continuously the Sobolev
space $L_{s}^{p}(\Omega)$ into itself;
\item for $0<\alpha<+\infty$, $P_{B}$ maps continuously the Lipschitz
space $\Lambda_{\alpha}(\Omega)$ into itself;
\item for $0<\alpha<1/M$, $P_{B}$ maps continuously the Lipschitz space
$\Lambda_{\alpha}(\Omega)$ into the non-isotropic Lipschitz space
$\Gamma_{\alpha}(\Omega)$.
\end{enumerate}
\item Let $P_{S}$ be the Szeg\"o projection of $\Omega$. Then:

\begin{enumerate}
\item for $1<p<+\infty$ and $s\in\mathbb{N}$, $P_{S}$ maps continuously
the Sobolev space $L_{s}^{p}(\partial\Omega)$ into itself;
\item for $0<\alpha<+\infty$, $P_{S}$ maps continuously the Lipschitz
space $\Lambda_{\alpha}(\partial\Omega)$ into itself;
\item for $0<\alpha<1/M$, $P_{S}$ maps continuously the Lipschitz space
$\Lambda_{\alpha}(\partial\Omega)$ into the non-isotropic Lipschitz
space $\Gamma_{\alpha}(\partial\Omega)$.
\end{enumerate}
\end{enumerate}
\end{stthm}
\begin{note*}
\quad\mynobreakpar
\begin{enumerate}
\item Statements (1) (c) and (2) (c) can be extended to all $\alpha>0$
with convenient definitions of the spaces $\Gamma_{\alpha}(\Omega)$
and $\Gamma_{\alpha}(\partial\Omega)$.
\item In view of \exaref{Example-PSH-adaped-to-geom-sep}, the previous
theorem applies in particular for all lineally convex domains of finite type.
\end{enumerate}
\end{note*}
\begin{cor*}
Suppose $\Omega$ satisfy the hypothesis of \thmref{stron-geom-sep-complete-local-domain}.
Let $D$ be the local domain considered in \thmref{stron-geom-sep-complete-local-domain}.
Then all the results stated for $\Omega$ in the previous Theorem
are valid for $D$.
\end{cor*}
Using an idea of M. Machedon \cite{Machedon-88-one-degenerate} we
deduce local estimates for the Szeg\"o projection:
\begin{stthm}
Suppose $\Omega$ satisfies the hypothesis of \thmref{stron-geom-sep-complete-local-domain}.
Let $P_{S}$ be the Szeg\"o projection of $\Omega$. Then if $f$ is
a $L^{2}(\partial\Omega)$ function which is locally near $p_{0}$ in
the Sobolev space $L_{s}^{p}$, $1<p<+\infty$ and $s\in\mathbb{N}$,
(resp. in the Lipschitz space $\Lambda_{\alpha}$, $0<\alpha<1/M$)
then its projection $P_{S}(f)$ is locally near $p_{0}$ in $L_{s}^{p}$
(resp. in the non-isotropic Lipschitz space $\Gamma_{\alpha}$). In
particular this applies if the Levi form of $\Omega$ is locally diagonalizable
at $p_{0}$.\end{stthm}
\begin{proof}
if $f\in L^{2}(\Omega)$ and if $\chi\in\MR C^{\infty}(\partial\Omega)$
has compact support in a sufficiently small neighborhood of $p_{0}$
and $\chi=1$ in a neighborhood of $p_{0}$, then the subelliptic
estimates for $\square_{b}$ and Kohn's theory (\cite{Kohn-Estimates-d-bar-b-85,Kohn-Nirenberg-1965})
implies $P_{S}((1-\chi)f)$ is $\MR C^{\infty}$ near $p_{0}$, and,
denoting $P_{S}^{D}$ the Szeg\"o projection of $D$, $(P_{S}-P_{S}^{D})(\chi f)$
is $\MR C^{\infty}$ in a neighborhood of $p_{0}$ (see also \cite{Kang-approximation-Szego-Kernels});
the result follows thus the previous Corollary.
\end{proof}

\subsection{A guide for the proofs of the results of \secref{Analytic-results-geom-sep}}

Let $U$ be a neighborhood of $\partial\Omega$ where we can define
a projection $\pi$ onto $\partial\Omega$ using the integral curve
of the real normal to $\rho$. We will always suppose that $V(p_{0})\subset U$.

The two notions of ``weak homogeneous space'' and ``adapted
pluri-subharmonic function'' plays a crucial role in \cite{Charpentier-Dupain-Geometery-Finite-Type-Loc-Diag,Charpentier-Dupain-Szego-Barcelone}:
\begin{sddefn}
We say that the domain $\Omega$ satisfies the hypothesis of ``weak
homogeneous space'' at a boundary point $p_{0}$ of finite type $\tau$ if
there exist two neighborhoods $V(p_{0})$ and $W(p_{0})\Subset V(p_{0})$
and a constant $K$ such that:
\begin{enumerate}
\item There exists $\delta_{0}>0$ such that, for every $p\in W(p_{0})$,
$\forall\delta\in[-\frac{1}{3}\rho(p),\delta_{0}]$, there exists
a basis of vector fields tangent to $\rho$ in $V(p_{0})$, $\brond(p,\delta)$,
for which there exists a $K$-adapted coordinate system
\item There exists two constants $C$ and $c_{0}$, depending on $K$ and
$\tau$, such that, for $c\leq c_{0}$, the sets $B^{c}\left(\brond(p,\delta),p,\delta\right)$
(associated to the coordinate system), $B_{\MR{C}}^{c}\left(\brond(p,\delta),p,\delta\right)$
and $B_{\mathrm{exp}}^{c}\left(\brond(p,\delta),p,\delta\right)$
satisfy, for all $p\in W(p_{0})\cap\bar{\Omega}$
and all $\delta\in[-\frac{1}{3}\rho(p),\delta_{0}]$, the following conditions:
\begin{enumerate}
\item for $q\in B_{0}^{c}(p,\delta)$, $B_{0}^{c}(\brond(q,\delta),q,\delta)\subset B_{1}^{c}(\brond(p,\delta),p,C\delta)$,
where $B_{0}^{c}$ and $B_{1}^{c}$ denotes one of the sets $B^{c}$,
$B_{\MR{C}}^{c}$ or $B_{\mathrm{exp}}^{c}$.
\item $\mathrm{Vol}\left(B_{0}^{c}(\brond(p,2\delta),p,2\delta)\right)\leq C\mathrm{Vol}\left(B_{0}^{c}(\brond(p,\delta),p,\delta)\right)$.
\end{enumerate}
\end{enumerate}
\end{sddefn}

Note that, in this Definition the weights $F_{i}$ are defined with
$M=M'(\tau)$.
\begin{sddefn}
\label{def:first-definition-PSH-function-adapted-to-fam-basis}Let
$\brond=\{L_{1},\ldots,L_{n-1}\}$ be a basis of vector fields tangent
to $\rho$ in a neighborhood $V(p_{0})$ of a boundary point $p_{0}$
and $0<\delta\leq\delta_{0}$. We say that a pluri-subharmonic function
$H\in\mathrm{PSH}(\Omega)$ is \emph{$(p_{0},K,c,\delta)$-adapted
to this basis} $\brond$ if the following properties are satisfied:

$\left|H\right|\leq1$ in $\Omega$, and, for all point $p\in W(p_{0})\cap\bar{\Omega}$,
$\rho(p)\geq-3\delta$, the two following inequalities are verified
for points $q\in B_{\MR{C}}^{c}(\brond,p,\delta)\cap\Omega$:
\begin{enumerate}
\item For all $L=\sum_{i=1}^{n}a_{i}L_{i}$, $a_{i}\in\mathbb{C}$,
\[
\left\langle \partial\bar{\partial}H,L,\overline{L}\right\rangle \geq\frac{1}{K}\sum_{i=1}^{n}\left|a_{i}\right|^{2}F(L_{i},p,\delta).
\]
\item For $\Lrond\in\Lrond_{3}(\brond\cup\{N\})$,
\[
\left|\Lrond H\right|\leq K\prod_{L\in\Lrond}F(L,p,\delta)^{1/2}.
\]
\end{enumerate}
\end{sddefn}

Note that this Definition depends on the values of the vector fields
$L_{i}^{p}$ at points $q$ in $\Omega$. But, in the situation of
the applications below (i.e. with a finite type hypothesis) it can
be shown that it depends only (up to uniform constants) on the restriction
of the basis on $\partial\Omega$.\smallskip{}

The following Proposition follows from the work in \cite{Charpentier-Dupain-Geometery-Finite-Type-Loc-Diag,Charpentier-Dupain-Szego-Barcelone}:
\begin{spprop}\label{prop:proposition-6-1}
Let $\Omega$ be a bounded pseudo-convex domain and $p_{0}$ be a
boundary point of finite type (resp. a bounded pseudo-convex domain
of finite type). Then, if $\Omega$ satisfies the hypothesis of {}``weak
homogeneous space'' at $p_{0}$ (resp. at every point of its boundary)
and if there exists a pluri-subharmonic function $\MR{H}_{\delta}$
adapted to $\brond(p,\delta)$ for all $p\in W(p_{0})\cap\bar{\Omega}$
and all $\delta\in[-\frac{1}{3},\delta_{0}]$ (resp. if this property
holds at every point $p_{0}$ of $\partial\Omega$) then the conclusions
of \thmref{estimate-bergman-kernel-compl-geom-sep} (resp. \thmref{extimates-Berg-Sezgo-compl-geom-sep}) hold.
\end{spprop}
To prove Theorems \ref{thm:estimate-bergman-kernel-compl-geom-sep}
and \ref{thm:extimates-Berg-Sezgo-compl-geom-sep} it suffices then
to use the properties of extremal bases and to note the two following
facts:
\begin{enumerate}
\item The existence of extremal bases and adapted coordinate systems for
points of $\partial\Omega\cap W(p_{0})$ allows us to define bases
and coordinate systems for points inside $\Omega$ (see \remref{extremal-basis-inside-vs-boundary});
\item if $p_{1}\in W(p_{0})\cap\Omega$, $p=\pi(p_{1})$, the sets $\widetilde{B_{0}^{c}}(\brond(p,\delta),p_{1},\delta)$,
$-\frac{1}{3}\rho(p_{1})<\delta\leq\delta_{0}$, defined by $q\in\widetilde{B_{0}^{c}}(\brond(p,\delta),p_{1},\delta)$
if $\pi(q)\in B_{0}^{c}(\brond(p,\delta),p,\delta)$ and $\left|\rho(q)-\rho(p_{1})\right|<c\delta$
induce a structure of {}``weak homogeneous space''.
\end{enumerate}

\subsection{Main articulations of the proof of \propref{proposition-6-1}}

In Section 2 of \cite{Charpentier-Dupain-Geometery-Finite-Type-Loc-Diag}
we showed that if the Levi form is locally diagonalizable then the
local hypothesis of the Proposition is satisfied, and in \cite{Charpentier-Dupain-Szego-Barcelone,Charpentier-Dupain-Geometery-Finite-Type-Loc-Diag},
even if the statements are given in the case of a locally diagonalizable
Levi form, the proofs of the estimates on the Bergman and Szeg\"o projections
are done only using the hypothesis of the Proposition. We just give
here the main articulations of the proofs:
\begin{itemize}
\item The Bergman kernel estimates on the diagonal is done using Theorem
6.1 of \cite{Catlin-Bergman-dim-2} and the change of coordinates
$\Phi_{p}$ adapted to the basis $\brond(p,\delta(p))$.
\item The estimates on the derivatives of the Bergman kernel outside the
diagonal follow the methods developed by A. Nagel, J. P. Rosay, E.
M. Stein and S. Wainger \cite{N-R-S-W-Bergman-dim-2} and J. Mc Neal
\cite{McNeal-Bergman-C2-hors-diag} for the pseudo-convex domains
of finite type in $\mathbb{C}^{2}$, and used for some generalizations
(see the introduction) in particular by J. Mc Neal \cite{McNeal-convexes-94}
in the case of convex domains. It consists to obtain uniform local
estimates for the Neumann operator $\mathcal{N}$ and then to apply
the ideas developed by N. Kerzman \cite{Kerzman-Bergman} in the study
of the strictly pseudo-convex case. This requires scaling.\\
The starting point is to write the Bergman kernel $K_{B}^{\Omega}$
using the Bergman projection. More precisely, if $\psi_{\zeta}$ is
a radial function centered at $\zeta$ with compact support in $\Omega$
and of integral $1$, and $P_{B}^{\Omega}$ is the Bergman projection
of $\Omega$, then $D^{\mu}\bar{D}^{\nu}K_{B}^{\Omega}(w,\zeta)=D_{w}^{\mu}P_{B}^{\Omega}(\bar{D}_{\zeta}^{\nu}\psi_{\zeta})(w)$.
Then, $P_{B}^{\Omega}$ being related to the $\bar{\partial}$-Neumann
problem by the formula $P_{B}^{\Omega}=\mathrm{Id}-\vartheta\mathcal{N}\bar{\partial}$,
where $\vartheta$ is the formal adjoint to $\bar{\partial}$ and
$\mathcal{N}$ the inverse operator of $\bar{\partial}\bar{\partial}^{*}+\bar{\partial}^{*}\bar{\partial}$,
the estimates on $P_{B}^{\Omega}$ are obtained via estimates on $\mathcal{N}$.
To obtain these estimates, we use the theory developed by J. J. Kohn
and L. Nirenberg \cite{Kohn-Nirenberg-1965} which gives local Sobolev
estimates for $\mathcal{N}$ if there exists a local sub-elliptic
estimates for the $\bar{\partial}$-Neumann problem and the famous
work of D. Catlin (\cite{Catlin-Est.-Sous-ellipt.}), where it is
proved that the existence of an adapted pluri-subharmonic function
implies the existence of a sub-elliptic estimates for the $\bar{\partial}$-Neumann
operator.\\
The study of the Bergman kernel is not directly done in $\Omega$
but in $\Phi_{p}(\Omega)$, where $\Phi_{p}$ is a coordinate system
adapted to the basis $\brond(p,\delta_{\partial\Omega}(p)+\delta_{\partial\Omega}(q)+\gamma(\pi(p),\pi(q)))$,
where $\gamma$ is the pseudo-distance on $\partial\Omega$. One difficulty
is to see that all the constants appearing in the estimates and all
the domains where the estimates are done are uniformly controlled.
\item The estimates for the Bergman and Szeg\"o projectors are obtained adapting
the methods developed by J. Mc Neal and E. M. Stein in \cite{McNeal-Stein-Bergman,McNeal-Stein-Szego}
(and also \cite{N-R-S-W-Bergman-dim-2}), related, in particular,
to the theory of non isotropic smoothing operators, to non convex
domains.\end{itemize}
\begin{rem*}
The results on the Szeg\"o projection are thus obtained adapting the
theory of NIS operators to our settings. The $\Lambda_{\alpha}$ estimates,
for example, for the domains considered by M. Derridj in \cite{Derridj-Holder-blocs-1999}
can also be obtained using the estimate for $\square_{b}$ of Derridj's
paper, the estimate on the Bergman projection derived from the fact
that these domains are completely geometrically separated and the
results on the comparison of the Bergman and Szeg\"o projection obtained
by K. D. Koenig in \cite{Koenig-Comparing-Bergman-Szego-Math-Ann}.
\end{rem*}

\section{Examples and additional remarks}

\subsection{The lineally convex case\label{sec:Examples-The-lineally-convex-case}}

In this Section we show, with some details the statements made on
lineally convex domains in \exaref{Example-extremal-basis}, \exaref{Example-geom-sep-domains}
and \exaref{Example-PSH-adaped-to-geom-sep}.

Suppose $\Omega=\left\{ \rho<0\right\} $ is lineally convex near
$p_{0}\in\partial\Omega$, a point of finite type, and $W$ is a small
neighborhood of $p_{0}$. $\left(Z_{i}\right)_{i}$ is a coordinate
system centered at $p_{0}$ such that $Z_{n}$ is the complex normal
to $\partial\Omega$ at $p_{0}$, and $\frac{\partial\rho}{\partial Z_{n}}\simeq1$
in $W$.

We begin with the statement in \exaref{Example-extremal-basis} (1).
Let $p\in\partial\Omega\cap W$ and $\delta>0$. Let $\left(z_{i}\right)_{i}$
be the $\delta$-extremal basis (considered as a coordinate system)
at $p$ defined by M. Conrad in \cite{Conrad_lineally_convex} (the
main results concerning this basis are summarized in \cite{Diederich-Fischer_Holder-linally-convex}),
which is centered at $p$. To be coherent with our previous notations,
we suppose that the complex normal to $\partial\Omega$ at $p$ is
$z_{n}$ (in M. Conrad paper this normal is $z_{1}$).

To each vector $v=\left(a_{1},\ldots,a_{n-1},0\right)\in\mathbb{C}^{n}$
we associate the $(1,0)$-vector field, tangent to $\rho$,
\begin{equation}
L_{v}=\sum_{i=1}^{n-1}a_{i}\frac{\partial}{\partial z_{i}}+\beta_{v}\frac{\partial}{\partial Z_{n}}:=V+\beta_{v}\frac{\partial}{\partial Z_{n}}\label{eq:lineal-cvx-vector-fields-assoc-extr-basis}
\end{equation}
 (thus $\beta_{v}=-V(\rho)\left(\frac{\partial\rho}{\partial Z_{n}}\right)^{-1}$).

If $v_{i}=\left(\delta_{k}^{i}\right)_{1\leq k\leq n}$, $1\leq i\leq n-1$,
we denote $L_{i}=L_{v_{i}}=\frac{\partial}{\partial z_{i}}+\beta_{i}\frac{\partial}{\partial Z_{n}}$.
Note that the vector fields $L_{i}$ depend on $p$ and $\delta$
($L_{i}=L_{i}(p,\delta)$) and are the vector fields of a basis of the complex tangent
space to $\rho$ in $W$.
\begin{spprop}
\label{prop:extremal-basis-lineally-convex}There exists a constant
$K$ such that, for all $p\in\partial\Omega\cap W$ and all $\delta\leq\delta_{0}$,
$\delta_{0}$ small enough, the basis $\left(L_{i}(p,\delta)\right)_{i}$
is $(K,\delta)$-extremal at $p$.\end{spprop}
\begin{proof}
$p$ and $\delta$ being fixed, we drop them in the notations. First
we express the weights $F(L_{v},p,\delta)$ in terms of the vector
field $V$ of (\ref{eq:lineal-cvx-vector-fields-assoc-extr-basis}).
\begin{lem*}
Let $\Lrond$ be a list composed of $\alpha$ $L_{v}$ and $\beta$
$\overline{L_{v}}$, $\left\Vert v\right\Vert \leq1$. Then
\[
\Lrond\left(\partial\rho\right)=2V^{\alpha}\bar{V}^{\beta}(\rho)+\sum_{\alpha'+\beta'<\alpha+\beta}*V^{\alpha'}\bar{V}^{\beta'}(\rho)
\]
where $*$ are functions of $\MR{C}^{2m-(\alpha+\beta)}$ norm uniformly
bounded in $p$ and $\delta$.\end{lem*}
\begin{proof}
Look first at $c_{v\bar{v}}=2\left[L_{v},\overline{L_{v}}\right](\partial\rho)$:
$c_{vv}=-2\overline{L_{v}}(\beta_{v})\frac{\partial\rho}{\partial Z_{n}}=2\overline{L_{v}}V(\rho)+*V(\rho)=2\bar{V}V(\rho)+*V(\rho)$.
The Lemma is then proved by induction.\end{proof}
\begin{cor*}
$F\left(L_{v},p,\delta\right)\simeq\sum\left|a_{i}\right|^{2}F(L_{i},p,\delta)$
uniformly in $p$ and $\delta$.\end{cor*}
\begin{proof}
It suffices to prove this formula when $\left\Vert v\right\Vert =1$.
By the Lemma
\begin{eqnarray*}
F\left(L_{v},p,\delta\right) & \simeq & \sum_{2\leq\alpha+\beta\leq m}\left|\frac{V^{\alpha}\bar{V}^{\beta}(\rho)}{\delta}\right|^{\frac{2}{\alpha+\beta}}\\
 & = & \sum_{2\leq\alpha+\beta\leq m}\left|\frac{\frac{\partial^{\alpha+\beta}}{\partial\lambda^{\alpha}\partial\bar{\lambda}^{\beta}}(\rho)(p+\lambda v)_{|\lambda=0}}{\delta}\right|^{\frac{2}{\alpha+\beta}}\\
 & \simeq & \left(\frac{2}{\tau(p,v,\delta)}\right)^{2},
\end{eqnarray*}
where $\tau(p,v,\delta)$ is Conrad's notation.

Using properties (iii) and (iv) of Proposition 3.1 of \cite{Diederich-Fischer_Holder-linally-convex}
we get
\[
F\left(L_{v},p,\delta\right)\simeq\sum_{2}^{m}\frac{\left|a_{i}\right|^{2}}{\tau(p,v_{i},\delta)^{2}}.
\]
 As all constants are uniform in $p$ and $\delta$ the Corollary
is proved.
\end{proof}
To finish the proof of \propref{extremal-basis-lineally-convex} we
have to prove property EB$_{\text{2}}$ of \defref{basis-extremal}.
For example, let us look at the bracket $[L_{i},\overline{L_{j}}]$:
\[
[L_{i},\overline{L_{j}}]=\left(-\frac{\partial}{\partial\bar{z}_{j}}+\bar{\beta}_{j}\frac{\partial}{\partial\overline{Z_{n}}}\right)\left(\beta_{i}\right)\frac{\partial}{\partial Z_{n}}+\left(\frac{\partial}{\partial z_{i}}+\beta_{i}\frac{\partial}{\partial Z_{n}}\right)\left(\overline{\beta_{i}}\right)\frac{\partial}{\partial\overline{Z_{n}}}=a\frac{\partial}{\partial Z_{n}}+b\frac{\partial}{\partial\overline{Z_{n}}}.
\]
 Let $\Lrond\in\Lrond_{M}\left(L_{1},\ldots,L_{n-1}\right)$. As,
for all $k$, $F_{k}^{-1/2}\geq\delta$ and $\frac{\partial}{\partial Z_{n}}=\sum\alpha_{i}\frac{\partial}{\partial z_{i}}$
with $\alpha_{i}$ uniformly bounded in $\MR{C}^{M}$ norm, it is
enough to show that
\[
\left(\left|\Lrond a\right|+\left|\Lrond b\right|\right)(p)\lesssim\delta F^{\alpha/2}(p,\delta)F_{i}^{1/2}(p,\delta)F_{j}^{1/2}(p,\delta).
\]
 If $\left|\Lrond\right|=0$, $a(p)=\frac{\partial^{2}\rho}{\partial z_{i}\partial\overline{z_{j}}}(0)\left(\frac{\partial\rho}{\partial Z_{n}}\right)^{-1}(p)$
and, if $\left|\Lrond\right|=1$, $\LB{k}a(p)=\frac{\partial^{3}\rho}{\partial\PBG{z_{k}}\partial z_{i}\partial\overline{z_{j}}}(0)\left(\frac{\partial\rho}{\partial Z_{n}}\right)^{-1}(p)+*\frac{\partial^{2}\rho}{\partial\PBG{z_{k}}\partial\overline{z_{j}}}$.
Thus, in those cases, the result follows from Lemma 3.2 of M. Conrad's
paper \cite{Conrad_lineally_convex} which states
\begin{equation}
\left|\frac{\partial^{\alpha+\beta}\rho}{\partial z^{\alpha}\partial\overline{z}^{\beta}}(p)\right|\lesssim\delta\prod\left(\frac{1}{\tau(p,v_{i},\delta)}\right)^{\alpha_{i}+\beta_{i}}\simeq\delta F(p,\delta)^{(\alpha+\beta)/2},\label{eq:maj-der-rho-linea-cvx}
\end{equation}
 the last equivalence resulting of the proof of the previous Corollary.
The case of a general $\Lrond$ is easily done similarly.
\end{proof}

\medskip{}

Now let us prove the statement made in \exaref{Example-PSH-adaped-to-geom-sep}
(1). The construction of the adapted plurisubharmonic function is
inspired by the McNeal's construction for convex domains, using support
function, written in \cite{McNeal-unif-subel-est-convex}. We use
the support function for lineally convex domains constructed by J.
E. Fornaess and K. Diederich in \cite{Diederich-Fornaess-Support-Func-lineally-cvx}.
The right behavior in the normal direction is obtained, as in \secref{PSH-func-strong-geom-sep},
adding the functions $Ke^{\rho/\delta}$ and $K\left|z\right|^{2}$.

Consider the support function constructed by K. Diederich and J. E.
Fornaess in \cite{Diederich-Fornaess-Support-Func-lineally-cvx} at
the point $p$:
\begin{eqnarray*}
S_{p}(z_{1},\ldots,z_{n}) & = & -\varepsilon\sum_{j=2}^{2m}M^{2^{j}}\sigma_{j}\sum_{\left|\alpha\right|=j,\,\alpha_{n}=0}\frac{1}{\alpha!}\frac{\partial^{j}\rho(0)}{\partial z^{\alpha}}z^{\alpha}\\
 &  & +z_{n}\left(\frac{1}{1-A_{p}(z)}\right)+K_{0}\left(\frac{z_{n}}{1-A_{p}(z)}\right)^{2},
\end{eqnarray*}
 where $A_{p}$ is a $\MR{C}^{\infty}$ function, uniformly bounded
(in $p$), such that $A_{p}(0)=0$. Shrinking $W(p_{0})$ if necessary,
$S_{p}$ is uniformly bounded on $W(p_{0})$.

Then there exists a constant $M_{0}$ ($>8n$ and independent of $p$
and $\delta$) such that, if $S=\frac{M_{0}\Re\mathrm{e}(S_{p})}{\delta}$,
we have:
\begin{enumerate}
\item $\Re\mathrm{e}(S)\leq0$;
\item $\Re\mathrm{e}\left(S(z)\right)\leq-n$, if there exists $i<n$ such
that $\left|z_{i}\right|\geq F_{i}(p,\delta)^{-1/2}$;
\item $-1/4\leq\Re\mathrm{e}\left(S(z)\right)$ if $z\in cP(p,\delta)=\left\{ z\mbox{ such that }\left|z_{i}\right|\leq cF_{i}(p,\delta)^{-1/2},\, i=1,\ldots,n\right\} $.
\end{enumerate}

Let $F$ be the function defined by $F(z)=\sum_{i=1}^{n-1}F_{i}(p,\delta)\left|z_{i}\right|^{2}$,
and $\chi$ be the convex function such that $\chi\equiv0$ on $]-\infty,-1[$
and $\chi(x)=e^{x-1}-1$ on $]-1,+\infty[$. Define then 
\[
H_{1}=\chi\left(F+S-\frac{n}{\delta^{2}}\left|z_{n}\right|^{2}\right).
\]
 Clearly $H_{1}\equiv0$ on a neighborhood of the boundary of $P(p,\delta)$.
Thus we denote by $H$ the function equal to $H_{1}$ in $P(p,\delta)$
and $0$ outside. Then:
\begin{enumerate}
\item $\mathrm{Supp}(H)\subset P(p,\delta)$;
\item $\left|H\right|\leq C_{0}$;
\item On $P(p,\delta)$ we have
\[
\left\langle \partial\bar{\partial}H;\sum_{i=1}^{n-1}a_{i}\frac{\partial}{\partial z_{i}},\overline{\sum_{i=1}^{n-1}a_{i}\frac{\partial}{\partial z_{i}}}\right\rangle \geq\chi'\left(F+S-\frac{n}{\delta^{2}}\left|z_{n}\right|^{2}\right)\sum_{i=1}^{n-1}\left|a_{i}\right|^{2}F_{i}(p,\delta).
\]
\end{enumerate}

We now estimate $\left\langle \partial\bar{\partial}H;L,\bar{L}\right\rangle $
in $P(p,\delta)$ for a vector field $L=_{i=1}^{n}b_{i}L_{i}^{0}$
where $L_{i}^{0}=\frac{\partial}{\partial Z_{i}}-\beta_{i}^{0}\frac{\partial}{\partial Z_{n}}$,
for $i<n$, and $L_{n}^{0}=N$ the complex normal vector (recall that
the extremal bases are linear combinations of these $L_{i}^{0}$).
Denote $L=L_{\tau}+b_{n}N$, so that $L_{\tau}$ is tangent to $\rho$.
Then
\begin{equation}
\left\langle \partial\bar{\partial}H;L,\bar{L}\right\rangle =\left\langle \partial\bar{\partial}H;L_{\tau},\overline{L_{\tau}}\right\rangle +2\Re\mathrm{e}\left(\overline{b_{n}}\left\langle \partial\bar{\partial}H;L_{\tau},\bar{N}\right\rangle \right)+\left|b_{n}\right|^{2}\left\langle \partial\bar{\partial}H;N,\bar{N}\right\rangle .\label{eq:expression-D-Dbar-H-linea-cvx}
\end{equation}

The last term of the right hand side of this equality is $\geq-\mathrm{O}\left(\left|b_{n}\right|^{2}\chi'\left(F+S-\frac{n}{\delta^{2}}\left|z_{n}\right|^{2}\right)\frac{1}{\delta^{2}}\right)$,
and, if $\left(L_{i}\right)$ is the $\delta$-extremal basis at $p$
and $L_{\tau}=\sum_{i=1}^{n-1}a_{i}L_{i}$, we have
\begin{eqnarray*}
\left\langle \partial\bar{\partial}H;L_{\tau},\bar{N}\right\rangle  & = & \left\langle \partial\bar{\partial}H;\sum_{i=1}^{n-1}a_{i}\frac{\partial}{\partial z_{i}},\bar{N}\right\rangle +\left\langle \partial\bar{\partial}H;\left(\sum a_{i}\beta_{i}\right)\frac{\partial}{\partial Z_{n}},\bar{N}\right\rangle .
\end{eqnarray*}
 Using (\ref{eq:maj-der-rho-linea-cvx}) the first term of the
right hand side of this equality is $\mathrm{O}\left(\frac{1}{\delta}\chi'\left(F+S-\frac{n}{\delta^{2}}\left|z_{n}\right|^{2}\right)\right)$,
and, using also the fact that $\beta_{i}(p)=0$ implies
$\beta_{i}=\mathrm{O}\left(\delta F_{i}(p,\delta)^{1/2}\right)$, the second term is
$\mathrm{O}\left(\frac{1}{\delta}\chi'\left(F+S-\frac{n}{\delta^{2}}\left|z_{n}\right|^{2}\right)\sum\left|a_{i}\right|F_{i}(p,\delta)^{1/2}\right)$. Notice that, by extremality,
$\sum\left|a_{i}\right|F_{i}(p,\delta)^{1/2}\simeq F(L_{\tau},p,\delta)^{1/2}$,
thus, there exists a constant $K_{1}$ such that
\[
2\Re\mathrm{e}\left(\overline{b_{n}}\left\langle \partial\bar{\partial}H;L_{\tau},\bar{N}\right\rangle \right)+\left|b_{n}\right|^{2}\left\langle \partial\bar{\partial}H;N,\bar{N}\right\rangle \geq-K_{1}\chi'\left(F+S-\frac{n}{\delta^{2}}\left|z_{n}\right|^{2}\right)\left(\frac{\left|b_{n}\right|^{2}}{\delta^{2}}+\frac{\left|b_{n}\right|}{\delta}F(L_{\tau},p,\delta)^{1/2}\right).
\]

Let us now look at the first term of the right hand side of 
(\ref{eq:expression-D-Dbar-H-linea-cvx}):
\begin{eqnarray*}
\left\langle \partial\bar{\partial}H;L_{\tau},\overline{L_{\tau}}\right\rangle  & = & \chi'\left(F+S-\frac{n}{\delta^{2}}\left|z_{n}\right|^{2}\right)\left\langle \partial\bar{\partial}\left(F+S-n\frac{\left|z_{n}\right|^{2}}{\delta^{2}}\right);L_{\tau},\overline{L_{\tau}}\right\rangle +\\
 &  & +\chi''\left(F+S-\frac{n}{\delta^{2}}\left|z_{n}\right|^{2}\right)\left|L_{\tau}\left(F+S-n\frac{\left|z_{n}\right|^{2}}{\delta^{2}}\right)\right|^{2}\\
 & = & A+B.
\end{eqnarray*}
 Shrinking $W(p_{0})$ if necessary, we have
\begin{eqnarray*}
A & \geq & \chi'\left(F+S-\frac{n}{\delta^{2}}\left|z_{n}\right|^{2}\right)\left[\frac{1}{2}\sum_{i=1}^{n-1}\left|a_{i}\right|^{2}F_{i}(p,\delta)-\frac{n}{\delta^{2}}\left|\left\langle \partial\bar{\partial}\left|z_{n}\right|^{2};L_{\tau},\overline{L_{\tau}}\right\rangle \right|\right]\\
 & \geq & \chi'\left(F+S-\frac{n}{\delta^{2}}\left|z_{n}\right|^{2}\right)\left[\frac{1}{2}\sum_{i=1}^{n-1}\left|a_{i}\right|^{2}F_{i}(p,\delta)-\frac{2n}{\delta^{2}}\left|\sum_{i=1}^{n-1}a_{i}\beta_{i}\right|^{2}\right].
\end{eqnarray*}
 To estimate $B$, write
\[
L_{\tau}\left(F+S-n\frac{\left|z_{n}\right|^{2}}{\delta^{2}}\right)=\sum_{i=1}^{n-1}a_{i}\frac{\partial}{\partial z_{i}}\left(F+S-n\frac{\left|z_{n}\right|^{2}}{\delta^{2}}\right)+\left(\sum_{i=1}^{n-1}a_{i}\beta_{i}\right)\frac{\partial}{\partial Z_{n}}\left(F+S-n\frac{\left|z_{n}\right|^{2}}{\delta^{2}}\right).
\]
 Then the first term of the right hand side of this equality is $\mathrm{O}\left(F(L_{\tau},p,\delta)^{1/2}\right)$
by extremality (use (\ref{eq:maj-der-rho-linea-cvx})), and
\[
\frac{\partial}{\partial Z_{n}}\left(F+S-n\frac{\left|z_{n}\right|^{2}}{\delta^{2}}\right)=\mathrm{O}\left(F(L_{\tau},p,\delta)^{1/2}\right)+\frac{M_{0}}{\delta}\frac{\partial}{\partial Z_{n}}\left(\frac{z_{n}}{1-A_{p}(z)}\right)+\mathrm{O}(1)-\frac{n}{\delta^{2}}\frac{\partial}{\partial Z_{n}}\left(z_{n}\overline{z_{n}}\right).
\]
 But, if $W(p_{0})$ is small enough, $\left|\frac{\partial}{\partial Z_{n}}\left(\frac{z_{n}}{1-A(z)}\right)\right|\in\left[1/2,3/2\right]$,
and, in $P(p,\delta)$, $\frac{1}{\delta^{2}}\left|\frac{\partial}{\partial Z_{n}}\left(z_{n}\overline{z_{n}}\right)\right|\leq\frac{2n}{\delta}$.

Thus, as $\chi''=\chi'$, for $\delta$ small enough, we have, by
the choose of $M_{0}$,
\[
\left\langle \partial\bar{\partial}H;L_{\tau},\overline{L_{\tau}}\right\rangle \geq\chi'\left(F+S-\frac{n}{\delta^{2}}\left|z_{n}\right|^{2}\right)\sum_{i=1}^{n-1}\left|a_{i}\right|^{2}F_{i}(p,\delta).
\]
 Using again the extremality of the basis $\left(L_{i}\right)$, we
conclude that
\[
\left\langle \partial\bar{\partial}H;L,\bar{L}\right\rangle \geq\alpha\chi'\left(F+S-\frac{n}{\delta^{2}}\left|z_{n}\right|^{2}\right)\left[F(L_{\tau},p,\delta)-K_{2}\frac{\left|b_{n}\right|^{2}}{\delta^{2}}+K_{2}\frac{\left|b_{n}\right|}{\delta}F(L_{\tau},p,\delta)^{1/2}\right],
\]
 and, using Cauchy-Schwarz inequality, we get
\[
\left\langle \partial\bar{\partial}H;L,\bar{L}\right\rangle \geq\beta\chi'\left(F+S-\frac{n}{\delta^{2}}\left|z_{n}\right|^{2}\right)\left[F(L_{\tau},p,\delta)-K_{3}\frac{\left|b_{n}\right|^{2}}{\delta^{2}}\right].
\]
In particular, on $cP(p,\delta)$, we have
\[
\left\langle \partial\bar{\partial}H;L,\bar{L}\right\rangle \geq\gamma F(L_{\tau},p,\delta)-K\frac{\left|b_{n}\right|^{2}}{\delta^{2}},
\]
 and $\left\langle \partial\bar{\partial}H;L,\bar{L}\right\rangle \geq-K\frac{\left|b_{n}\right|^{2}}{\delta^{2}}$
on $P(p,\delta)$.

If we note that choosing $c$ sufficiently small we have $F(L_{\tau},p,\delta)\simeq F(L_{\tau},q,\delta)$,
we get:
\begin{spprop}
There exist two constants $\gamma$ and $K$ depending only on the
data such that, if $L=L_{\tau}+b_{n}N=\sum_{i=1}^{n-1}b_{i}L_{i}^{0}+b_{n}N$,
we have $\left\langle \partial\bar{\partial H;L,\bar{L}}\right\rangle \geq-K\frac{\left|b_{n}\right|^{2}}{\delta^{2}}$,
and, if $q\in cP(p,\delta)$, we have $\left\langle \partial\bar{\partial H;L,\bar{L}}\right\rangle (q)\geq\gamma F(L_{\tau},q,\delta)-K\frac{\left|b_{n}\right|^{2}}{\delta^{2}}$.
\end{spprop}
To finish the construction of the plurisubharmonic function adapted
to the structure of $\Omega$ at $p$, as in the proof of \thmref{PSH-function-for-geom-separable-with-strong-ext-basis},
we have to add functions of the precedent type to get a local function.
Thus, we cover $\partial\Omega\cap W(p_{0})$ with a minimal system
of polydiscs $\frac{c}{2}P(p_{k}\delta)$, $p_{k}\in\partial\Omega\cap W(p_{0})$
and, then, there exists an integer $J$, independent of $\delta$,
such that every point of $\Omega$ belongs to at most $J$ polydiscs
$P(p_{k},\delta)$. Indeed, there exists a constant $C$ such that
\[
P\left(p,\frac{c}{C}\delta\right)\subset\frac{c}{2}P(p,\delta)\subset P(p,cC\delta)
\]
 and the polydiscs $P(p,\delta)$ are associated to a structure of
homogeneous space.

Consider $H=\sum H_{p_{k}}$ where the function $H_{p_{k}}$ is the
one considered in the previous Proposition relatively to the point
$p=p_{k}$ (notice that $\left\Vert H\right\Vert \leq JC_{0}$). Then,
shrinking eventually $W(p_{0})$ and choosing $\rho$ equivalent to
the distance to the boundary with a constant close to $1$, for all
point $q\in W(p_{0})\cap\left\{ 0>\rho>-\frac{c}{2}\delta\right\} $
there exists $k_{0}$ such that $q\in P(p_{k_{0}},\delta)$ and the
set $E(q)$ of index $k$ so that $q\in cP(p_{k},\delta)$ has at
most $J$ elements and we have
\[
\left\langle \partial\bar{\partial}H;L\bar{L}\right\rangle (q)\geq\gamma F(L_{\tau},q,\delta)-KJ\frac{\left|b_{n}\right|^{2}}{\delta^{2}}.
\]
 Moreover, without conditions on $q$, we have
\[
\left\langle \partial\bar{\partial}H;L\bar{L}\right\rangle (q)\geq-KJ\frac{\left|b_{n}\right|^{2}}{\delta^{2}},
\]
 and $\partial\bar{\partial}H(q)=0$ if $\rho(q)<-2\delta$.

We now evaluate $\left\langle \partial\bar{\partial}e^{\rho/\delta};L,\bar{L}\right\rangle $
in $W(p_{0})$:
\[
\left\langle \partial\bar{\partial}e^{\rho/\delta};L,\bar{L}\right\rangle \geq e^{\rho/\delta}\left[\frac{1}{\delta}\left(\frac{1}{2}\sum_{i,j=1}^{n-1}b_{i}\overline{b_{j}}c_{ij}^{0}+\Re\mathrm{e}\sum_{i=1}^{n-1}b_{i}\overline{b_{n}}\left\langle \partial\bar{\partial}\rho;L_{i}^{0},\bar{N}\right\rangle \right)+\frac{\left|b_{n}\right|^{2}}{\delta^{2}}\right]
\]
 $c_{ij}^{0}$ being the coefficient of the Levi form in the direction
$\left(L_{i}^{0},\overline{L_{j}^{0}}\right)$. As the level set of
$\rho$ are pseudo-convex (in $W(p_{0})$), we get
\[
\left\langle \partial\bar{\partial}e^{\rho/\delta};L,\bar{L}\right\rangle \geq e^{\rho/\delta}\left(\frac{1}{2}\frac{\left|b_{n}\right|^{2}}{\delta^{2}}-K_{1}\right).
\]

Consider now $\tilde{H}=H+K_{1}e^{\rho/\delta}+K_{2}\left|z\right|^{2}$,
for $K_{1}$ and $K_{2}$ large enough (independent of $\delta$).
Then $\tilde{H}$ is plurisubharmonic on $\Omega\cap W(p_{0})$, uniformly
bounded (with respect to $\delta$) and satisfies, on $W(p_{0})\cap\left\{ 0>\rho>-\frac{c}{2}\delta\right\} $
\[
\left\langle \partial\bar{\partial}\tilde{H};L,\bar{L}\right\rangle \geq\gamma F(L_{\tau},.,\delta)+\frac{\left|b_{n}\right|^{2}}{2\delta^{2}}.
\]
 To change $\frac{c}{2}\delta$ in $2\delta$ it suffices to apply
the relations between $F(.,.,\alpha\delta)$ and $F(.,.,\delta)$.

Finally, we extend $\tilde{H}$ to a bounded plurisubharmonic function
in $\Omega$ using the function $\vartheta_{1}$ of the end of the
proof of \thmref{PSH-function-for-geom-separable-with-strong-ext-basis}.

\subsection{Example of non geometrically separated domain\label{sec:example-non-geometrically-separated}}

The example presented here is the domain of $\mathbb{C}^{3}$ introduced
by G. Herbort in \cite{Herbort-logarithm}:
\[
\Omega=\left\{ z\in\mathbb{C}^{3}\mbox{ such that }\Re\mathrm{e}z_{1}+\left|z_{2}\right|^{6}+\left|z_{3}\right|^{6}+\left|z_{2}\right|^{2}\left|z_{3}\right|^{2}<0\right\} .
\]

Let $L_{i}^{0}=\frac{\partial}{\partial z_{i}}+\beta_{i}\frac{\partial}{\partial z_{1}},$
$i=2,3$, $\beta_{2}=-\left(6\left|z_{2}\right|^{4}\overline{z_{2}}+2\left|z_{3}\right|^{2}\overline{z_{2}}\right)$
and $\beta_{3}=-\left(6\left|z_{3}\right|^{4}\overline{z_{3}}+2\left|z_{2}\right|^{2}\overline{z_{3}}\right)$
so that $\left(L_{2}^{0},L_{3}^{0}\right)$ is a basis of $(1,0)$
tangent vector fields in a neighborhood of the origin.

The fact that this domain is not geometrically separated at the origin
is a consequence of the stronger following result:
\begin{spprop}
For all real constants $K$ and $C$, there exists $\delta_{0}>0$
such that for all $\delta$, $0<\delta\leq\delta_{0}$, there does
not exist a basis $\left(L_{1}^{\delta},L_{2}^{\delta}\right)$ of $(1,0)$
tangent vector fields in a neighborhood of the origin of $\MR{C}^{6}$
norm bounded by $C$ satisfying property EB$_{\text{1}}$ of \defref{basis-extremal},
for the constant $K$, at the origin.\end{spprop}
\begin{proof}
Let $L$ be a $(1,0)$ tangent vector field in a neighborhood of the
origin, and, with our usual notations, $F(L,0,\delta)=\sum_{\Lrond\in\Lrond_{6}(L)}\left|\frac{\Lrond(\partial\rho)}{\delta}\right|^{2/\left|\Lrond\right|}(0)$.
We write $L=aL_{2}^{0}+bL_{3}^{0}$.
\begin{lem*}
$F(L,0,\delta)\simeq\frac{\left|a(0)b(0)\right|}{\delta^{1/2}}+\left(\frac{1}{\delta}\right)^{1/3}$.\end{lem*}
\begin{proof}
Because $c_{L\bar{L}}=2\left[L,\bar{L}\right](\partial\rho)$, it
is easy to see that:
\begin{itemize}
\item $c_{L\bar{L}}(0)=Lc_{L\bar{L}}(0)=\bar{L}c_{L\bar{L}}(0)=0$:
\item $LLc_{L\bar{L}}(0)=\bar{L}\bar{L}c_{L\bar{L}}(0)=0$ and $L\bar{L}c_{L\bar{L}}(0)=\bar{L}Lc_{L\bar{L}}(0)=4\left|a(0)b(0)\right|^{2}$;
\item There exists a constant $C_{0}$ depending only of the $\MR{C}^{6}$
norm of $a$ and $b$ (i.e. of $L$) such that, if $\left|\Lrond\right|=3$,
$\left|\Lrond c_{L\bar{L}}(0)\right|\leq C_{0}\left|a(0)b(0)\right|$;
\item There exists a constant $\alpha_{0}$ depending only of the $\MR{C}^{6}$
norm of $L$, such that $F(L,0,\delta)\geq\alpha_{0}\delta^{-1/3}$.
Indeed, the origin being of type $6$, this follows from a result of T.
Bloom \cite{Bloom-81-finite-type-C3} and a compactness argument.
\end{itemize}
Then, the Lemma follows the fact that, for all $x\geq0$, $\left(\frac{x}{\delta}\right)^{2/5}\leq\frac{x}{\delta^{1/2}}+\left(\frac{1}{\delta}\right)^{1/3}$.
\end{proof}
We now finish the proof of the Proposition. Let $\delta$ be small
enough. Suppose that there exists a $(K,\delta)$- extremal
basis at the origin, $\left(L_{1}^{\delta},L_{2}^{\delta}\right)$,
the $\MR{C}^{6}$ norms of the vector fields bounded by $C$. Let
$L=\alpha L_{1}^{\delta}+\beta L_{2}^{\delta}$ and $L'=\alpha'L_{1}^{\delta}+\beta'L_{2}^{\delta}$
with $\alpha,\beta,\alpha',\beta'\in\mathbb{C}$ chosen so that $L(0)=L_{2}^{0}(0)$
and $L'(0)=L_{3}^{0}(0)$. Then, by extremality of $\left(L_{1}^{\delta},l_{2}^{\delta}\right)$
and the Lemma, we get
\[
\left|\alpha\right|^{2}F(L_{1}^{\delta},0,\delta)+\left|\beta\right|^{2}F(L_{2}^{\delta},0,\delta)\simeq_{K}F(L,0,\delta)\simeq_{C,K}\left(\frac{1}{\delta}\right)^{1/3}
\]
 and
\[
\left|\alpha'\right|^{2}F(L_{1}^{\delta},0,\delta)+\left|\beta'\right|^{2}F(L_{2}^{\delta},0,\delta)\simeq_{K}F(L',0,\delta)\simeq_{C,K}\left(\frac{1}{\delta}\right)^{1/3}.
\]
 Similarly, the extremality would imply
\[
F(L+L',0,\delta)\simeq_{K}\left|\alpha+\alpha'\right|^{2}F(L_{1}^{\delta},0,\delta)+\left|\beta+\beta'\right|^{2}F(L_{2}^{\delta},0,\delta)\lesssim_{C,K}\left(\frac{1}{\delta}\right)^{1/3}.
\]

But the Lemma gives $F(L+L',0,\delta)\simeq_{C}\frac{1}{\delta^{1/2}}$
which is a contradiction for $\delta$ small.
\end{proof}

\subsection{Additional remarks}

Let $\Omega$ be geometrically separated at $p\in\partial\Omega$.
In Definitions \ref{def:definition-polydissc-ext-basis} and \ref{def:definition-pseudo-balls-curves-exp}
we defined the pseudo-balls $B^{c}(p,\delta)$, $B_{\MR{C}}^{c}(p,\delta)$
and $B_{\mathrm{exp}}^{c}(p,\delta)$, which are equivalent by \propref{comparison-exp-balls-curves-polydisc},
and we expressed the Bergman kernel at $(p,p)$ with their volumes.

Let $\left(z_{i}\right)$ be the coordinate system adapted to the
extremal basis $\left(L_{i}\right)_{1\leq i\leq n-1}=\left(L_{i}^{p,\delta}\right)$
at $p$. $B(p,\delta)$ is defined (in the coordinate system) using
only the directions of the extremal basis. Let us now define a new
pseudo-ball using all the directions of the linear space generated
by the vector fields $L_{i}$ (i.e. the space $E_{0}$) (compare to
the last point of \remref{directional-pseudobals}):

For $\left|Z\right|=1$, $Z\in\mathbb{C}^{n}$, define $L_{Z}=\sum_{i=1}^{n-1}Z_{i}L_{i}+Z_{n}N$
and (in the coordinate system $\left(z_{i}\right)$)
\[
D_{Z}(p,\delta)=\left\{ \alpha Z\mbox{ such that }\left|\alpha\right|<cF(L_{Z},p,\delta)^{-1/2}\right\} 
\]
 and 
\[
D(p,\delta)=\bigcup_{\left|Z\right|=1}D_{Z}(p,\delta).
\]

Then, property EB$_{\text{1}}$ of extremality for $\left(L_{i}\right)$,
implies that these pseudo-balls $D(p,\delta)$ are equivalent (in the sense that
they define the same structure of homogeneous space) to the previous
ones. Indeed, if $z\in D_{Z}(p,\delta)$, 
\[
\left|z_{i}\right|=\left|\alpha z_{i}\right|\lesssim\left|Z_{i}\left(\sum\left|Z_{i}\right|^{2}F(L_{i}p,\delta)\right)^{-1/2}\right|\leq F(L_{i},p,\delta),
\]
and use Propositions \ref{prop:comparison-exp-balls-curves-polydisc}
and \ref{prop:pseudodistance-pseudoballs}. Conversely, if $z\in B^{c}(p,\delta)$,
$z\neq0$, and $Z=z/\left\Vert z\right\Vert $, then $F(L_{Z},p,\delta)\simeq\sum\frac{\left|z_{i}\right|^{2}}{\left\Vert z\right\Vert ^{2}}F(L_{i},p,\delta)\leq\frac{nc^{2}}{\left\Vert z\right\Vert ^{2}}$,
thus $\left\Vert z\right\Vert ^{2}\lesssim nc^{2}F(L_{Z},p,\delta)$
and we conclude as before.

\medskip{}

Note that this shows that, if $\Omega$ is completely geometrically
separated at $p_{0}\in\partial\Omega$ then the Bergman kernel $K(p,p)$
at a point $p$ near $p_{0}$ is equivalent to the inverse of the
volume of $D(p,\delta)$. 

If $\Omega$ is not geometrically separated at $p_{0}$ choosing a
coordinate system and a basis of tangent $(1,0)$ vector fields conveniently
associated (in a sense to be defined), one can always define a {}``pseudo-ball''
$D(p,\delta)$. 

Let us do this, for example, for the domains considered in the previous
Section, at the origin with the canonical coordinate system $\left(z_{i}\right)$ and
the vector fields $L_{i}^{0}$ (note that $\left(z_{i}\right)$ is
not adapted to the basis $\left(L_{i}^{0}\right)$ in the sense of
\defref{basis-and-coordinates-adapted} because, even if condition
(3) is satisfied, the conditions on the derivatives of $\rho$ are
not).

A direct computation shows that, at the point $p_{\delta}=\left(-\delta,0,0\right)$,
the volume of $D(p_{\delta},\delta)$ is $\mathrm{Vol}\left(D(p_{\delta},\delta)\right)\simeq\left(\delta^{3}\log\left(\frac{1}{\delta}\right)\right)^{-1}$
(uniformly in $\delta$), thus $B(p_{\delta},\delta)$ and $D(p_{\delta},\delta)$
are not equivalent, and the result of G. Herbort (\cite{Herbort-logarithm})
shows that the Bergman kernel $K$ of the domain satisfies $K(p_{\delta},p_{\delta})\simeq\mathrm{Vol}\left(D(p_{\delta},\delta)\right)^{-1}$.

Then it is natural to ask if, for that example, the {}``pseudo-balls''
$D(p,\delta)$ define a structure of homogeneous space. Unfortunately
this is absolutely not the case. Indeed, in $D(0,\delta)$ consider
the two points $p=\left(0,-\alpha\delta^{1/4},0\right)$ and $q=\left(0,0,\alpha\delta^{1/4}\right)$
(for $\alpha$ small enough, these points are in $D(0,\delta)$ for
all $\delta$, $0<\delta\leq\delta_{0}$) and estimate a constant
$K$ so that $q\in D(p,K\delta)$. In the coordinate system centered
at $p$, we have $q=\left(0,\alpha\delta^{1/6},\alpha\delta^{1/6}\right)=\sqrt{2}\alpha\delta^{1/6}\left(0,1/\sqrt{2},1/\sqrt{2}\right)$;
then calculating $c_{L\bar{L}}$ for $L=\frac{1}{\sqrt{2}}L_{1}^{0}+\frac{1}{\sqrt{2}}L_{2}^{0}$
we see that 
\[
F(L,p,K\delta)\gtrsim\frac{\alpha^{2}\delta^{-2/3}}{K}\mbox{ i.e. }F(L,p,K\delta)^{-1/2}\lesssim\sqrt{K}\delta^{1/3}.
\]

Then $q$ belongs to $D(p,K\delta)$ implies $K\gtrsim\delta^{-1/3}$.

\section{Appendix\label{sec:Appendix}}

The following Lemma is an improvement of Lemma 3.9 of \cite{Charpentier-Dupain-Geometery-Finite-Type-Loc-Diag}:
\begin{sllem}
\label{lem:deriv-positives-functions}Let $B_{j}$ be the unit ball
in $\mathbb{C}^{j}$. Let $K_{1}$ be a positive real number, $M$
and $n$ two positive integers. There exists a constant $C(K_{1})$
depending on $K_{1}$, $M$ and $n$ such that, for $j=1,\ldots,n-1$,
if $g$ is a non negative function of class $\MR C^{M}$on $B_{j}$
satisfying $\sup_{B_{j}}\{\left|D^{\alpha\beta}g(w)\right|,\,\left|\alpha+\beta\right|\leq M\}\leq K_{1}$,
where $D^{\alpha\beta}=\frac{\partial^{\left|\alpha+\beta\right|}}{\partial w^{\alpha}\partial\bar{w}^{\beta}}$,
then, for all $(\alpha^{0},\beta^{0})\in\left(\mathbb{N}^{j}\right)^{2}$,
$\left|\alpha^{0}+\beta^{0}\right|<M$, there exists $a\in\mathbb{N}^{j}$,
$2\left|a\right|\leq\left|\alpha^{0}+\beta^{0}\right|$ such that, denoting $\Delta^a$ the
differential operator $\prod_{i=1}^{j}\Delta_{i}^{a_{i}}$,
where $\Delta_{i}=\frac{\partial^{2}}{\partial z_{i}\partial\bar{z}_{i}}$
is the Laplacian in the $z_{i}$ coordinate,
\[
\Delta^{a}g(0)\geq\frac{1}{C(K_{1})}\left|D^{\alpha^{0}\beta^{0}}g(0)\right|^{2^{\left|\alpha^{0}+\beta^{0}\right|}}.
\]
 \end{sllem}
Note that there is no absolute value in the left hand side of the
inequality.
\begin{proof}
We only indicate how the proof of Lemma 3.9 of \cite{Charpentier-Dupain-Geometery-Finite-Type-Loc-Diag}
has to be modified.

Without loss of generality, we can suppose $\left|D^{\alpha_{0}\beta_{0}}g(0)\right|=\max_{\left|\alpha+\beta\right|=\left|\alpha^{0}+\beta^{0}\right|}\left|D^{\alpha\beta}g(0)\right|$.
By induction, it is enough to prove that there exist two constants
$c$ and $C$, depending on $M$ and $n$, such that one of the following
two cases holds:\smallskip{}

\begin{lyxlist}{00.00.000000}
\item [{\emph{First~case}}] there exists $a\in\mathbb{N}^{j}$, $2\left|a\right|=\left|\alpha^{0}+\beta^{0}\right|$
such that $\Delta^{a}g(0)\geq c\left|D^{\alpha^{0}\beta^{0}}g(0)\right|$;
\item [{\emph{Second~case}}] there exists $(\tilde{\alpha},\tilde{\beta})\in\left(\mathbb{N}^{j}\right)^{2}$,
$\left|\tilde{\alpha}+\tilde{\beta}\right|<\left|\alpha^{0}+\beta^{0}\right|$
such that
\[
\left|D^{\tilde{\alpha}\tilde{\beta}}g(0)\right|\geq\frac{1}{C}\left|D^{\alpha^{0}\beta^{0}}g(0)\right|^{-\left|\tilde{\alpha}+\tilde{\beta}\right|+\left|\alpha^{0}+\beta^{0}\right|+1}.
\]
\end{lyxlist}
\smallskip{}

Let $p=\left|\alpha^{0}+\beta^{0}\right|$, $\xi=\mu\varepsilon$,
$\mu\in]0,1[$, $\varepsilon=(\varepsilon_{i})$, $\left|\varepsilon_{i}\right|\leq1$,
and, as in the proof of Lemma 3.9 of \cite{Charpentier-Dupain-Geometery-Finite-Type-Loc-Diag}
let us write Taylor formula:
\begin{eqnarray*}
g(\xi) & = & \sum_{k=0}^{p-1}\mu^{k}\sum_{\left|\alpha+\beta\right|=k}*D^{\alpha\beta}g(0)\varepsilon^{\alpha}\bar{\varepsilon}^{\beta}+\mu^{p}\sum_{\left|\alpha+\beta\right|=p}*D^{\alpha\nu}g(0)\varepsilon^{\alpha}\bar{\varepsilon}^{\beta}+\mu^{p+1}R(\varepsilon,\mu)\\
 & = & A_{1}(\xi)+\mu^{p}A_{2}(\xi)+\mu^{p+1}R(\varepsilon,\mu),
\end{eqnarray*}
 where $*$ are multinomial coefficients and $\left|R\right|\leq K_{1}K_{2}$,
$K_{2}$ depending only on $M$ and $n$.

Remark now that, $g$ being non negative,

\begin{minipage}[c]{8mm}%
({*})$\left\{ \begin{array}{c}
\begin{array}{c}
\\
\\
\end{array}\end{array}\right.$%
\end{minipage}%
\begin{minipage}[c]{15cm}%
If there exists $\mu\simeq\left|D^{\alpha^{0}\beta^{0}}g(0)\right|$
such that $A_{2}(\xi)+\mu R(\varepsilon,\mu)<-c_{1}\left|D^{\alpha^{0}\beta^{0}}g(0)\right|$,
$c_{1}>0$, then the \emph{Second case} hold.\vspace{0.3\baselineskip}%
\end{minipage}

\smallskip{}

In the proof of Lemma 3.9 of \cite{Charpentier-Dupain-Geometery-Finite-Type-Loc-Diag}
we introduced a multi-index $c$ ($\left|c\right|=p$), depending
on $g$, and complex numbers $\varepsilon_{i}$ ($\forall i$, $\left|\varepsilon_{i}\right|\geq c(M,n)$),
depending on $g$ and $K(M,n)$, such that
\begin{equation}
\sum_{\SU{\left|\alpha+\beta\right|=p\AS\alpha+\beta\neq c}}\left|*D^{\alpha\beta}g(0)\varepsilon^{\alpha}\bar{\varepsilon}^{\beta}\right|\leq\frac{\left|D^{\alpha^{0}\beta^{0}}g(0)\right|}{K}\label{eq:multiindex-c-I}
\end{equation}
 and
\begin{equation}
\left|\sum_{\alpha+\beta=c}*D^{\alpha\beta}g(0)\varepsilon^{\alpha}\bar{\varepsilon}^{\beta}\right|\geq4\frac{\left|D^{\alpha^{0}\beta^{0}}g(0)\right|}{K}.\label{eq:multiindex-c-II}
\end{equation}

To finish the proof, we show now that, either we can find $\varepsilon$
and $\mu$ satisfying the hypothesis of ({*}), or we are in the \emph{First
case}.

We take $\mu=\frac{\left|D^{\alpha^{0}\beta^{0}}g(0)\right|}{KK_{1}K_{2}}$.
Then $\left|A_{2}(\xi)+\mu R(\xi)\right|\geq\frac{\left|D^{\alpha^{0}\beta^{0}}g(0)\right|}{K}$
and $A_{2}(\xi)+\mu R(\xi)$ has the sign of $\sum_{\alpha+\beta=c}*D^{\alpha\beta}g(0)\varepsilon^{\alpha}\bar{\varepsilon}^{\beta}$.

If $\sum_{\alpha+\beta=c}*D^{\alpha\beta}g(0)\varepsilon^{\alpha}\bar{\varepsilon}^{\beta}<0$,
then ({*}) is satisfied, thus we consider the case where $\sum_{\alpha+\beta=c}*D^{\alpha\beta}g(0)\varepsilon^{\alpha}\bar{\varepsilon}^{\beta}>0$.

If there is an index $i$ such that $c_{i}$ is odd, taking $\varepsilon'$
defined by $\varepsilon'_{j}=\varepsilon_{j}$ if $j\neq i$ and $\varepsilon'_{i}=-\varepsilon_{i}$,
then
\[
\sum_{\alpha+\beta=c}*D^{\alpha\beta}g(0){\varepsilon'}^{\alpha}\overline{\varepsilon'}^{\beta}\leq-\frac{4}{K}\left|D^{\alpha^{0}\beta^{0}}g(0)\right|,
\]
and, by (\ref{eq:multiindex-c-I}), ({*}) is verified.

So we suppose that for all $i$, $c_{i}=2c'_{i}$, and we write
\[
\sum_{\alpha+\beta=c}*D^{\alpha\beta}g(0){\varepsilon'}^{\alpha}\overline{\varepsilon'}^{\beta}=\sum_{k=0}^{c_{1}}{\varepsilon'}_{1}^{k}\overline{\varepsilon'}_{1}^{c_{1}-k}A_{k}^{1}(\varepsilon'_{2},\ldots,\varepsilon'_{n}),
\]
 with $\left|\varepsilon'_{i}\right|=\left|\varepsilon_{i}\right|$,
and choose $c\ll4/K$. We separate two cases.

First suppose that $A_{c'_{1}}^{1}(\varepsilon_{2},\ldots,\varepsilon_{n})\leq c\left|D^{\alpha^{0}\beta^{0}}g(0)\right|$.
If $c_{1}=0$ then (\ref{eq:multiindex-c-II}) implies
\[
\sum_{\alpha+\beta=c}*D^{\alpha\beta}g(0)\varepsilon^{\alpha}\overline{\varepsilon}^{\beta}\leq-c'\left|D^{\alpha^{0}\beta^{0}}g(0)\right|
\]
 which gives ({*}). Thus suppose $c_{1}\neq0$. Let
\[
\MR E_{0}=\{\varepsilon',\mbox{ such that }\varepsilon'_{i}=\varepsilon_{i},\, i>1,\,\varepsilon'_{1}=\vartheta\varepsilon_{1},\mbox{ with }\vartheta^{c_{1}}=1\}.
\]
Thus
\[
\sum_{\varepsilon'\in\MR E_{0}}\sum_{\alpha+\beta=c}*D^{\alpha\beta}g(0){\varepsilon'}^{\alpha}\overline{\varepsilon'}^{\beta}=c_{1}A_{c'_{1}}^{1}\left|\varepsilon_{1}\right|^{c_{1}}.
\]
Then, by (\ref{eq:multiindex-c-II}), there exists $\varepsilon'\in\MR E_{0}$
such that
\[
\sum_{\alpha+\beta=c}*D^{\alpha\beta}g(0){\varepsilon'}^{\alpha}\overline{\varepsilon'}^{\beta}\leq-c''\left|D^{\alpha^{0}\beta^{0}}g(0)\right|,
\]
(recall $\left|\varepsilon_{1}\right|>c(M,n)$) and ({*}) is verified
as before.

Suppose now $A_{c'_{1}}^{1}(\varepsilon_{2},\ldots,\varepsilon_{n})>c'\left|D^{\alpha^{0}\beta^{0}}g(0)\right|$.
Write
\[
A_{c'_{1}}^{1}=\sum_{k=0}^{c_{2}}\varepsilon_{2}^{k}\overline{\varepsilon_{2}}^{c_{2}-k}A_{k}^{2}(\varepsilon_{3},\ldots\varepsilon_{n}).
\]
 As before, if $c_{2}=0$ or if $A_{c'_{2}}^{2}(\varepsilon_{3},\ldots\varepsilon_{n})\leq c'''\left|D^{\alpha^{0}\beta^{0}}g(0)\right|$
then we can change $\varepsilon_{2}$ such that we obtain
\[
A_{c'_{1}}^{1}(\varepsilon'_{2},\ldots\varepsilon'_{n})\leq-c''''\left|D^{\alpha^{0}\beta^{0}}g(0)\right|,
\]
 and we conclude that ({*}) is satisfied. If $A_{c'2}^{2}(\varepsilon_{3},\ldots\varepsilon_{n})\geq c'''\left|D^{\alpha^{0}\beta^{0}}g(0)\right|$,
we do another time the same thing, on the third variable. Then, by
induction, if the process does not stop, the last step shows that
if ({*}) is not satisfied, then the inequality on $D^{c'c'}g(0)$
implies that we are in the \emph{First case}.
\end{proof}

\bibliographystyle{amsalpha}
\providecommand{\bysame}{\leavevmode\hbox to3em{\hrulefill}\thinspace}
\providecommand{\MR}{\relax\ifhmode\unskip\space\fi MR }
\providecommand{\MRhref}[2]{%
  \href{http://www.ams.org/mathscinet-getitem?mr=#1}{#2}
}
\providecommand{\href}[2]{#2}

\end{document}